\documentclass{scrartcl}

\usepackage[T1]{fontenc}
\usepackage[utf8]{inputenc}
\usepackage{amsmath, amsthm, amssymb}
\usepackage{dsfont}
\usepackage{graphicx}
\usepackage[format=plain]{caption}%ben?tigt, um mit dem Befehl \caption*{} eine Bildunterschrift zu setzen, die nicht mit Abbildung X eingeleitet wird. 'normal'=\setcapindent{0}
\usepackage{capt-of}%ben?tigt, um 'figure' und 'table' nebeneinander zu setzen, siehe 'http://www.faqs.org/faqs/de-tex-faq/part6/' Punkt 6.1.7, aufgerufen am 01.08.11.
\usepackage{subcaption}
\usepackage{bm}
\usepackage{tikz}
\usetikzlibrary{arrows,shapes,trees,calc,through}

\newcommand{\norm}[1]{\left\| #1 \right\|}  %Norm
\newcommand{\scprd}[1]{\left\langle #1 \right\rangle}  %Skalarprodukt
\renewcommand{\d}{\,\mathrm{d}} %Differenzial in Integralen
\newcommand{\e}{\mathrm{e}} %eulersche Zahl
\newcommand{\N}{\mathbb{N}}  %Zahlbereiche

\newcommand{\R}{\mathbb{R}}

\newcommand{\eps}{\varepsilon}
\renewcommand{\phi}{\varphi}
\newcommand{\ul}{\underline}
\newcommand{\ol}{\overline}

\numberwithin{equation}{section}
\newtheorem{thm}{Theorem}[section]
\newtheorem{cor}[thm]{Corollary}

\newtheorem{lm}[thm]{Lemma}
\newtheorem{cond}[thm]{Condition}

\theoremstyle{definition} \newtheorem{ex}[thm]{Example}
\theoremstyle{definition}

\title{Stabilization of port-Hamiltonian systems by nonlinear boundary control in the presence of disturbances}

\author{Jochen Schmid$^1$ and Hans Zwart$^{2,3}$\\  
\small $^1$ Institut f\"ur Mathematik, Universit\"at W\"urzburg, 97074 W\"urzburg, Germany\\
\small $^2$ Department of Applied Mathematics, University of Twente, 7500 AE Enschede %The Netherlands, 
\\
\small $^3$ Department of Mechanical  Engineering, Technische Universiteit Eindhoven, \\
\small 5600 MB Eindhoven, The Netherlands \\
\small jochen.schmid@mathematik.uni-wuerzburg.de, h.j.zwart@utwente.nl}

\date{}

\begin{document}

\maketitle

\begin{abstract}
\small{ \noindent 
In this paper, we are concerned with the stabilization of linear port-Hamiltonian systems of arbitrary order $N \in \N$ on a bounded $1$-dimensional spatial domain $(a,b)$. %bounded interval $(a,b)$
In order to achieve stabilization, we couple the system %by standard feedback interconnection
to a dynamic boundary controller, that is, a controller that acts on the system  only via the boundary points $a,b$ of the spatial domain. We use a  nonlinear controller in order to capture the nonlinear behavior that realistic actuators often exhibit %due to nonlinear potential energy or damping terms, for instance 
and, moreover, we allow the output of the controller to be corrupted by actuator disturbances %$d$ 
before it is fed back into the system. 
What we show here is that the resulting nonlinear closed-loop system %coupled 
%We show that the resulting nonlinear closed-loop system %with input $d$
is input-to-state stable w.r.t.~square-integrable disturbance inputs. %A bit more precisely, 
In particular, we obtain uniform input-to-state stability for systems of order $N=1$ and a special class of nonlinear controllers, and weak input-to-state stability for systems of arbitrary order $N \in \N$ and a more %fairly 
general class of nonlinear controllers. 
Also, in both cases, we obtain convergence to $0$ of all solutions as $t \to \infty$.
Applications are given to vibrating strings and beams.  
}
\end{abstract}

{ \small \noindent 
%\emph{AMS Subject Classification (2010):} 
%\\
Index terms:  Input-to-state stability, infinite-dimensional systems, port-Hamiltonian systems, nonlinear boundary control, actuator disturbances
}

\section{Introduction}

%\subsection{Setting and goal}

%1.1 Was betrachten wir in dieser Arbeit fuer Systeme (fuer eine Systemklasse)? Welche Anwendungsbeispiele fallen in diese Systemklasse?
In this paper, we consider linear port-Hamiltonian systems of arbitrary order $N \in \N$ on a bounded $1$-dimensional spatial domain $(a,b)$. Such systems are described by a linear partial differential equation of the form
\begin{align}
\partial_t x(t,\zeta) = P_N \partial_{\zeta}^N( \mathcal{H}(\zeta) x(t,\zeta) ) + \dotsb + P_1 \partial_{\zeta}( \mathcal{H}(\zeta) x(t,\zeta) ) + P_0 \mathcal{H}(\zeta) x(t,\zeta) 
%\qquad (t \in [0,\infty), \zeta \in (a,b))
\end{align}
for $t \in [0,\infty), \zeta \in (a,b)$, 
and the energy of such a system in the state $x(t,\cdot)$ is given by
\begin{align}
E(x(t,\cdot)) = \frac{1}{2} \int_a^b x(t,\zeta)^{\top} \mathcal{H}(\zeta) x(t,\zeta) \d \zeta,
\end{align}
where $x(t,\zeta) \in \R^m$, %is the state (vector evaluated) at $\zeta$
$\mathcal{H}(\zeta) \in \R^{m\times m}$ is the energy density at $\zeta$, and $P_0, \dots, P_N \in \R^{m \times m}$ are alternately skew-symmetric and symmetric matrices and $P_N$ is invertible. 
Simple examples of such systems are given by vibrating strings or beams. Also, many of the systems of linear conservation laws considered in~\cite{BaCo}, namely those of the special form
\begin{align} \label{eq:system of lin conservation laws}
\partial_t x(t,\zeta) = P \partial_{\zeta} x(t,\zeta) 
\qquad (t \in [0,\infty), \zeta \in (a,b))
\end{align}
with a $\zeta$-independent invertible matrix $P$, fall in the above class of linear port-Hamiltonian systems. 
%
%1.2 Was wollen wir in dieser Arbeit grob erreichen? Wodurch und womit wollen wir das erreichen?
What we are interested in here is the stabilization of such a system $\mathfrak{S}$ by means of dynamic boundary control, that is, by coupling the system to a dynamic controller $\mathfrak{S}_c$ that acts on the system  only via the boundary points $a,b$ of the spatial domain $(a,b)$. 
Since realistic controllers often exhibit nonlinear behavior (due to nonquadratic potential energy or nonlinear damping terms, for instance), we want to work with nonlinear controllers -- just like~\cite{Au15}, \cite{MiSt16}, \cite{ZwRaLe16}, \cite{RaZwLe17}. 
Since, moreover, realistic controllers are typically affected by external disturbances, we -- unlike ~\cite{Au15}, \cite{MiSt16}, \cite{ZwRaLe16}, \cite{RaZwLe17} -- %the above-mentioned works -- 
also want to incorporate such actuator disturbances which corrupt the output of the controller before it is fed back into the system.  
%
%1.3 Wie sieht das gekoppelte System (bildlich) aus?
Coupling such a controller to the system $\mathfrak{S}$ by standard feedback interconnection
\begin{align}
y(t) = u_c(t) \qquad \text{and} \qquad -y_c(t)+d(t) = u(t)
\end{align}
(with $u,y$ and $u_c,y_c$ being the in- and outputs of $\mathfrak{S}$ and $\mathfrak{S}_c$ respectively and with $d$ being the disturbance), we obtain a nonlinear %semilinear
closed-loop system $\tilde{\mathfrak{S}}$ with input $d$ and output $y$. Schematically, this closed-loop system can be depicted as in the following figure.

\begin{figure}[h!]
\centering
\begin{tikzpicture}[scale =1]
%\draw[step=1.0,black,thin] (0,0) grid (8,8);
\tikzstyle{tf} = [rectangle, draw = black, minimum width= 1.618cm, minimum height = 1cm]
\tikzstyle{sum} = [draw=black, shape=circle, minimum width = .5, minimum height = .5]

\node(BC) [tf,scale=1, minimum width = 1cm] at (4,4){System $\mathfrak{S}$};
\node(NL) [tf,scale=1, minimum width = 1cm] at (4,2){Controller $\mathfrak{S}_c$};
\node(s1) [sum, scale = .7, minimum width = .5] at (2,4) {$+$};
\node(m) [ right] at (2,3.5){$-$};

\node(d) [above] at (1.2,4) {$d$};
%\node(e) [signal, above] at (8,-3) {$e_{i}$};
\node(u) [signal, above] at (2.6,4) {$u$};
\node(y) [above left] at (6,4){} ;
\node(o)[signal, above] at (5.6,4) {$y$};
\node(uc) [above left] at (6,2) {$u_c$};
\node(yc) [above right] at (2,2) {$y_c$};

\draw [->] (0.6,4)--(s1);
\draw [->] (s1)--(BC);
\draw [->] (BC)--(7,4);
\draw [->] (y.south east) -- (uc.south east) -- (NL);
\draw [->] (NL.west) -| (s1.south); %-| it makes a corner
\end{tikzpicture}
\caption{Closed-loop system $\tilde{\mathfrak{S}}$ \label{fig:CL}}
%\caption{Structure of the closed-loop system $\tilde{\mathfrak{S}}$ \label{fig:CL}}
\end{figure}
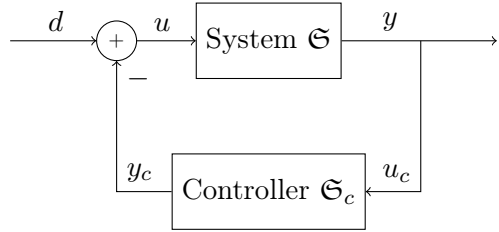

%\subsection{Contributions}

%2.1 Was machen wir (tragen wir bei) in *einem* Satz?
What we do in this paper, is to establish input-to-state stability for the closed-loop system $\tilde{\mathfrak{S}}$ w.r.t.~square-integrable disturbance inputs $d$. 
%
%2.2 Was bedeutet ISS bzgl. L^2-inputs ueberhaupt?
In rough terms, this means that $\tilde{x}_* := 0$ is an asymptotically stable equilibrium point of the undisturbed closed-loop system with $d = 0$ and that this stability property %of $0$ 
is robust w.r.t.~disturbances $d \ne 0$ in the sense that small disturbances affect the asymptotic stability %property 
only slightly (see~\eqref{eq:ISS-informal def, stab estimate} and~\eqref{eq:ISS-informal def, attractivity  estimate} below). 
In more precise terms, the input-to-state stability of %the closed-loop system 
$\tilde{\mathfrak{S}}$ w.r.t.~square-integrable disturbance inputs $d$ means 
(i) that $\tilde{\mathfrak{S}}$, for every initial state $\tilde{x}_0$ and every square-integrable disturbance $d$, has a unique global (generalized) solution $\tilde{x}(\cdot,\tilde{x}_0,d)$ and 
(ii) that for all these (generalized) solutions, the following perturbed stability and attractivity estimates hold true: 
\begin{gather}
\norm{ \tilde{x}(t,\tilde{x}_0,d) } \le \ul{\sigma}(\norm{\tilde{x}_0}) + \ul{\gamma}(\norm{d}_2) 
\qquad (t \in [0,\infty)) \label{eq:ISS-informal def, stab estimate}
\\
\text{and} \notag \\
\limsup_{t\to \infty} \norm{ \tilde{x}(t,\tilde{x}_0,d) }  \le \ol{\gamma}(\norm{d}_2),  \label{eq:ISS-informal def, attractivity  estimate}
\end{gather}
where $\ul{\sigma}, \ul{\gamma}, \ol{\gamma}$ are monotonically increasing comparison functions that are zero at $0$. %and that do not depend on $\tilde{x}_0$ or $d$
According to whether the limit relation~\eqref{eq:ISS-informal def, attractivity  estimate} holds (locally) uniformly or just pointwise  w.r.t.~$\tilde{x}_0$ and $d$,
%holds (locally) uniformly w.r.t.~$\tilde{x}_0$ and $d$ or just pointwise w.r.t.~$\tilde{x}_0$ and $d$, 
one speaks of uniform or weak input-to-state stability, respectively. Also, instead of uniform input-to-state stability one usually just speaks of input-to-state stability. 
%
%2.3 Was (hauptsaechlich) zeigen wir in dieser Arbeit genauer? Was tragen wir an Saetzen und Anwendungen bei?
%3.2.1 Die Saetze
In this paper, we show (i) that for a system $\mathfrak{S}$ of order $N = 1$ and a special class of nonlinear controllers $\mathfrak{S}_c$, % of the above type, 
the resulting closed-loop system $\tilde{\mathfrak{S}}$ is uniformly input-to-state stable w.r.t.~square-integrable disturbances $d$, 
and (ii) that for a system $\mathfrak{S}$ of arbitrary order $N \in \N$ and a more general class of nonlinear controllers $\mathfrak{S}_c$, %of the above type, 
the resulting closed-loop system $\tilde{\mathfrak{S}}$ is weakly input-to-state stable w.r.t.~square-integrable disturbances $d$. In particular, we show in both cases that unique global (generalized) solutions $\tilde{x}(\cdot,\tilde{x}_0,d)$ exist for $\tilde{\mathfrak{S}}$. Additionally, we will see that in both cases %both in the uniform and in the weak case
every such solution converges to zero: %$0$ as $t \to \infty$:
\begin{align}
\tilde{x}(t,\tilde{x}_0,d) \longrightarrow 0 \qquad (t \to \infty)
\end{align}
for every initial state $\tilde{x}_0$ and every square-integrable disturbance $d$. 
%
%2.3.2 Die Voraussetzungen an das zugrundeliegende System
In all these results, we have to impose only mild assumptions on the system $\mathfrak{S}$ to be stabilized, namely an impedance-passivity condition and an approximate observability condition. %its impedance-passivity and 
%
%2.3.3 Anwendungen
We finally apply both our uniform and our weak input-to-state stability result to vibrating strings and beams (modeled according to Timoshenko). 
\smallskip

%\subsection{Comparison with the existing literature}

In the literature, the stabilization of port-Hamiltonian systems has been considered so far, to the best of our knowledge, only in the case without actuator disturbances. 
%
%3.1 Arbeiten, in denen gar keine Stoersignale (und damit insbesondere keine Aktuatorstoerungsignale) auftreten
In~\cite{ViZwLeMa09}, \cite{Au15}, \cite{MiSt16}, \cite{ZwRaLe16}, \cite{RaZwLe17}, no disturbances are considered at all, that is, the situation depicted above is considered in the special case where $d =0$. Stabilization of port-Hamiltonian systems of various degrees of generality is achieved by means of linear dynamic boundary controllers in~\cite{ViZwLeMa09} and by means of nonlinear dynamic boundary controllers in \cite{Au15}, \cite{MiSt16}, \cite{ZwRaLe16}, \cite{RaZwLe17}.  
%
%3.2 Arbeiten, in denen Stoersignale auftreten, aber in der Form von Sensorstoersignalen 
In~\cite{TaPrTa17}, %disturbances do occur but in the form of sensor disturbances, that is,
sensor -- instead of actuator -- disturbances are considered, that is,  disturbances do occur in~\cite{TaPrTa17} but they corrupt the input of the controller instead of its output. %In other words/In pictures, this paper considers the situation depicted in the figure below. 
It is shown in~\cite{TaPrTa17} that for a port-Hamiltonian system of the special form~\eqref{eq:system of lin conservation laws} with negative definite $P$ and a linear dynamic boundary controller, the resulting closed-loop system is uniformly input-to-state stable w.r.t.~essentially bounded disturbance inputs $d$ (meaning that the perturbed stability and attractivity estimates~\eqref{eq:ISS-informal def, stab estimate} and~\eqref{eq:ISS-informal def, attractivity  estimate} are satisfied with the $2$-norm $\norm{d}_2$ replaced by the $\infty$-norm $\norm{d}_{\infty}$).  
\smallskip

%\subsection{Structure and notation}

%We organize the paper as follows. In Section~\ref{sect:setting} we desribe in detail the setting with the precise assumptions on the system $\mathfrak{S}$ and the controller $\mathfrak{S}_c$. In Section~... In order to get a quick overview of the core results in simplified form, the reader can consult~\cite{ScZw18-MTNS} 
%We conclude the introduction with some remarks on notational conventions. %on the notation used throughout the paper. ...

We conclude the introduction with some remarks on the organization of the paper and on notational conventions used throughout the paper. 
%
%4.1 Struktur der Arbeit
Section~\ref{sect:setting} provides a detailed description of the setting with the precise assumptions on the system $\mathfrak{S}$ to be stabilized and the controller $\mathfrak{S}_c$ used for that purpose. In Section~\ref{sect:solvb} we prove the solvability of the closed-loop system -- first, in the classical sense for classical data $(\tilde{x}_0,d)$, and then in the generalized sense for general data $(\tilde{x}_0,d)$. In Section~\ref{sect:stab} we establish the main results of this paper, namely the uniform (Section~\ref{sect:ISS}) and the weak (Section~\ref{sect:wISS}) input-to-state stability of the closed-loop system. And finally, in Section~\ref{sect:applications} we present some applications of our general results. In order to get a quick overview of the core results in simplified form, the reader can  consult the conference paper~\cite{ScZw18-MTNS}.
\smallskip

%
%4.2 Vereinbarungen zur Notation
In the entire paper, $|\cdot|$ denotes the standard norm on $\R^k$ for every $k \in \N$. As usual, $\mathcal{K}$, $\mathcal{K}_{\infty}$, $\mathcal{L}$ denote the following classes of comparison functions:
\begin{gather*}
\mathcal{K} := \{ \gamma \in C([0,\infty),[0,\infty)): \gamma \text{ strictly increasing with } \gamma(0) = 0 \}, \\
\mathcal{K}_{\infty} := \{ \gamma \in \mathcal{K}: \gamma \text{ unbounded} \}, \\
\mathcal{L} := \{ \gamma \in C([0,\infty),[0,\infty)): \gamma \text{ strictly decreasing with } \lim_{r\to\infty} \gamma(r) = 0 \}.
\end{gather*}
Also, $C_c^2([0,\infty),\R^k)$ denotes the subspace of $C^2([0,\infty),\R^k)$-functions $d$ with compact support in $[0,\infty)$, %interpretation: classical disturbance signals
and for $d \in L^2([0,\infty),\R^k)$ we will use the following short-hand notations: %for norms: %write for the sake of brevity
\begin{align*}
\norm{d}_2 := \norm{d}_{L^2([0,\infty),\R^k)},
\qquad
%\quad \text{and} \quad
\norm{d}_{[0,t],2} := \norm{d|_{[0,t]}}_{L^2([0,t],\R^k)}.
\end{align*}  
And finally, for a semigroup generator $A$ and bounded operators $B, C$ between appropriate spaces, the symbol $\mathfrak{S}(A,B,C)$ will stand for the state-linear system~\cite{CurtainZwart} 
\begin{align*}
x' = Ax + Bu \quad \text{with} \quad y = Cx,
\end{align*}
where the prime stands for derivative w.r.t.~time.

\section{Setting} \label{sect:setting}

\subsection{Setting: the system to be stabilized}

%1. A rough (algebraic) description of port-Hamiltonian systems with boundary control and boundary observation
%
We consider a linear port-Hamiltonian system $\mathfrak{S}$ of order $N \in \N$ on a bounded  interval $(a,b)$ with control and observation at the boundary~\cite{JacobZwart},  \cite{DuMaStBr}, \cite{Augner}. %points $a,b$. 
%As has been pointed out above, the system $\mathfrak{S}$ to be stabilized is a linear port-Hamiltonian system~\cite{JacobZwart},  \cite{DuMaStBr}, \cite{Augner} of order $N \in \N$ on a bounded  interval $(a,b)$ with control and observation at the boundary. %points $a,b$. 
Such a system evolves according to the following differential equation with boundary control and boundary observation conditions: 
\begin{gather}
x' = \mathcal{A}x %= \mathcal{J} \mathcal{H}x 
=  P_N \partial_{\zeta}^N(\mathcal{H}x) + P_{N-1} \partial_{\zeta}^{N-1}(\mathcal{H}x) + \dotsb + P_1 \partial_{\zeta}(\mathcal{H}x) + P_0 \mathcal{H}x  \label{eq:S diff eq}\\
u(t) = \mathcal{B}x(t) \quad \text{and} \quad y(t) = \mathcal{C}x(t) \label{eq:S input/output eq}
\end{gather}
%State space: $X := L^2((a,b),\R^m)$ and $D(\mathcal{A}) := \{x \in X: \mathcal{H}x \in W^{1,2}((a,b),\R^m)\}$, 
and the energy of such a system in the state $x$ is given by \begin{align}
E(x) = \frac{1}{2} \int_a^b x(\zeta)^{\top} (\mathcal{H}x)(\zeta)\d \zeta. 
%=: \frac{1}{2} \norm{x}_X^2,
\end{align}
In these equations, $\zeta \mapsto \mathcal{H}(\zeta) \in \R^{m\times m}$ is a measurable matrix-valued function (the %so-called 
energy density) such that for some positive constants $\ul{m}, \ol{m}$ and almost all $\zeta \in (a,b)$
\begin{align} \label{eq:H bd below and above}
0 < \ul{m}  \le \mathcal{H}(\zeta) \le \ol{m}  < \infty,
\end{align}
and $P_0,P_1, \dots, P_N \in \R^{m\times m}$ are matrices such that $P_N$ is invertible and $P_1,\dots, P_N$ are alternately symmetric and skew-symmetric while $P_0$ is dissipative: %Die Sprechweise "dissipativ" ist uebernommen aus Corollary 3.2.14 (Augner)
\begin{align*}
P_l^{\top} = (-1)^{l+1} P_l \qquad (l \in \{1,\dots,N\}) 
\qquad \text{and} \qquad P_0^{\top} + P_0 \le 0.
\end{align*} 
(Strictly speaking, such systems should be called port-Hamilitonian only if $P_0$ is skew-symmetric. If $P_0$ is only dissipative, they could more precisely be called port-dissipative systems.) 
%(Strictly speaking, such systems should be called port-dissipative systems, and port-Hamilitonian only if $P_0$ is skew-symmetric.)
%
%
%2. A detailed (analytical) description: spaces, operators, domains
%
As the state space of $\mathfrak{S}$ one chooses $X := L^2((a,b),\R^m)$ with norm $\norm{\cdot}_X$ given by the system energy
\begin{align*}
\frac{1}{2}\norm{x}_X^2 := E(x) = \frac{1}{2} \int_a^b x(\zeta)^{\top} (\mathcal{H}x)(\zeta)\d \zeta.
\end{align*}
It is clear by~\eqref{eq:H bd below and above} that the norm $\norm{\cdot}_X$ is equivalent to the standard norm of $L^2((a,b),\R^m)$ and that it %$\norm{\cdot}_X$ 
is induced by a scalar product which we denote by $\scprd{\cdot,\cdot \cdot}_X$.
As the domain of the linear differential operator $\mathcal{A}: D(\mathcal{A}) \subset X \to X$ from~\eqref{eq:S diff eq} one chooses 
\begin{align}
D(\mathcal{A}) := \big\{ x \in X: \mathcal{H}x \in W^{N,2}((a,b),\R^m) \text{ and } W_{B,1} (\mathcal{H}x)|_{\partial} = 0 \big\}
\end{align}
where $W_{B,1} \in \R^{(mN-k) \times 2mN}$ with $k \in \{1,\dots,mN\}$ and where 
\begin{align*}
(\mathcal{H}x)|_{\partial} 
:= 
\begin{pmatrix}
(\mathcal{H}x)|_{\partial,b} \\ (\mathcal{H}x)|_{\partial,a}
\end{pmatrix},
\qquad 
(\mathcal{H}x)|_{\partial,\zeta} 
:=
\big( (\mathcal{H}x)(\zeta)^{\top}, \partial_{\zeta}(\mathcal{H}x)(\zeta)^{\top}, \dots, \partial_{\zeta}^{N-1} (\mathcal{H}x)(\zeta)^{\top} \big)^{\top}.
\end{align*}
%\begin{align*}
%(\mathcal{H}x)|_{\partial} := \big( (\mathcal{H}x)(b)^{\top}, \partial_{\zeta}(\mathcal{H}x)(b)^{\top}, \dots, \partial_{\zeta}^{N-1} (\mathcal{H}x)(b)^{\top},(\mathcal{H}x)(a)^{\top}, \partial_{\zeta}(\mathcal{H}x)(a)^{\top}, \dots, \partial_{\zeta}^{N-1} (\mathcal{H}x)(a)^{\top} \big)^{\top}
%\end{align*}
In other words, the domain of $\mathcal{A}$ incorporates %contains/comprises
the zero boundary condition
\begin{align} \label{eq:S bdry cond}
W_{B,1} (\mathcal{H}x)|_{\partial} = 0
\end{align}  
which consists of $Nm-k$ (scalar) equations and is linear in the boundary values of $(\mathcal{H}x), \partial_{\zeta} (\mathcal{H}x), \dots, \partial_{\zeta}^{N-1} (\mathcal{H}x)$. 
Similarly, the boundary control and boundary observation conditions~\eqref{eq:S input/output eq} consist of $k$ (scalar) equations each and are again linear in $(\mathcal{H}x(t))|_{\partial}$, that is, the boundary control and boundary observation operators $\mathcal{B},\mathcal{C}: D(\mathcal{A}) \to \R^k$ from~\eqref{eq:S input/output eq} are linear and of the form
\begin{align} \label{eq:S bdry control/observ operators}
\mathcal{B}x := W_{B,2} (\mathcal{H}x)|_{\partial}
\qquad \text{and} \qquad
\mathcal{C}x := W_C (\mathcal{H}x)|_{\partial}
\end{align}
%for $x \in D(\mathcal{A})$ %BEACHTE, DASS FUER SOLCHE x DIE AUSWERTUNGEN (\mathcal{H} x)(a), (\mathcal{H} x)(b) SINNVOLL SIND
with matrices $W_{B,2}, W_C \in \R^{k\times 2mN}$. %are called the boundary control and boundary observation matrix, respectively. 
In all our results, %In all that follows/In what follows
we will impose the following additional impedance-passivity condition on $\mathfrak{S}$.

\begin{cond} \label{cond:S imp-passive}
$\mathfrak{S}$ is impedance-passive, that is,
\begin{align}  \label{eq:S imp-passive}
\scprd{x,\mathcal{A}x}_X \le (\mathcal{B}x)^{\top} \mathcal{C}x \qquad (x \in D(\mathcal{A})).
\end{align}
\end{cond}

It follows from Condition~\ref{cond:S imp-passive} by virtue of~\cite{JaKa17} (Theorem~2.3) that 
\begin{align}
A := \mathcal{A}|_{D(\mathcal{A}) \cap \ker \mathcal{B}}
\end{align}
is a contraction semigroup generator on $X$ and that the matrix $W_B := (W_{B,1}^{\top}, W_{B,2}^{\top})^{\top} \in \R^{mN \times 2mN}$ has full row rank. 
In particular, $\mathcal{A}$ and $\mathcal{B}$ define a boundary control system~\cite{JacobZwart} %in the sense of~\cite{JacobZwart} 
(by the same arguments as for Theorem~11.3.2 of~\cite{JacobZwart} and by Lemma~A.3 of~\cite{LeZwMa05}), so that for every initial state $x_0 \in D(\mathcal{A})$ and every control input $u \in C^2([0,\infty),\R^k)$ with $u(0) = \mathcal{B}x_0$, the system~\eqref{eq:S diff eq} and~\eqref{eq:S input/output eq} has a unique global classical solution
\begin{align} \label{eq:class-sol-of-S}
x(\cdot,x_0,u) \in C^1([0,\infty),X).
\end{align} 
Also, %And 
along every such classical solution, the following energy dissipation inequality is satisfied by virtue of~\eqref{eq:S imp-passive}:
\begin{align}
E_x'(t) = \scprd{x(t),\mathcal{A}x(t)}_X \le (\mathcal{B}x(t))^{\top} \mathcal{C}x(t) = u(t)^{\top} y(t) 
\qquad (t \in [0,\infty)),
\end{align}
where $E_x(t) := E(x(t))$ and $x(t) := x(t,x_0,u)$ and $y(t) := y(t,x_0,u) := \mathcal{C}x(t,x_0,u)$. 
%In particular, $\mathcal{A}$ and $\mathcal{B}$ define a boundary control system~\cite{JacobZwart} %in the sense of~\cite{JacobZwart} 
%(by the same arguments as for Theorem~11.3.2 of~\cite{JacobZwart} and by Lemma~A.3 of~\cite{LeZwMa05}) 
%and along classical solutions of this boundary control system the following energy dissipation inequality is satisfied:

\subsection{Setting: the controller}

As the controller $\mathfrak{S}_c$ to stabilize $\mathfrak{S}$ we choose a finite-dimensional nonlinear system which evolves according to the ordinary differential equation (with input $u_c$ and output $y_c$)
\begin{gather}
v' = \begin{pmatrix} v_1' \\ v_2' \end{pmatrix}
= 
\begin{pmatrix}
K v_2 \\
-\nabla \mathcal{P}(v_1) - \mathcal{R}(Kv_2) + B_c u_c 
\end{pmatrix} 
\\
y_c = B_c^{\top} K v_2 + S_c u_c 
\end{gather}
%\begin{align}
%v_1' &= K v_2 \notag \\
%v_2' &= -\nabla \mathcal{P}(v_1) - \mathcal{R}(Kv_2) + B_c u_c \\
%y_c &= B_c^{\top} K v_2 + S_c u_c, \notag
%\end{align}
and whose energy in the state $v = (v_1,v_2) \in \R^{2m_c}$ %from the controller state space $V := \R^{2m_c}$ 
is given by 
\begin{align}
E_c(v) := \mathcal{P}(v_1) + \frac{1}{2} v_2^{\top}Kv_2
\end{align}
(potential energy plus kinetic energy). %It should be noted that without the damping term $-\mathcal{R}(Kv_2)$ and the input term $B_c u_c$ the above ode would be Hamiltonian with Hamiltonian function~$E_c$
In these equations, $K \in \R^{m_c \times m_c}$, $B_c \in \R^{m_c \times k}$, $S_c \in \R^{k \times k}$ represent a generalized mass matrix, an input matrix, and a direct feedthrough matrix respectively  and they are such that
\begin{align} 
K > 0 \qquad \text{and} \qquad S_c \ge \varsigma > 0  \qquad (\varsigma := \min \sigma(S_c)).
\label{eq:S_c pos def} 
%\varsigma := \min \sigma(S_c) \label{eq:varsigma def}
\end{align}
Additionally, %Also, 
the potential energy $\mathcal{P}: \R^{m_c} \to [0,\infty)$ is differentiable such that $\nabla \mathcal{P}$ is locally Lipschitz continuous and $\mathcal{P}(0) = 0$ and the damping function $\mathcal{R}: \R^{m_c} \to \R^{m_c}$ is locally Lipschitz continuous such that $\mathcal{R}(0) = 0$.
% 
%State space is $V := \R^{2 m_c}$ 
%and energy in state $v = (v_1,v_2) \in V$ is %$E_c(v) := 
%\begin{align*}
%E_c(v) := \mathcal{P}(v_1) + \frac{1}{2} v_2^{\top}Kv_2.
%\end{align*}
As the norm on the controller state space $V := \R^{2m_c}$ we choose $|\cdot|_V$ defined by
\begin{align*}
|v|_V^2 = |(v_1,v_2)|_V^2:= |v_1|^2 + v_2^{\top} K v_2
\end{align*}
which is obviously equivalent to the standard norm on $\R^{2m_c}$ and is induced by a scalar  product $\scprd{\cdot,\cdot \cdot}_V$.
In all our results, %In all that follows/In what follows
we will impose the following additional conditions on the controller $\mathfrak{S}_c$.

\begin{cond} \label{cond:P und R}
\begin{itemize}
\item[(i)]  $\mathcal{P}$ is positive definite and radially unbounded, %proper,
that is, $\mathcal{P}(v_1) > 0$ for all $v_1 \in \R^{m_c} \setminus \{0\}$ %$\R^{m_c} \ni v_1 \ne 0$ 
and $\mathcal{P}(v_1) \longrightarrow \infty$ as $|v_1| \to \infty$
\item[(ii)] $\mathcal{R}$ is damping, that is, $v_2^{\top} \mathcal{R}(v_2) \ge 0$ for all $v_2 \in \R^{m_c}$.
\end{itemize}
\end{cond}

It follows from Condition~\ref{cond:P und R}~(ii) and from~\eqref{eq:S_c pos def} that the controller system $\mathfrak{S}_c$ is passive (in fact, strictly input-passive) w.r.t.~the storage function $E_c$.

\subsection{Setting: the closed-loop system}

Coupling $\mathfrak{S}$ and $\mathfrak{S}_c$ by standard feedback interconnection
\begin{align*}
y(t) = u_c(t) \quad \text{and} \quad  -y_c(t)+d(t) = u(t),
\end{align*}
we obtain the closed-loop system $\tilde{\mathfrak{S}}$ described by the following evolution equation and boundary input and output  conditions:
\begin{gather}
\tilde{x}' = \tilde{\mathcal{A}} \tilde{x} + \tilde{f}(\tilde{x})
\label{eq:cl-diff-eq}
%= \begin{pmatrix}
%\mathcal{A}x \\ Kv_2 \\ -v_1 + B_c \mathcal{C}x
%\end{pmatrix} 
%+
%\begin{pmatrix}
%0 \\ 0 \\ v_1 - \nabla \mathcal{P}(v_1) - \mathcal{R}(Kv_2)
%\end{pmatrix}
\\
d(t) = \tilde{\mathcal{B}} \tilde{x}(t) \qquad \text{and} \qquad y(t) = \tilde{\mathcal{C}} \tilde{x}(t).
\end{gather}
Its state space is the Hilbert space $\tilde{X} := X \times V = X \times \R^{2m_c}$ with norm $\norm{\cdot} = \norm{\cdot}_{\tilde{X}}$ defined by 
$$\norm{\tilde{x}}^2 = \norm{(x,v)}^2 := \norm{x}_X^2 + |v|_V^2$$ and the energy of the closed-loop system in the state $\tilde{x} = (x,v) \in \tilde{X}$ is
\begin{align*}
\tilde{E}(\tilde{x}) := E(x) + E_c(v) = \frac{1}{2}\norm{x}_X^2 + \mathcal{P}(v_1) + \frac{1}{2} v_2^{\top}Kv_2.
\end{align*}
%We record here for later use %It is worth mentioning here 
It follows from Condition~\ref{cond:P und R}~(i) (taking into account  Lemma~2.5 of~\cite{ClLeSt98}) that the energy $\tilde{E}$ is equivalent to the norm $\norm{\cdot}$ of $\tilde{X}$ 
%$\norm{\cdot} = \norm{\cdot}_{\tilde{X}}$ with $\norm{\tilde{x}}^2 := \norm{x}_X^2 + |v|^2
in the following sense: there exist $\ul{\psi}, \ol{\psi} \in \mathcal{K}_{\infty}$ such that for all $\tilde{x} \in \tilde{X}$
\begin{align} \label{eq:E-tilde equiv to norm}
\ul{\psi}(\norm{\tilde{x}}) \le \tilde{E}(\tilde{x}) \le \ol{\psi}(\norm{\tilde{x}}).
\end{align}
In the equations above, the linear and nonlinear operators $\tilde{\mathcal{A}}: D(\tilde{\mathcal{A}}) \to \tilde{X}$ and $\tilde{f}: \tilde{X} \to \tilde{X}$ %maps/parts 
with $D(\tilde{\mathcal{A}}) := D(\mathcal{A}) \times \R^{2m_c}$ 
are given respectively by
\begin{gather*}
\tilde{\mathcal{A}} \tilde{x}
:=
\begin{pmatrix}
\mathcal{A}x \\ Kv_2 \\ -v_1 + B_c \mathcal{C}x
\end{pmatrix}
\qquad \text{and} \qquad
\tilde{f}(\tilde{x})
:=
\begin{pmatrix}
0 \\ 0 \\ v_1 - \nabla \mathcal{P}(v_1) - \mathcal{R}(Kv_2)
\end{pmatrix}
\end{gather*}
%\begin{align*}
%\tilde{\mathcal{A}} \tilde{x}
%&:=
%\begin{pmatrix}
%\mathcal{A}x \\ Kv_2 \\ -v_1 + B_c \mathcal{C}x
%\end{pmatrix}, \\
%&\qquad \qquad \qquad \tilde{f}(\tilde{x})
%:=
%\begin{pmatrix}
%0 \\ 0 \\ v_1 - \nabla \mathcal{P}(v_1) - \mathcal{R}(Kv_2)
%\end{pmatrix}
%\end{align*}
and the linear boundary input and output operators $\tilde{\mathcal{B}}, \tilde{\mathcal{C}}: D(\tilde{\mathcal{A}}) \to \R^k$ are given by 
\begin{align}
\tilde{\mathcal{B}} \tilde{x} := \mathcal{B} x + B_c^{\top}K v_2 + S_c \mathcal{C}x, \qquad %\text{and} \qquad
\tilde{\mathcal{C}} \tilde{x} := \mathcal{C} x.
\end{align}
%and $\tilde{\mathcal{B}}_1 \tilde{x} := \mathcal{B}_1 x + B_c^{\top}K v_2 + S_c \mathcal{C}x$, 
%for $\tilde{x} = (x,v_1,v_2) \in D(\tilde{\mathcal{A}}) := D(\mathcal{A}) \times \R^{2m_c}$. 

\begin{lm} \label{lm:sgr gen with comp resolv}
If Condition~\ref{cond:S imp-passive} and~\ref{cond:P und R} are  satisfied, then 
\begin{itemize}
\item[(i)] the operator $\tilde{A} := \mathcal{A}|_{D(\tilde{\mathcal{A}}) \cap \ker \tilde{\mathcal{B}}}$ is a contraction semigroup generator on $\tilde{X}$ with compact resolvent
\item[(ii)] $\tilde{f}$ is Lipschitz continuous on bounded subsets of $\tilde{X}$. 
\end{itemize}
\end{lm}

\begin{proof}
Since assertion~(ii) is clear by our assumptions on $\mathcal{P}$ and $\mathcal{R}$, %made above, 
we have only to prove assertion~(i) and we do so in three steps.
As a first step, we observe that $\tilde{A}$ is dissipative in $\tilde{X}$. Indeed, for every $\tilde{x} = (x,v) \in D(\tilde{A})$ we have $x \in D(\mathcal{A})$ and $B_c^{\top}Kv_2 = -\mathcal{B}x-S_c \mathcal{C}x$ and hence
\begin{align*}
\langle \tilde{x}, \tilde{A} \tilde{x} \rangle_{\tilde{X}} &= \scprd{x,\mathcal{A}x}_X + v_1^{\top} K v_2 + v_2^{\top} K (-v_1 + B_c \mathcal{C}x) \\
&\le (\mathcal{B}x)^{\top} \mathcal{C}x + (B_c^{\top}Kv_2 )^{\top} \mathcal{C}x 
\le 0,
\end{align*}
where the last two inequalities follow from~\eqref{eq:S imp-passive}  and~\eqref{eq:S_c pos def}. Consequently, $\tilde{A}$ is dissipative in $\tilde{X}$, as desired.
\smallskip

As a second step, we show that $\tilde{A}-\lambda$ is surjective onto $\tilde{X}$ for %some $\lambda \in (0,\infty)$.
every $\lambda \in (0,\infty)$ with
\begin{align} \label{eq: wahl lambda}
2 \lambda  \ge A_c + A_c^{\top} \qquad 
(A_c := 
\begin{pmatrix}
0 & K \\ -I & 0
\end{pmatrix}
\in \R^{2m_c \times 2m_c}).
\end{align}
So let $\lambda \in (0,\infty)$ as above and let $\tilde{y} = (y,w) \in \tilde{X}$. We then have to find an $\tilde{x} = (x,v) \in D(\tilde{A})$, that is, an $\tilde{x} \in \tilde{X}$ with 
\begin{align} \label{eq: semigr 1}
x \in D(\mathcal{A}) \qquad \text{and} \qquad
(W_{B,2} + S_c W_C)(\mathcal{H}x)|_{\partial} + B_c^{\top} Kv_2 =  0, 
\end{align}
such that $(\tilde{A}-\lambda)\tilde{x} = \tilde{y}$, that is, 
\begin{align} \label{eq: semigr 2} 
(\mathcal{A}-\lambda)x = y 
\qquad \text{and} \qquad
(A_c - \lambda)v + \begin{pmatrix} 0 \\ B_c \mathcal{C}x \end{pmatrix} = w. %\label{eq: semigr 3} 
\end{align}
Since $A_c-\lambda$ is invertible for $\lambda \in (0,\infty)$, finding an $\tilde{x} \in \tilde{X}$ with~\eqref{eq: semigr 1} and  \eqref{eq: semigr 2} is equivalent to finding $x \in X$ %$x \in D(\mathcal{A})$ 
such that
\begin{gather}
x \in D(\mathcal{A}) \qquad \text{and} \qquad
\mathcal{B}'x %:= (W_{B,1}+G_{\lambda}W_C) x|_{\partial} 
= B_c^{\top}K \begin{pmatrix} 0 & I \end{pmatrix} (\lambda-A_c)^{-1} w \label{eq: semigr 4} \\
(\mathcal{A}-\lambda)x = y, \label{eq: semigr 5} 
\end{gather}
where $\mathcal{B}'x := (W_{B,2}+G_{\lambda}W_C) (\mathcal{H}x)|_{\partial}$ and 
\begin{gather}  \label{eq: G_lambda}
G_{\lambda} := S_c + B_c^{\top} K \begin{pmatrix} 0 & I \end{pmatrix} (\lambda-A_c)^{-1} \begin{pmatrix}
0 \\ B_c
\end{pmatrix}  
= S_c + \begin{pmatrix}
0 \\ B_c
\end{pmatrix}^{\top} L^{\top} (\lambda-A_c)^{-1} L \begin{pmatrix}
0 \\ B_c
\end{pmatrix} \notag \\
(L := 
\begin{pmatrix}
K^{1/2} & 0 \\
0       & K^{1/2}
\end{pmatrix}).
\end{gather}
(In the last equation we used that $L$ commutes with $A_c$.) 
In order to find an $\tilde{x}$ with~\eqref{eq: semigr 4} and~\eqref{eq: semigr 5}, we will show that the port-Hamiltonian operator $A' := \mathcal{A}|_{D(A')}$ with domain 
\begin{gather*}
D(A') := D(\mathcal{A}) \cap \ker \mathcal{B}' \\
= \big\{ x \in X: \mathcal{H}x \in W^{N,2}((a,b),\R^{m}), \, W_{B,1} (\mathcal{H}x)|_{\partial} = 0  \text{ and } (W_{B,2}+G_{\lambda}W_C)(\mathcal{H}x)|_{\partial} = 0 \big\}
\end{gather*}
%$D(A') := \{x \in D(\mathcal{A}): W'x|_{\partial} = 0\}$ and 
%\begin{align} \label{eq: W'}
%%A' := \mathcal{A}|_{D(A')} \qquad
%%D(A') := \{x \in D(\mathcal{A}): W'x = 0\} \qquad
%W' := \begin{pmatrix}
%W_{B,1} + G_{\lambda} W_C \\ W_{B,2}
%\end{pmatrix}
%=
%\begin{pmatrix}
%I_k & 0 & G \\
%0   & I_{d-k} & 0
%\end{pmatrix}
%\begin{pmatrix}
%W_{B,1} \\ W_{B,2} \\ W_C
%\end{pmatrix}
%\in \R^{d \times 2d}
%\end{align}
generates a contraction semigroup on $X$. 
%We will do so by using the standard criterion for a port-hamiltionian operator to generate a contraction semigroup, which was recalled after Condition~\ref{cond: 1} and which tells us that it is sufficient to establish dissipativity.
%
So let $x \in D(A')$ and $u := W_{B,2}(\mathcal{H}x)|_{\partial}$ and $y := W_C (\mathcal{H}x)|_{\partial}$, then $x \in D(\mathcal{A})$ and $u+G_{\lambda}y = 0$ and hence
\begin{align} \label{eq: A' dissip}
\scprd{x,A'x}_X = \scprd{x,\mathcal{A}x}_X \le u^{\top} y
= -\frac{1}{2} y^{\top} (G_{\lambda} + G_{\lambda}^{\top}) y
\end{align}
by virtue of~\eqref{eq:S imp-passive}. It follows from~\eqref{eq: G_lambda}, \eqref{eq:S_c pos def}, \eqref{eq: wahl lambda} that
\begin{align*}
G_{\lambda} + G_{\lambda}^{\top} 
=
%2 S_c + \begin{pmatrix}
%0 \\ B_c
%\end{pmatrix}^{\top} L^{\top} \big( (\lambda-A_c)^{-1} + (\lambda-A_c^{\top})^{-1} \big) L \begin{pmatrix}
%0 \\ B_c
%\end{pmatrix}
%= 
2 S_c + \begin{pmatrix}
0 \\ B_c
\end{pmatrix}^{\top} L^{\top}  (\lambda-A_c)^{-1} \big( 2 \lambda - (A_c + A_c^{\top}) \big) (\lambda-A_c^{\top})^{-1} L \begin{pmatrix}
0 \\ B_c
\end{pmatrix}
\ge 0
\end{align*}
and thus~\eqref{eq: A' dissip} yields the dissipativity of $A'$ in $X$. So, by the characterization of the contraction semigroup generator property for port-Hamiltonian operators from~\cite{JaKa17} (Theorem~2.3), $A'$ is a contraction semigroup generator on $X$ and, moreover, the boundary matrix 
\begin{align*}
W' := \begin{pmatrix}
W_{B,1} \\
W_{B,2} + G_{\lambda} W_C
\end{pmatrix}
\in \R^{mN \times 2mN}
\end{align*}
associated with $A'$ has full row rank. In particular, $\mathcal{A}$ and $\mathcal{B}'$ define a boundary control system by the same arguments as for Theorem~11.3.2 from~\cite{JacobZwart} and by Lemma~A.3 from~\cite{LeZwMa05}, and hence 
there exists an operator $B' \in L(U,X)$ with $U := \R^k$ such that $B'U \subset D(\mathcal{A}) = D(\mathcal{B}')$ and $\mathcal{A} B' \in L(U,X)$ and 
\begin{align} \label{eq: curly-B' surj}
\mathcal{B}' B' u = u \qquad (u \in U).
\end{align}
%that is, $A' = \mathcal{A}|_{D(\mathcal{A}) \cap \ker \mathcal{B}'}$ is a semigroup generator on $X$ and there exists an operator $B' \in L(U,X)$ with $U := \R^k$ such that $B'U \subset D(\mathcal{A}) = \mathcal{B}'$ and $\mathcal{A} B' \in L(U,X)$ and 
%\begin{align} \label{eq: curly-B' surj}
%\mathcal{B}' B' u = u \qquad (u \in U).
%\end{align}
With these preliminary considerations about $A'$ we can now finally prove the existence of an $x \in X$ with~\eqref{eq: semigr 4} and~\eqref{eq: semigr 5} and hence the surjectivity of $\tilde{A}-\lambda$. Indeed, with the ansatz 
\begin{align} \label{eq: ansatz surj}
x = h + x' \qquad (h \in D(\mathcal{A}) \text{ and } x' \in D(A')),
\end{align}
the conditions \eqref{eq: semigr 4} and~\eqref{eq: semigr 5} become equivalent to
\begin{align}
\mathcal{B}'h %:= (W_{B,1}+G_{\lambda}W_C) h|_{\partial}    
&= B_c^{\top}K \begin{pmatrix} 0 & I \end{pmatrix} (\lambda-A_c)^{-1} w   \label{eq: semigr 6}\\
(A'-\lambda)x' &= y - (\mathcal{A}-\lambda)h.    \label{eq: semigr 7}
\end{align}
And since by~\eqref{eq: curly-B' surj} $\mathcal{B}'$ is surjective and $A'$ is a contraction semigroup generator, %onto $U$
we really can find $h \in D(\mathcal{B}') = D(\mathcal{A})$ and $x' \in D(A')$ such that~\eqref{eq: semigr 6} and~\eqref{eq: semigr 7} are satisfied, as desired.
\smallskip

As a third step, we show that $\tilde{A}$ has compact resolvent. So let $\lambda \in (0,\infty)$ with~\eqref{eq: wahl lambda}. Also, let $(\tilde{y}_n) = (y_n,w_n)$ be a bounded sequence in $\tilde{X}$ and write 
\begin{align*}
\begin{pmatrix}
x_n \\ v_n
\end{pmatrix}
=
\tilde{x}_n := (\tilde{A}-\lambda)^{-1} \tilde{y}_n 
= 
(\tilde{A}-\lambda)^{-1}
\begin{pmatrix}
y_n \\ w_n
\end{pmatrix}
\qquad (n \in \N).
\end{align*}
It then follows by~\eqref{eq: ansatz surj}, \eqref{eq: semigr 6}, \eqref{eq: semigr 7}, \eqref{eq: curly-B' surj} that $x_n = h_n + x_n'$ with 
\begin{align}
h_n &= B' \big( B_c^{\top}K \begin{pmatrix} 0 & I \end{pmatrix} (\lambda-A_c)^{-1} w_n \big)  \\
x_n' &= (A'-\lambda)^{-1} \Big( y_n - (\mathcal{A}-\lambda) B' \big( B_c^{\top}K \begin{pmatrix} 0 & I \end{pmatrix} (\lambda-A_c)^{-1} w_n \big) \Big).
\end{align}
Since $(\tilde{y}_n) = (y_n,w_n)$ is bounded, the sequences %$(w_n)$, $(v_n)$, 
\begin{align*}
(w_n), \qquad (v_n), \qquad \big( y_n - (\mathcal{A}-\lambda) B' \big( B_c^{\top}K \begin{pmatrix} 0 & I \end{pmatrix} (\lambda-A_c)^{-1} w_n \big)\big)
\end{align*}
are bounded as well. It follows by the finite-dimensionality of $\R^{2m_c}$ and the compactness of $(A'-\lambda)^{-1}$ (Theorem~2.28 in~\cite{Villegas}) that there exists a subsequence $(n_k)$ of $(n)$ sucht that $(w_{n_k})$ and $(v_{n_k})$ converge in $\R^{2m_c}$ and such that $(x_{n_k}')$ converges in $X$. Consequently, $(\tilde{x}_{n_k})$ is convergent as well, and we are done.
\end{proof}

\section{Solvability of the closed-loop system} \label{sect:solvb}

%After having set the stage, we now show
In this section, we show that under suitable conditions the initial value problem 
\begin{equation} \label{eq:ivp closed-loop d ne 0}
\begin{gathered}
\tilde{x}' = \tilde{\mathcal{A}} \tilde{x} + \tilde{f}(\tilde{x}) \qquad \text{and} \qquad \tilde{x}(0) = \tilde{x}_0 \\
d(t) = \tilde{\mathcal{B}}\tilde{x}(t)
\end{gathered}
\end{equation}
of the closed-loop system has a global solution for suitable initial values $\tilde{x}_0$ and disturbance inputs $d$. 
%of the closed-loop system, for sufficiently regular initial values $\tilde{x}_0$ and external inputs $d$, %$\tilde{x}_0 \in D(\tilde{\mathcal{A}})$ and $d \in C^2([0,\infty),\R^k)$ with $d(0) = \tilde{\mathcal{B}}\tilde{x}_0$.
%has a (unique) global classical solution %existing globally in time. 
%and, for arbitrary $\tilde{x}_0 \in \tilde{X}$ and $d \in L^2_{\mathrm{loc}}([0,\infty),\R^k)$, has a global generalized solution that continuously depends on $(\tilde{x}_0,d)$. 
We will achieve this by applying the standard theory of semilinear evolution equations from~\cite{Pazy}. 
As it stands, however, the closed-loop equation~\eqref{eq:ivp closed-loop d ne 0} is not a semilinear evolution equation in the sense of~\cite{Pazy} because, for one thing, the linear part $\tilde{\mathcal{A}}$ of the differential equation is not a semigroup generator and because, for another thing, in addition to the differential equation the side condition %there is the side condition
\begin{align*}
d(t) = \tilde{\mathcal{B}}\tilde{x}(t)
\end{align*}
occurs. %appears/shows up. 
%Yet, this theory is not directly applicable here because~\eqref{eq:ivp closed-loop d ne 0} is not a semilinear evolution equation in the sense of~\cite{Pazy}: for one thing, the linear part $\tilde{\mathcal{A}}$ of the differential equation is not a semigroup generator and for another thing, in addition to the differential equation there is the side condition
%\begin{align*}
%d(t) = \tilde{\mathcal{B}}\tilde{x}(t).
%\end{align*}
A way out of this difficulty is to impose the following extra condition, 
%In order to overcome this difficulty, %turn \eqref{eq:ivp closed-loop d ne 0} into a truly semilinear evolution equation
%we impose the following condition. 
which is easily seen to be satisfied if, for instance, 
\begin{align} \label{eq:W-matrix, def}
W := \begin{pmatrix}
W_{B} \\ W_C
\end{pmatrix}
\in \R^{(mN+k)\times 2mN}
\qquad \text{with} \qquad
W_B := \begin{pmatrix}
W_{B,1} \\ W_{B,2}
\end{pmatrix}
\end{align}
is a matrix of full row rank $mN+k$.

\begin{cond} \label{cond:B-tilde surj}
$\tilde{\mathcal{B}}$ is a surjective linear map from $D(\tilde{\mathcal{A}})$ onto $\R^k$.
\end{cond}

With the help of this extra condition, we can turn~\eqref{eq:ivp closed-loop d ne 0} into a truly semilinear evolution equation in the sense of~\cite{Pazy}. 
In fact, Condition~\ref{cond:B-tilde surj} implies the existence of a linear right-inverse $\tilde{R}: \R^k \to D(\tilde{\mathcal{A}})$ of $\tilde{\mathcal{B}}$, that is,
\begin{align*}
\tilde{\mathcal{B}}\tilde{R}d = d 
\end{align*} 
for all $d \in \R^k$. We can thus perform the transformation
\begin{align} \label{eq:Fattorini transformation}
\tilde{\xi}(t) = \tilde{x}(t) - \tilde{R}d(t)
\end{align}
which is well-known from linear boundary control problems~\cite{Fa68}, \cite{JacobZwart}, \cite{CurtainZwart}. %(Fattorini trick)
Via this transformation, the classical solutions $\tilde{x}$ of~\eqref{eq:ivp closed-loop d ne 0} are in one-to-one correspondence -- for continuously differentiable disturbances $d$ -- with the classical solutions $\tilde{\xi}$ %Ich ignoriere hier, dass die beiden Awp jeweils eine *eindeutige* Loesung haben. Die Eindeutigkeit ist hier naemlich nicht der Punkt. 
%Variante, wo ich die Eindeutigkeit schon miteinbeziehe/mit einfliessen lasse: the classical solution $\tilde{x}$ of~\eqref{eq:ivp closed-loop d ne 0} is in one-to-one correspondence -- for continuously differentiable disturbances $d$ -- with the classical solution $\tilde{\xi}$
of 
\begin{align} \label{eq:ivp closed-loop d ne 0, transformed}
\tilde{\xi}' = \tilde{A}\tilde{\xi} + \tilde{f}(\tilde{\xi}+\tilde{R}d(t)) + \tilde{\mathcal{A}} \tilde{R}d(t) - \tilde{R}d'(t) 
\qquad \text{and} \qquad 
\tilde{\xi}(0) = \tilde{x}_0 - \tilde{R}d(0).
%\tilde{\xi}_0
\end{align}
%with $\tilde{\xi}_0 := \tilde{x}_0 - \tilde{R}d(0)$. See the proof of Lemma~\ref{lm:lokal klass loesb, output stet} for details. 
And this is now, in view of Lemma~\ref{lm:sgr gen with comp resolv}, a truly semilinear evolution equation (with an explicitly time-dependent nonlinearity).
We can therefore apply standard semilinear theory to obtain classical solvability of~\eqref{eq:ivp closed-loop d ne 0} for sufficiently regular $\tilde{x}_0$ and $d$ (Section~\ref{sect:solvb classical}), %We then also obtain, by a suitable density and approximation argument,  %at least 
and, by a suitable density and approximation argument, we then also  obtain generalized solvability of~\eqref{eq:ivp closed-loop d ne 0} for sufficiently irregular $\tilde{x}_0$ and $d$ (Section~\ref{sect:solvb generalized}). 
%We can therefore apply standard semilinear theory to obtain classical solvability of~\eqref{eq:ivp closed-loop d ne 0} for sufficiently regular $\tilde{x}_0$ and $d$ (Section~\ref{sect:solvb classical}).  And to obtain %at least 
%generalized solvability of~\eqref{eq:ivp closed-loop d ne 0} for sufficiently irregular $\tilde{x}_0$ and $d$, we can then apply a suitable density and approximation argument (Section~\ref{sect:solvb generalized}). 

\subsection{Solvability in the classical sense} \label{sect:solvb classical}

In this section, we show that for sufficiently regular initial states $\tilde{x}_0$ and disturbances $d$, namely for  
\begin{align} \label{eq:class data, def}
(\tilde{x}_0,d) \in \mathcal{D} := \big\{ (\tilde{x}_0,d) \in D(\tilde{\mathcal{A}}) \times C_c^2([0,\infty),\R^k): d(0) = \tilde{\mathcal{B}} \tilde{x}_0 \big\}
\end{align}
(set of classical data), the closed-loop equation~\eqref{eq:ivp closed-loop d ne 0} has a classical solution existing globally in time. 
A \emph{classical solution} of~\eqref{eq:ivp closed-loop d ne 0} is a continuously differentiable function $\tilde{x}: J \to \tilde{X}$ on an interval $J \subset [0,\infty)$ containing $0$ such that for every $t \in J$ one has (i) that $\tilde{x}(t) \in D(\tilde{\mathcal{A}})$ and (ii) that %~\eqref{eq: ivp closed-loop, r ne 0}.
\begin{equation} 
\begin{gathered}
\tilde{x}'(t) = \tilde{\mathcal{A}} \tilde{x}(t) + \tilde{f}(\tilde{x}(t)) \qquad \text{and} \qquad \tilde{x}(0) = \tilde{x}_0 \\
d(t) = \tilde{\mathcal{B}}\tilde{x}(t).
\end{gathered}
\end{equation}

\begin{thm} \label{thm:class solvb}
Suppose that Conditions~\ref{cond:S imp-passive}, \ref{cond:P und R}, \ref{cond:B-tilde surj} are satisfied. Then 
\begin{itemize}
\item[(i)] for every $(\tilde{x}_0,d) \in \mathcal{D}$ the closed-loop equation~\eqref{eq:ivp closed-loop d ne 0} has a unique global classical solution 
\begin{align}
\tilde{x}(\cdot,\tilde{x}_0,d) \in C^1([0,\infty),\tilde{X})
\end{align}
and the output $y(\cdot,\tilde{x}_0,d) := \tilde{\mathcal{C}} \tilde{x}(\cdot,\tilde{x}_0,d)$
%\begin{align}
%y(\cdot,\tilde{x}_0,d) := \tilde{\mathcal{C}} \tilde{x}(\cdot,\tilde{x}_0,d)
%\end{align}
is a continuous function from $[0,\infty)$ to $\R^k$
%the output $y(\cdot,\tilde{x}_0,d) := \tilde{\mathcal{C}} \tilde{x}(\cdot,\tilde{x}_0,d)$ is a continuous function from $[0,\infty)$ to $\tilde{X}$
\item[(ii)] there exist $\ul{\sigma}, \ul{\gamma} \in \mathcal{K}$ such that for every $(\tilde{x}_0,d) \in \mathcal{D}$
\begin{align} \label{eq:UGS, class}
\norm{\tilde{x}(t,\tilde{x}_0,d)} \le \ul{\sigma}(\norm{\tilde{x}_0}) + \ul{\gamma}(\norm{d}_{[0,t],2}) 
\qquad (t \in [0,\infty)).
\end{align}
\end{itemize}
\end{thm}

\begin{proof}
%Schritt 1
As a first step, we show that for every $(\tilde{x}_0,d) \in \mathcal{D}$ the initial value problem~\eqref{eq:ivp closed-loop d ne 0} has a unique maximal classical solution $\tilde{x}(\cdot,\tilde{x}_0,d)$ on some half-open interval $[0,T_{\tilde{x}_0,d})$. 
So let $(\tilde{x}_0,d) \in \mathcal{D}$. As is easily verified, a function $\tilde{x}$ is a maximal classical solution of~\eqref{eq:ivp closed-loop d ne 0} if and only if the function $\tilde{x}-\tilde{R}d$ is a maximal classical solution of~\eqref{eq:ivp closed-loop d ne 0, transformed}. We therefore have only to show that~\eqref{eq:ivp closed-loop d ne 0, transformed} has a unique maximal classical solution existing on some half-open interval of the form $[0,T)$. And, in view of Lemma~\ref{lm:sgr gen with comp resolv}, this immediately follows by the standard solvability results for semilinear evolution equations in reflexive spaces (Theorem~6.1.4 and~6.1.6 in~\cite{Pazy}). 
In particular, it follows by variation of constants that for every $(\tilde{x}_0,d) \in \mathcal{D}$ the corresponding maximal classical solution $\tilde{x} = \tilde{x}(\cdot,\tilde{x}_0,d)$ %whose unique existence we have just proven
satisfies the following integral equation: 
\begin{align}  \label{eq:voc/mild sol, kappa=0}
\tilde{x}(t) = \e^{\tilde{A}t} \tilde{x}_0  + \int_0^t \e^{\tilde{A}(t-s)} \tilde{f}(\tilde{x}(s)) \d s + \Phi_t(d)
%
%\tilde{x}(t,\tilde{x}_0,d) = \e^{\tilde{A}t} \tilde{x}_0  + \int_0^t \e^{\tilde{A}(t-s)} \tilde{f}(\tilde{x}(s,\tilde{x}_0,d)) \d s + \Phi_t(d)
\end{align}
for every $t \in [0,T_{\tilde{x}_0,d})$, where
\begin{align} \label{eq:Phi_t, def}
\Phi_t(d) := - \e^{\tilde{A}t} \tilde{R}d(0) + \tilde{R}d(t) + \int_0^t \e^{\tilde{A}(t-s)} \big( \tilde{\mathcal{A}}\tilde{R} d(s) - \tilde{R}d'(s) \big) \d s.
\end{align}

%Schritt 2
As a second step, we show that for every $(\tilde{x}_0,d) \in \mathcal{D}$ the output $y(\cdot,\tilde{x}_0,d) := \tilde{\mathcal{C}}\tilde{x}(\cdot,\tilde{x}_0,d)$ is a continuous function from $[0,T_{\tilde{x}_0,d})$ to $\R^k$. 
So let $(\tilde{x}_0,d) \in \mathcal{D}$, write $(x,v) := \tilde{x} := \tilde{x}(\cdot,\tilde{x}_0,d)$, and let 
$$\norm{\cdot}_{\mathcal{A}} := (\norm{\cdot}_X^2 + \norm{\mathcal{A} \,\cdot \,}_X^2)^{1/2}$$ 
denote the graph norm of $\mathcal{A}$. Since $\tilde{x}$ is a classical solution of~\eqref{eq:ivp closed-loop d ne 0}, the functions
$t \mapsto x(t)$ and $t \mapsto \mathcal{A}x(t) = x'(t)$ are continuous from $[0,T_{\tilde{x}_0,d})$ to $X$ and therefore
\begin{align}
[0,T_{\tilde{x}_0,d}) \ni t \mapsto x(t) \in (D(\mathcal{A}),\norm{\cdot}_{\mathcal{A}})
\end{align}
is continuous. Since, moreover, 
\begin{align} \label{eq:class solvb, 2.1}
\ul{c} \norm{\mathcal{H}\xi}_{W^{N,2}} \le \norm{\xi}_{\mathcal{A}} \le \ol{c} \norm{\mathcal{H}\xi}_{W^{N,2}} \qquad (\xi \in D(\mathcal{A}))
\end{align}
for some positive constants $\ul{c}, \ol{c}$ (Lemma~3.2.3 in~\cite{Augner}), the map
\begin{align} \label{eq:class solvb, 2.2}
D(\mathcal{A}) \ni \xi \mapsto \mathcal{C}\xi = W_C (\mathcal{H}\xi)|_{\partial} \in \R^k
\end{align}
is continuous as well. Combining now~\eqref{eq:class solvb, 2.1} and~\eqref{eq:class solvb, 2.2}, the second step follows. 
\smallskip

%Schritt 3
As a third step, we show that there exist $\ul{\sigma}, \ul{\gamma} \in \mathcal{K}$ such that for every $(\tilde{x}_0,d) \in \mathcal{D}$
\begin{align} \label{eq:class solvb, step 3}
\norm{\tilde{x}(t,\tilde{x}_0,d)} \le \ul{\sigma}(\norm{\tilde{x}_0}) + \ul{\gamma}(\norm{d}_{[0,t],2}) 
\end{align}
for all $t \in [0,T_{\tilde{x}_0,d})$. 
So let $(\tilde{x}_0,d) \in \mathcal{D}$ and write $(x,v) := \tilde{x}(\cdot,\tilde{x}_0,d)$ and $y := y(\cdot,\tilde{x}_0,d) = \mathcal{C}x$. Since $\tilde{x}(\cdot,\tilde{x}_0,d)$ is a classical solution by the first step, the function 
\begin{align*}
s \mapsto \tilde{E}_{\tilde{x}}(s) := \tilde{E}(\tilde{x}(s,\tilde{x}_0,d) = \frac{1}{2} \norm{x(s)}_X^2 + \mathcal{P}(v_1(s)) + \frac{1}{2} v_2(s)^{\top}K v_2(s)
\end{align*}
%$s \mapsto \tilde{E}_{\tilde{x}}(s) := \tilde{E}(\tilde{x}(s,\tilde{x}_0,d) = \frac{1}{2} \norm{x(s)}_X^2 + \mathcal{P}(v_1(s)) + \frac{1}{2} v_2(s)^{\top}K v_2(s)$ 
is continuously differentiable and its derivative satisfies 
\begin{align} \label{eq:class solvb, derivative of energy}
\tilde{E}_{\tilde{x}}'(s) 
&= 
\scprd{x(s),\mathcal{A}x(s)}_X - (Kv_2(s))^{\top} \mathcal{R}(Kv_2(s)) + \big( B_c^{\top}Kv_2(s)\big)^{\top} y(s) \notag \\
&=
\scprd{x(s),\mathcal{A}x(s)}_X - (Kv_2(s))^{\top} \mathcal{R}(Kv_2(s)) \notag \\
&\qquad \qquad \qquad \quad \,\,\,\,\, + d(s)^{\top}y(s) - (\mathcal{B}x(s))^{\top}y(s) - y(s)^{\top}S_c y(s)
\end{align}
for all $s \in [0,T_{\tilde{x}_0,d})$. With the help of Condition~\ref{cond:S imp-passive} and~\ref{cond:P und R} we therefore get that
\begin{align} \label{eq:class solvb, derivative of energy, estimate}
\tilde{E}_{\tilde{x}}'(s) 
%
%\le d(s)^{\top} y(s) - (Kv_2(s))^{\top} \mathcal{R}(Kv_2(s)) - y(s)^{\top}S_c y(s)
%
&\le 
|d(s)||y(s)| - (Kv_2(s))^{\top} \mathcal{R}(Kv_2(s)) - \varsigma |y(s)|^2 \notag \\
&\le 
\frac{\alpha}{2} |d(s)|^2 + \frac{1}{2\alpha} |y(s)|^2  - \varsigma |y(s)|^2 - (Kv_2(s))^{\top} \mathcal{R}(Kv_2(s)) \notag \\
&\le 
\frac{\alpha}{2} |d(s)|^2 + \frac{1}{2\alpha} |y(s)|^2  - \varsigma |y(s)|^2
\end{align}
for all $s \in [0,T_{\tilde{x}_0,d})$ and arbitrary $\alpha \in (0,\infty)$, where $\varsigma$ as in~\eqref{eq:S_c pos def} is the smallest eigenvalue of $S_c$. 
Choosing now $\alpha := 1/(2\varsigma)$, we see from~\eqref{eq:class solvb, derivative of energy, estimate} by integration that
\begin{align} \label{eq:UGS, energy, class}
\tilde{E}(\tilde{x}(t,\tilde{x}_0,d)) 
%
%= \tilde{E}(\tilde{x}_0) + \int_0^t \tilde{E}_{\tilde{x}}'(s) \d s
%
\le \tilde{E}(\tilde{x}_0) + \frac{1}{4\varsigma} \norm{d}_{[0,t],2}^2
\end{align}
for every $t \in [0,T_{\tilde{x}_0,d})$. 
Since $\tilde{E}$ is equivalent to the norm of $\tilde{X}$ by~\eqref{eq:E-tilde equiv to norm}, we further conclude that
\begin{align*}
\norm{\tilde{x}(t,\tilde{x}_0,d)} \le \ul{\psi}^{-1}\Big( \ol{\psi}(\norm{\tilde{x}_0}) + \frac{1}{4\varsigma} \norm{d}_{[0,t],2}^2 \Big)
\le 
\ul{\psi}^{-1}\big( 2 \ol{\psi}(\norm{\tilde{x}_0}) \big) + \ul{\psi}^{-1}\Big( \frac{1}{2\varsigma} \norm{d}_{[0,t],2}^2 \Big)
\end{align*}
for every $t \in [0,T_{\tilde{x}_0,d})$. So, \eqref{eq:class solvb, step 3} follows because $\ul{\sigma}, \ul{\gamma}$ defined by
\begin{align}
\ul{\sigma}(r) := \ul{\psi}^{-1}\big( 2 \ol{\psi}(r)\big)
\qquad \text{and} \qquad
\ul{\gamma}(r) := \ul{\psi}^{-1}\Big( \frac{1}{2\varsigma} r^2 \Big)
\end{align}
obviously belong to the class $\mathcal{K}$. 
\smallskip

%Schritt 4
As a fourth and last step, we show that the maximal existence time $T_{\tilde{x}_0,d} = \infty$. Combined with the previous steps, %In conjunction with the previous steps, 
this proves the theorem. 
So let $(\tilde{x}_0,d) \in \mathcal{D}$ and assume that $T_{\tilde{x}_0,d} < \infty$, that is, the existence interval $[0,T_{\tilde{x}_0,d})$ of the maximal classical solution $\tilde{x}(\cdot,\tilde{x}_0,d)-\tilde{R}d$ of~\eqref{eq:ivp closed-loop d ne 0, transformed} is bounded. So, by the standard blow-up result for semilinear evolution equations (Theorem~6.1.4 of~\cite{Pazy}), the solution $\tilde{x}(\cdot,\tilde{x}_0,d)-\tilde{R}d$ of~\eqref{eq:ivp closed-loop d ne 0, transformed} blows up. And therefore $\tilde{x}(\cdot,\tilde{x}_0,d)$ blows up as well:
\begin{align}
\sup_{t \in [0,T_{\tilde{x}_0,d})} \norm{ \tilde{x}(t,\tilde{x}_0,d) }
\ge \sup_{t \in [0,T_{\tilde{x}_0,d})} \norm{ \tilde{x}(t,\tilde{x}_0,d) - \tilde{R}d(t)} - \|\tilde{R}\| \norm{d}_{[0,T_{\tilde{x}_0,d}),\infty} = \infty
\end{align}
(recall that $d$ is continuous with compact support). Contradiction to the estimate from the third step!
\end{proof}

%Bemerkung zur eindeutigen Existenz maximaler milder Loesungen 
It can be shown that for $(\tilde{x}_0,d) \in \tilde{X} \times W^{1,2}_{\mathrm{loc}}([0,\infty),\R^k)$, the initial value problem~\eqref{eq:ivp closed-loop d ne 0} still has a (unique) global mild solution, that is, a continuous function $\tilde{x}: [0,\infty) \to \tilde{X}$ satisfying the variation-of-constants formula~\eqref{eq:voc/mild sol, kappa=0}. In order to see this -- especially the global existence -- one can make a density and approximation argument similar to the one from the next section: one first trivially extends $\Phi_t$ to a bounded linear operator on $W^{1,2}([0,\infty),\R^k)$ and then uses the same arguments as in Theorem~\ref{thm:density and approximation}. %(See the proof of Theorem~\ref{thm:density and approx} -- the bounded extendability of $\Phi_t$ to $W^{1,2}([0,\infty),\R^k)$ is trivial.)
Since we will not need mild solutions in the sequel, however, we omit the details. %we do not elaborate on them any further
%
%With a suitable density and approximation argument similar to the one from the next section, it can be seen that for $(\tilde{x}_0,d) \in \tilde{X} \times W^{1,2}_{\mathrm{loc}}([0,\infty),\R^k)$, the initial value problem~\eqref{eq:ivp closed-loop d ne 0} still has a (unique) global mild solution, that is, a continuous function $\tilde{x}: [0,\infty) \to \tilde{X}$ satisfying the variation-of-constants formula~\eqref{eq:voc, kappa=0}.

\subsection{Solvability in the generalized sense} \label{sect:solvb generalized}

In this section, we show that for sufficiently irregular initial states $\tilde{x}_0$ and disturbances~$d$, namely for  
\begin{align} \label{eq:general data, def}
(\tilde{x}_0,d) \in \tilde{X} \times L^2_{\mathrm{loc}}([0,\infty),\R^k)
\end{align}
(set of generalized data), the closed-loop equation~\eqref{eq:ivp closed-loop d ne 0} still has a generalized solution existing globally in time and arising as a suitable limit of classical solutions.
%a global generalized solution which arises as a limit of classical solutions. 
In showing the  existence of such limits we will make use of the  integral equation
\begin{align}  \label{eq:voc, kappa=0}
\tilde{x}(t,\tilde{x}_0,d) = \e^{\tilde{A}t} \tilde{x}_0  + \int_0^t \e^{\tilde{A}(t-s)} \tilde{f}(\tilde{x}(s,\tilde{x}_0,d)) \d s + \Phi_t(d)
%\qquad (t \in [0,\infty))
\end{align}
for classical solutions $\tilde{x}(\cdot,\tilde{x}_0,d)$ with $(\tilde{x}_0,d) \in \mathcal{D}$ from~\eqref{eq:voc, kappa=0} above. 
In order to extend this equation %A first step in extending this equation
to arbitrary $(\tilde{x}_0,d) \in \tilde{X} \times L^2_{\mathrm{loc}}([0,\infty),\R^k)$, we first extend the linear operators $\Phi_t$ from~\eqref{eq:Phi_t, def}. %show that the linear operators $\Phi_t$ from~\eqref{eq:Phi_t, def} can be continuously extended.
%We first show that the linear operators $\Phi_t$ from~\eqref{eq:Phi_t, def} can be continuously extended. 

\begin{lm} \label{lm:S_lin admissible}
Suppose that Conditions~\ref{cond:S imp-passive}, \ref{cond:P und R}, \ref{cond:B-tilde surj} are satisfied. Then the linear operator $\Phi_t: C_c^2([0, \infty),\R^k) \to \tilde{X}$ defined by~\eqref{eq:Phi_t, def} can be (uniquely) extended to a bounded linear operator $\ol{\Phi}_t: L^2([0,\infty),\R^k) \to \tilde{X}$ for every $t \in [0,\infty)$ and
\begin{align}
\sup_{t\in[0,\tau]} \norm{\ol{\Phi}_t} \le \norm{\ol{\Phi}_{\tau}}
\qquad (\tau \in [0,\infty)).
\end{align}
%for every $\tau \in [0,\infty)$. %In particular, $[0,\infty) \ni t \mapsto \ol{\Phi_t}(d)$ is continuous for every $d \in L^2_{\mathrm{loc}}([0,\infty),\R^k)$. %$(t,d) \mapsto \ol{\Phi_t}(d)$ is continuous. 
\end{lm}

\begin{proof}
In the entire proof, we use the short-hand notation
\begin{align}
R(\tilde{x}_0,d) := \ul{\sigma}(\norm{\tilde{x}_0}) + \ul{\gamma}(\norm{d}_2)
\end{align}
for $\tilde{x}_0 \in \tilde{X}$ and $d \in L^2([0,\infty),\R^k)$, where $\ul{\sigma}, \ul{\gamma} \in \mathcal{K}$ are chosen as in Theorem~\ref{thm:class solvb}. We also denote by $L(R)$ for $R \in [0,\infty)$ a Lipschitz constant of $\tilde{f}|_{\ol{B}_R(0)}$ such that $R \mapsto L(R)$ is continuous and monotonically increasing. Such a choice of Lipschitz constants exists because
\begin{align*}
R \mapsto L_0(R) := \min \big\{ L \in [0,\infty): L \text{ is a Lipschitz constant of } \tilde{f}|_{\ol{B}_R(0)} \big\} 
\end{align*}
is monotonically increasing and therefore obviously has a continuous monotonically increasing majorant. (See Lemma~2.5 of~\cite{ClLeSt98}.)
\smallskip

%Schritt 1
As a first step, we show that the restriction $\Phi_{t,0}$ of $\Phi_t$ to 
\begin{align*}
C_{c,0}^2([0,\infty),\R^k) := \big\{ d \in C_c^2([0,\infty),\R^k): d(0) = 0 \big\}
\end{align*}
can be (uniquely) extended to a bounded linear operator  $\ol{\Phi}_{t,0}: L^2([0,\infty),\R^k) \to \tilde{X}$ for every $t \in [0,\infty)$ and that 
\begin{align} \label{eq:S_lin admissible, 1}
\sup_{t\in[0,\tau] } \norm{ \ol{\Phi}_{t,0} } \le \norm{ \ol{\Phi}_{\tau,0} }
\end{align}
for every $\tau \in [0,\infty)$.
So let $t \in [0,\infty)$. We see by~\eqref{eq:voc, kappa=0} that
\begin{align} \label{eq:voc, kappa=0, x_0 = 0}
\Phi_{t,0}(d) = \tilde{x}(t,0,d) - \int_0^t \e^{\tilde{A}(t-s)} \tilde{f}(\tilde{x}(s,0,d)) \d s
\end{align}
for every $d \in C_{c,0}^2([0,\infty),\R^k)$ (because for such $d$ one has $(0,d) \in \mathcal{D}$). 
%
%Schritt 1.1
It follows by Theorem~\ref{thm:class solvb}~(ii) that
\begin{align}  \label{eq:S_lin admissible, 1.1}
\norm{\tilde{x}(t,0,d)} \le \ul{\gamma}(\norm{d}_{[0,t],2})
\end{align}
for all $d \in C_{c,0}^2([0,\infty),\R^k)$.
%
%Schritt 1.2
It also follows by Theorem~\ref{thm:class solvb}~(ii) that
\begin{align}  \label{eq:S_lin admissible, 1.2}
\norm{ \int_0^t \e^{\tilde{A}(t-s)} \tilde{f}(\tilde{x}(s,0,d)) \d s } \le L(R(0,d)) \, \ul{\gamma}(\norm{d}_{[0,t],2}) \, t
\end{align}
for all $d \in C_{c,0}^2([0,\infty),\R^k)$.
%
%Schritt 1.3
Combining now~\eqref{eq:voc, kappa=0, x_0 = 0}, \eqref{eq:S_lin admissible, 1.1} and~\eqref{eq:S_lin admissible, 1.2} we see that
\begin{align} \label{eq:absch Phi_t,0(d)}
\norm{\Phi_{t,0}(d)} \le \big( 1+L(R(0,d)) t \big) \,  \ul{\gamma}( \norm{d}_{[0,t],2} ) 
\end{align}
for all $d \in C_{c,0}^2([0,\infty),\R^k)$.
It follows %from this 
that $$\Phi_{t,0}: C_{c,0}^2([0,\infty),\R^k) \to \tilde{X}$$ is a linear operator that is bounded w.r.t.~the norm of $L^2([0,\infty),\R^k)$ and therefore, by the density of $C_{c,0}^2([0,\infty),\R^k)$ in $L^2([0,\infty),\R^k)$, can be uniquely extended to a bounded linear operator $\ol{\Phi}_{t,0}: L^2([0,\infty),\R^k) \to \tilde{X}$. 
Since for every $\tau \ge t$ one has the elementary relation
\begin{align}
\Phi_{t,0}(d) = \Phi_{\tau,0}(0 \,\&_{\tau-t} \,d) \qquad (d \in C_{c,0}^2([0,\infty),\R^k))
\end{align}
(called composition property in~\cite{We89}), it further follows that
\begin{align*} %\label{eq:absch ol(Phi)_t,0(d)}
&\sup_{t\in [0,\tau]} \norm{\ol{\Phi}_{t,0}} 
= \sup_{t\in [0,\tau]} \sup \{ \Phi_{t,0}(d): d \in C_{c,0}^2([0,\infty),\R^k) \text{ with } \norm{d}_2 \le 1 \big\} \notag \\
&\qquad \le 
\sup_{t\in [0,\tau]} \sup \{ \Phi_{\tau,0}(0 \, \&_{\tau-t} \, d): d \in C_{c,0}^2([0,\infty),\R^k) \text{ with } \norm{d}_2 \le 1 \big\}
\le 
\norm{ \ol{\Phi}_{\tau,0} }.
\end{align*}
In the relations above, $u\,\&_{\tau} \,v$ denotes the concatenation of two functions $u,v:[0,\infty) \to \R^k$ at time $\tau$, that is, $(u\,\&_{\tau} \,v)(s) := u(s)$ for $s \in [0,\tau)$ and $(u\,\&_{\tau} \,v)(s) := v(s-\tau)$ for $s \in [\tau,\infty)$. 
\smallskip

%Schritt 2
As a second step, we will show that also the non-restricted operator $\Phi_t$ can be (uniquely) extended to a bounded linear operator  $\ol{\Phi}_{t}: L^2([0,\infty),\R^k) \to \tilde{X}$ and that $\ol{\Phi}_{t} = \ol{\Phi}_{t,0}$ for every $t \in [0,\infty)$, which in conjunction with~\eqref{eq:S_lin admissible, 1} proves the lemma. 
In order to do so, we have only to show that 
\begin{align}  \label{eq:S_lin admissible, 2}
\ol{\Phi}_{t,0}(d) = \Phi_t(d)
\end{align} 
for every $d \in C_c^2([0,\infty),\R^k)$. So let $t \in (0,\infty)$ and $d \in C_c^2([0,\infty),\R^k)$ be fixed (for $t =0$ the claim is obvious). Choose a sequence $(d_n)$ in $C_{c,0}^2([0,\infty),\R^k)$ such that
\begin{align} \label{eq:d_n for extension from C_c,0^2 to C_c^2}
d_n\big|_{[t/n,\infty)} = d\big|_{[t/n,\infty)} \qquad \text{and} \qquad \sup_{n \in \N} \norm{d_n \big|_{[0,t/n)}}_{\infty} < \infty.
\end{align}
It then follows that $d_n \longrightarrow d$ in $L^2([0,\infty),\R^k)$. 
%\begin{align*}
%d_n \longrightarrow d \qquad (n \to \infty)
%\end{align*}
%in $L^2([0,\infty),\R^k)$. 
So, on the one hand
\begin{align}  \label{eq:S_lin admissible, 2.1}
\Phi_{t,0}(d_n) \longrightarrow \ol{\Phi}_{t,0}(d) \qquad (n \to \infty)
\end{align}
and on the other hand
\begin{align}   \label{eq:S_lin admissible, 2.2}
&\Phi_{t,0}(d_n) = \tilde{R}d_n(t) + e^{\tilde{A}(t-t/n)} \int_0^{t/n} \e^{\tilde{A}(t/n-s)} \big( \tilde{\mathcal{A}}\tilde{R} d_n(s) - \tilde{R}d_n'(s) \big) \d s \notag \\ 
&\qquad + \int_{t/n}^t e^{\tilde{A}(t-s)} \big( \tilde{\mathcal{A}}\tilde{R} d(s) - \tilde{R}d'(s) \big) \d s  \notag \\
&\quad =  \tilde{R}d_n(t) + e^{\tilde{A}(t-t/n)} \Big( \Phi_{t/n,0}(d_n) - \tilde{R}d_n(t/n) \Big) + \int_{t/n}^t e^{\tilde{A}(t-s)} \big( \tilde{\mathcal{A}}\tilde{R} d(s) - \tilde{R}d'(s) \big) \d s \notag \\
&\quad  \longrightarrow  \Phi_t(d) \qquad (n \to \infty),
\end{align}
where we used~\eqref{eq:absch Phi_t,0(d)} to get $\Phi_{t/n,0}(d_n) \longrightarrow 0$ as $n \to \infty$. Combining now~\eqref{eq:S_lin admissible, 2.1} and~\eqref{eq:S_lin admissible, 2.2} we finally obtain~\eqref{eq:S_lin admissible, 2}, as desired.
\end{proof}

%Bemerkung: was bedeutet das obige Lemma?
It should be noticed that the above lemma means nothing but %means precisely that %simply amounts to 
the (finite-time) admissibility of the linear boundary control system
\begin{align}
\tilde{x}' = \tilde{\mathcal{A}}\tilde{x} 
\qquad \text{with} \qquad
d(t) = \tilde{\mathcal{B}}\tilde{x}(t)
\end{align}
w.r.t.~inputs $d \in L^2([0,\infty),\R^k)$. 
With the lemma at hand, we can now prove the following approximation result.

\begin{thm} \label{thm:density and approximation}
Suppose that Conditions~\ref{cond:S imp-passive}, \ref{cond:P und R}, \ref{cond:B-tilde surj} are satisfied. Then for every $(\tilde{x}_0,d) \in \tilde{X} \times L^2_{\mathrm{loc}}([0,\infty),\R^k)$ one has:
\begin{itemize}
\item[(i)] there exists a sequence $(\tilde{x}_{0n},d_n)$ in $\mathcal{D}$ converging to $(\tilde{x}_0,d)$ in the locally convex topology of $\tilde{X} \times L^2_{\mathrm{loc}}([0,\infty),\R^k)$
%$\mathcal{D}$ dense in locally convex space $\tilde{X} \times L^2_{\mathrm{loc}}([0,\infty),\R^k)$
\item[(ii)] for every such approximating 
sequence $(\tilde{x}_{0n},d_n)$ in $\mathcal{D}$, the corresponding sequence 
$( \tilde{x}(\cdot,\tilde{x}_{0n},d_n) )$ 
of classical solutions of~\eqref{eq:ivp closed-loop d ne 0} is a Cauchy sequence in the locally convex space $C([0,\infty),\R^k)$ and its limit is independent of the particular choice of the approximating sequence $(\tilde{x}_{0n},d_n)$. %with the convergence property from (i)
\end{itemize}
\end{thm}

\begin{proof}
(i) Since $D(\tilde{A})$ is dense in $\tilde{X}$ (Lemma~\ref{lm:sgr gen with comp resolv}) and since $C_c^{\infty}((0,\infty),\R^k)$ is dense in $L_{\mathrm{loc}}^2([0,\infty),\R^k)$, we see that
$$D(\tilde{A}) \times C_c^{\infty}((0,\infty),\R^k)$$
is dense in %$\tilde{X} \times L^2([0,\infty),\R^k)$ and in 
$\tilde{X} \times L_{\mathrm{loc}}^2([0,\infty),\R^k)$. Since, moreover, 
\begin{align*}
D(\tilde{A}) \times C_c^{\infty}((0,\infty),\R^k) 
\subset \mathcal{D}
%\subset \big\{  (\tilde{x}_0,d) \in D(\tilde{\mathcal{A}}) \times C_c^2([0,\infty),\R^k): d(0) = \tilde{\mathcal{B}} \tilde{x}_0 \big\}
\subset \tilde{X} \times L_{\mathrm{loc}}^2([0,\infty),\R^k),
\end{align*}
we see that  $\mathcal{D}$ is a fortiori dense in $\tilde{X} \times L_{\mathrm{loc}}^2([0,\infty),\R^k)$.
\smallskip

(ii) As in the proof of Lemma~\ref{lm:S_lin admissible}, choose Lipschitz constants $L(R)$ of $\tilde{f}|_{\ol{B}_R(0)}$ such that $R \mapsto L(R)$ is continuous and monotonically increasing. 
We show that for arbitrary $(\tilde{x}_{01},d_1), (\tilde{x}_{02},d_2) \in \mathcal{D}$ and $\tau \in [0,\infty)$ one has the estimate
\begin{align} \label{eq:approx, estimate difference of class sol}
\norm{ \tilde{x}(\cdot, \tilde{x}_{01},d_1) - \tilde{x}(\cdot, \tilde{x}_{02},d_2) }_{[0,\tau],\infty}
&\le 
\Big( \norm{ \tilde{x}_{01}-\tilde{x}_{02} } + \norm{\ol{\Phi}_{\tau}} \norm{d_1-d_2}_{[0,\tau],2} \Big) \cdot \notag \\
&\qquad \qquad \cdot \exp\big( L\big(R_{\tau}(\tilde{x}_{01},d_1,\tilde{x}_{02},d_2)\big) \tau \Big),
\end{align}
where $R_{\tau}(\tilde{x}_{01},d_1,\tilde{x}_{02},d_2) := R_{\tau}(\tilde{x}_{01},d_1) + R_{\tau}(\tilde{x}_{02},d_2)$ and
\begin{align}
R_{\tau}(\tilde{x}_{0},d) := \ul{\sigma}(\norm{\tilde{x}_0}) + \ul{\gamma}( \norm{d}_{[0,\tau],2} )
\end{align}
for $(\tilde{x}_0,d) \in \tilde{X} \times L^2_{\mathrm{loc}}([0,\infty),\R^k)$ with $\ul{\sigma}, \ul{\gamma}$ as in Theorem~\ref{thm:class solvb}. 
So let $(\tilde{x}_{01},d_1), (\tilde{x}_{02},d_2) \in \mathcal{D}$ and $\tau \in [0,\infty)$. It then follows from~\eqref{eq:voc, kappa=0} using Lemma~\ref{lm:sgr gen with comp resolv}, Theorem~\ref{thm:class solvb} and Lemma~\ref{lm:S_lin admissible} that
\begin{align}
&\norm{ \tilde{x}(t, \tilde{x}_{01},d_1) - \tilde{x}(t, \tilde{x}_{02},d_2) }
\le 
\norm{ \tilde{x}_{01}-\tilde{x}_{02} } + \norm{\ol{\Phi}_{\tau}} \norm{d_1-d_2}_{[0,\tau],2} \notag \\
&\qquad \qquad \qquad + L\big(R_{\tau}(\tilde{x}_{01},d_1,\tilde{x}_{02},d_2)\big) \int_0^t \norm{ \tilde{x}(s, \tilde{x}_{01},d_1) - \tilde{x}(s, \tilde{x}_{02},d_2) } \d s
\end{align}
for every $t \in [0,\tau]$. So, by Gr\"onwall's lemma, the claimed estimate~\eqref{eq:approx, estimate difference of class sol}  follows. 
And from this, in turn, assertion~(ii) immediately follows because
$\big( \tilde{X} \times L^2_{\mathrm{loc}}([0,\infty),\R^k) \big)^2 \ni (\tilde{x}_{01},d_1,\tilde{x}_{02},d_2) \mapsto R_{\tau}(\tilde{x}_{01},d_1,\tilde{x}_{02},d_2)$ is continuous. 
\end{proof}

As a consequence of the above theorem, for every generalized datum $(\tilde{x}_0,d) \in \tilde{X}\times L^2_{\mathrm{loc}}([0,\infty),\R^k)$, 
the assignment %relation 
\begin{align} \label{eq:general sol, def}
\tilde{x}(\cdot,\tilde{x}_0,d) := \lim_{n\to\infty} \tilde{x}(\cdot,\tilde{x}_{0n},d_n) \in C([0,\infty), \tilde{X}) 
\end{align}
with $(\tilde{x}_{0n},d_n)$ being an arbitrary approximating sequence in $\mathcal{D}$ for $(\tilde{x}_0,d)$, yields %gives 
a well-defined function. 
%
%the function 
%\begin{align*}
%\tilde{x}(\cdot,\tilde{x}_0,d) := \lim_{n\to\infty} \tilde{x}(\cdot,\tilde{x}_{0n},d_n) 
%\end{align*}
%with $(\tilde{x}_{0n},d_n)$ being an arbitrary approximating sequence in $\mathcal{D}$ for $(\tilde{x}_0,d)$, is well-defined. 
%
We call this function -- following~\cite{TuWe14} -- the \emph{generalized solution} of the closed-loop system~\eqref{eq:ivp closed-loop d ne 0} corresponding to $(\tilde{x}_0,d)$ 
because it obviously coincides with the classical solution for $(\tilde{x}_0,d) \in \mathcal{D}$ (and with the mild solution for $(\tilde{x}_0,d) \in \tilde{X}\times W^{1,2}_{\mathrm{loc}}([0,\infty),\R^k)$) and because, by the next corollary, it shares many important properties with classical solutions. In particular, the generalized solutions of our closed-loop system satisfy the axioms from~\cite{MiWi16a} (Definition~1). %are/fall within the scope of \cite{MiWi16a}

\begin{cor} \label{cor:general sol, properties}
Suppose that Conditions~\ref{cond:S imp-passive}, \ref{cond:B-tilde surj}, \ref{cond:P und R} are satisfied. Then 
\begin{itemize}
\item[(i)] the generalized solution map $(\tilde{x}_0,d) \mapsto \tilde{x}(\cdot,\tilde{x}_0,d)$ satisfies the cocycle (or flow) property, that is,
\begin{align*}
\tilde{x}(t+s,\tilde{x}_0,d) = \tilde{x}(t, \tilde{x}(s,\tilde{x}_0,d), d(s+\cdot))
\end{align*}
for all $s,t \in [0,\infty)$ and all $(\tilde{x}_0,d) \in \tilde{X}\times L^2_{\mathrm{loc}}([0,\infty),\R^k)$
\item[(ii)] the generalized solution map  $(\tilde{x}_0,d) \mapsto \tilde{x}(\cdot,\tilde{x}_0,d) \in C([0,\infty),\tilde{X})$ is  continuous and causal. 
\end{itemize}
\end{cor}

\begin{proof}
Assertion~(i) and the causality part of assertion~(ii) easily follow by approximation from the cocycle property and the causality of classical solutions. Similarly, the continuity part of assertion~(ii) follows by extending the estimate~\eqref{eq:approx, estimate difference of class sol} to generalized solutions. 
\end{proof}

%With a bit more effort
It can be shown that the closed-loop system, in addition to having a generalized solution, also has a generalized output $y(\cdot,\tilde{x}_0,d) \in L^2_{\mathrm{loc}}([0,\infty),\R^k)$ for every $(\tilde{x}_0,d) \in \tilde{X}\times L^2_{\mathrm{loc}}([0,\infty),\R^k)$ and that the generalized output map
\begin{align*}
\tilde{X}\times L^2_{\mathrm{loc}}([0,\infty),\R^k) \ni (\tilde{x}_0,d) \mapsto y(\cdot,\tilde{x}_0,d) \in L^2_{\mathrm{loc}}([0,\infty),\R^k)
\end{align*}
is continuous in the respective locally convex topology. In particular, the closed-loop system is well-posed in the spirit of~\cite{TuWe14} -- but we will not need this in the sequel. See~\cite{Sc18-wp-semilin-bdry-contr} for proofs and general well-posedness results.

\section{Stability of the closed-loop system} \label{sect:stab}

After having established %settled 
the global solvability of the closed-loop system, we can now move on to stability. A first very simple result is the following uniform global stability theorem. We recall from~\cite{MiWi16a} that the system $\tilde{\mathfrak{S}}$ is called \emph{uniformly globally stable} iff there exist comparison functions $\ul{\sigma}, \ul{\gamma} \in \mathcal{K}$ such that for every $\tilde{x}_0 \in \tilde{X}$ and $d \in L^2([0,\infty),\R^k)$ %one has
\begin{align} \label{eq:UGS def}
\norm{\tilde{x}(t,\tilde{x}_0,d)} \le \ul{\sigma}(\norm{\tilde{x}_0}) + \ul{\gamma}(\norm{d}_2) 
\qquad (t \in [0,\infty)).
\end{align}
%for every $\tilde{x}_0 \in \tilde{X}$ and every $d \in L^2([0,\infty),\R^k)$. 

\begin{thm} \label{thm:UGS}
Suppose  that the assumptions (Conditions~\ref{cond:S imp-passive}, \ref{cond:P und R}, \ref{cond:B-tilde surj}) of the solvability  theorems are satisfied. Then the closed-loop system $\tilde{\mathfrak{S}}$ is uniformly globally stable. 
\end{thm}

\begin{proof}
An immediate consequence of Theorem~\ref{thm:class solvb}~(ii) and~\eqref{eq:general sol, def}. %Theorem~\ref{thm:density and approximation}~(i).
\end{proof}

We are now going to improve this %uniform global stability 
result to a (uniform) input-to-state stability result (Section~\ref{sect:ISS}) and to a weak input-to-state stability result (Section~\ref{sect:wISS}), respectively.

\subsection{Input-to-state stability of the closed-loop system} \label{sect:ISS}

In this section we show that, for systems $\mathfrak{S}$ of order $N=1$ and for a special class of controllers $\mathfrak{S}_c$, the closed-loop system $\tilde{\mathfrak{S}}$ is (uniformly) input-to-state stable. We recall that $\tilde{\mathfrak{S}}$ is \emph{(uniformly) input-to-state stable} w.r.t.~inputs from $L^2([0,\infty),\R^k)$ iff it is uniformly globally stable and of uniform asymptotic gain. See~\cite{MiWi16a} for this and other characterizations of input-to-state stability. In this context, the \emph{uniform asymptotic gain} property by definition means the following: there is a function $\ol{\gamma} \in \mathcal{K} \cup \{0\}$ %a so-called uniform (asymptotic) gain function 
such that for every $\eps, r > 0$ there is a time $\ol{\tau} = \ol{\tau}(\eps,r)$ such that for every $\tilde{x}_0 \in \tilde{X}$ with $\norm{\tilde{x}_0} \le r$ and every $d \in L^2([0,\infty),\R^k)$
\begin{align} \label{eq:UAG def}
\norm{\tilde{x}(t,\tilde{x}_0,d)} \le \eps + \ol{\gamma}(\norm{d}_2)  
\qquad (t \ge \ol{\tau}).
\end{align}
%for all $\tilde{x}_0 \in \ol{B}_r(0)$  and $d \in L^2([0,\infty),\R^m)$. 
A function $\ol{\gamma}$ as above is called a \emph{uniform (asymptotic) gain (function)} for $\tilde{\mathfrak{S}}$. 

\begin{lm} \label{lm:ISS implies conv to 0}
If the assumptions (Conditions~\ref{cond:S imp-passive}, \ref{cond:P und R}, \ref{cond:B-tilde surj}) of the solvability results are satisfied and $\tilde{\mathfrak{S}}$ is uniformly input-to-state stable w.r.t.~inputs from $L^2([0,\infty),\R^k)$, then for every $(\tilde{x}_0,d) \in \tilde{X} \times L^2([0,\infty),\R^k)$ one has
\begin{align} \label{eq:ISS conv to 0}
\tilde{x}(t,\tilde{x}_0,d) \longrightarrow 0 \qquad (t \to \infty). 
\end{align}
\end{lm}

\begin{proof}
Choose $\eps > 0$ and $(\tilde{x}_0,d) \in \tilde{X} \times L^2([0,\infty),\R^k)$. 
Since $\tilde{\mathfrak{S}}$ is uniformly globally stable and of uniform asymptotic gain, %(with a uniform gain function $\ol{\gamma}$), 
there exist $\ul{\sigma}, \ul{\gamma} \in \mathcal{K}$ as in~\eqref{eq:UGS def} and a uniform gain function $\ol{\gamma} \in \mathcal{K} \cup \{0\}$ along with a time  $\ol{\tau}_0 = \ol{\tau}_0(\eps, \ul{\sigma}(\norm{\tilde{x}_0})+ \ul{\gamma}(\norm{d}_2))$ such that
\begin{align} \label{eq:ISS conv to 0, 1}
\big\| \tilde{x}(t,\tilde{\xi}_0,d) \big\| \le \eps + \ol{\gamma}(\norm{d}_2)
\end{align}
for every $\tilde{\xi}_0 \in \tilde{X}$ with $\|\tilde{\xi}_0 \| \le \ul{\sigma}(\norm{\tilde{x}_0})+\ul{\gamma}(\norm{d}_2)$ and every $t \ge \ol{\tau}_0$. 
Since moreover $d \in L^2([0,\infty),\R^k)$, there exists a time $t_0 = t_0(\eps,d)$ such that
\begin{align} \label{eq:ISS conv to 0, 2}
\ol{\gamma}(\norm{d(t_0+\cdot)}_2) < \eps.
\end{align}
Setting now $\ol{\tau} := \ol{\tau}(\eps,\norm{\tilde{x}_0},d) := \ol{\tau}_0+t_0$, we see by the cocycle property and by~\eqref{eq:UGS def}, \eqref{eq:ISS conv to 0, 1}, \eqref{eq:ISS conv to 0, 2} that 
$\norm{ \tilde{x}(t',\tilde{x}_0,d) } < 2\eps$ for all $t' \ge \ol{\tau}$. And so \eqref{eq:ISS conv to 0} ensues.
\end{proof}

In order to achieve input-to-state stability, we add the following conditions on the system $\mathfrak{S}$ and the controller $\mathfrak{S}_c$ to the assumptions from the solvability results. 

\begin{cond} \label{cond:ISS, system} 
$\mathfrak{S}$ is of order $N = 1$ and $\zeta \mapsto \mathcal{H}(\zeta)$ is absolutely continuous. Additionally, there exists a constant $\kappa > 0$ such that for $\eta = a$ or $\eta = b$ one has
\begin{align} \label{eq:ISS, matrungl}
|\mathcal{B}x|^2 + |\mathcal{C}x|^2 \ge \kappa |(\mathcal{H}x)(\eta)|^2 %\qquad (x \in D(\mathcal{A}))
\qquad (x \in D(\mathcal{A})).
\end{align}
%for all $x \in D(\mathcal{A})$.
\end{cond}

\begin{cond} \label{cond:ISS, controller}
\begin{itemize}
\item[(i)] $\mathcal{P}$ is quasi-quadratic in the sense that for some constants $\ul{c}_1, \ol{c}_1 >0$
\begin{align*}
\ol{c}_1 v_1^{\top} \nabla \mathcal{P}(v_1) \ge  \mathcal{P}(v_1) \ge \ul{c}_1 |v_1|^2 \qquad (v_1 \in \R^{m_c})
\end{align*}
\item[(ii)] $\mathcal{R}$ is quasi-linear in the sense that for some constants $\ul{c}_2, \ol{c}_2 >0$
\begin{align*}
\ol{c}_2 v_2^{\top} \mathcal{R}(v_2) \ge   |v_2|^2 \ge \ul{c}_2 |\mathcal{R}(v_2)|^2 \qquad (v_2 \in \R^{m_c}).
\end{align*}
\end{itemize}
\end{cond}

Condition~\ref{cond:ISS, system} can be viewed as an approximate observability condition on our system $\mathfrak{S}$ by virtue of Lemma~\ref{lm:hinr bed fuer klass appr beobb-vor an S aus wISS-satz} below.

\begin{thm} \label{thm:ISS}
Suppose  that the assumptions (Conditions~\ref{cond:S imp-passive}, \ref{cond:P und R}, \ref{cond:B-tilde surj}) of the solvability  theorems are satisfied along with Conditions~\ref{cond:ISS, system} and~\ref{cond:ISS, controller}. %Suppose that the assumptions of the well-posedness theorem (Conditions~\ref{cond:S imp-passive}, \ref{cond:P und R}, \ref{cond:B-tilde surj}) are satisfied and that Condition~\ref{cond:ISS} is satisfied. 
%Suppose further that there exists a constant $\kappa > 0$ such that for $\eta = a$ or $\eta = b$ one has
%\begin{align} \label{eq:ISS, matrungl}
%|\mathcal{B}x|^2 + |\mathcal{C}x|^2 \ge \kappa |(\mathcal{H}x)(\eta)|^2 %\qquad (x \in D(\mathcal{A})).
%\end{align}
%for all $x \in D(\mathcal{A})$. 
Then the closed-loop system $\tilde{\mathfrak{S}}$ is input-to-state stable and the function~$\ol{\gamma}$, given by $\ol{\gamma}(r) = \ul{\psi}^{-1}(2C r^2)$ with arbitrary $C > 1/(4\varsigma)$ and with $\varsigma$, $\ul{\psi}$ as in~\eqref{eq:S_c pos def} and~\eqref{eq:E-tilde equiv to norm}, is a uniform asymptotic gain for $\tilde{\mathfrak{S}}$. %a uniform gain function $\gamma$ for $\tilde{\mathfrak{S}}$ is given by $\gamma(r) = \ul{\psi}^{-1}(2C r^2)$ 
In particular, %for every $\tilde{x}_0 \in \tilde{X}$ and $d \in L^2([0,\infty),\R^k)$ 
\begin{align}
\tilde{x}(t,\tilde{x}_0,d) \longrightarrow 0 \qquad (t \to \infty)
\end{align}
for every $\tilde{x}_0 \in \tilde{X}$ and $d \in L^2([0,\infty),\R^k)$.
\end{thm}

We now turn to the proof of the theorem and we begin by recording some  central ingredients.

\subsubsection{Central ingredients of the proof}

A first important ingredient is the following estimate for the energy along solutions which essentially uses our assumption that $\tilde{\mathfrak{S}}$ be of order $N=1$. It is a perturbed version of the respective sideways energy estimate %final-time observation inequality
from~\cite{RaZwLe17}.

\begin{lm} \label{lm:ISS, E-tilde}
Under the assumptions of the above theorem, %Suppose Conditions~ and~\ref{cond:ISS}~0 ($N =1$ and $\zeta \mapsto \mathcal{H}(\zeta)$ Lipschitz)
there exists a function $c_0 \in \mathcal{L}$ and a $t_0 > 0$ such that for $\eta = a$ and $\eta = b$ one has
\begin{align*}
\tilde{E}(\tilde{x}(t,\tilde{x}_0,d)) \le c_0(t) \Bigg( \int_0^t |(\mathcal{H}x(s))(\eta)|^2 \d s + \int_0^t E_c(v(s)) \d s \Bigg) + \int_0^t |d(s)||y(s)| \d s
\end{align*}
for every $(\tilde{x}_0,d) \in \mathcal{D}$ and for every $t \ge t_0$, where $(x,v) := \tilde{x}(\cdot,\tilde{x}_0,d)$ and $y := y(\cdot,\tilde{x}_0,d)$. 
\end{lm}

\begin{proof}
%Step 1
As a first step, we show that there exist constants $\gamma_0, \kappa_0 \in (0,\infty)$ such that for every $(\tilde{x}_0,d) \in \mathcal{D}$ one has
\begin{align} \label{eq:ISS-E-tilde, step 1}
F_{\tau}^+(\zeta) \le F_{\tau}^+(b) \e^{\kappa_0 (b-a)}
\qquad \text{and} \qquad 
F_{\tau}^-(\zeta) \le F_{\tau}^-(a) \e^{\kappa_0 (b-a)}
\end{align}
for every $\zeta \in (a,b)$ and every $\tau > 2 \gamma_0 (b-a)$, where
\begin{align*}
&F_{\tau}^+(\zeta) := F_{\tau, (\tilde{x}_0,d)}^+(\zeta) 
:= \int_{\gamma_0(b-\zeta)}^{\tau - \gamma_0(b-\zeta)} x(s,\zeta)^{\top} \mathcal{H}(\zeta) x(s,\zeta) \d s \\
&F_{\tau}^-(\zeta) := F_{\tau, (\tilde{x}_0,d)}^-(\zeta) 
:= \int_{\gamma_0(\zeta-a)}^{\tau - \gamma_0(\zeta-a)} x(s,\zeta)^{\top} \mathcal{H}(\zeta) x(s,\zeta) \d s
\end{align*}
and where $x(s,\cdot)$ denotes the continuous representative of $x(s)$ and, as usual, $(x,v) := \tilde{x}(\cdot,\tilde{x}_0,d)$. 
Indeed, set 
\begin{align} \label{eq:ISS-E-tilde, def gamma und kappa}
\gamma_0 := \norm{P_1^{-1}} / \ul{m}
\qquad \text{and} \qquad
\kappa_0 := \big( 2 \norm{P_1^{-1} P_0} \ol{m} + \ol{m}' \big) / \ul{m}
\end{align}
where $\ul{m}, \ol{m}$ are as in~\eqref{eq:H bd below and above} and $\ol{m}' := (\int_a^b \norm{\mathcal{H}'(\theta)} \d \theta )/(b-a)$. Also, choose and fix $(\tilde{x}_0,d) \in \mathcal{D}$. Since $x$ is a classical solution of 
\begin{align*}
x' = \mathcal{A}x,
\end{align*}
%(Theorem~\ref{thm:class solvb}),
and $\mathcal{H}$ is absolutely continuous,  
it follows by Lemma~3.2 of~\cite{Sc18-BV} that $F_{\tau}^{\pm}$ for every $\tau > 2 \gamma_0 (b-a)$ is absolutely continuous and hence differentiable almost everywhere with derivative given by
\begin{align} \label{eq:ISS-E-tilde, abl F_tau}
%\partial_{\zeta} F_{\tau}^{\pm}(\zeta) 
(F_{\tau}^{\pm})'(\zeta)
&= x(s,\zeta)^{\top} \big( \pm \gamma_0 \mathcal{H}(\zeta) + P_1^{-1} \big) x(s,\zeta) \Big|_{s=t^{\pm}(\zeta)} 
+ x(s,\zeta)^{\top} \big( \pm \gamma_0 \mathcal{H}(\zeta) - P_1^{-1} \big) x(s,\zeta) \Big|_{s=r^{\pm}(\zeta)} \notag \\ 
&\qquad - \int_{r^{\pm}(\zeta)}^{t^{\pm}(\zeta)} x(s,\zeta)^{\top} \Big( (P_1^{-1} P_0 \mathcal{H}(\zeta))^{\top} + \mathcal{H}'(\zeta) + P_1^{-1} P_0 \mathcal{H}(\zeta) \Big) x(s,\zeta) \d s
\end{align} 
for a.e.~$\zeta \in (a,b)$, where $r^+(\zeta) := \gamma_0(b-\zeta)$, $t^+(\zeta) := \tau - \gamma_0(b-\zeta)$ and $r^-(\zeta) := \gamma_0(\zeta-a)$, $t^-(\zeta) := \tau - \gamma_0(\zeta-a)$.
%\begin{align*}
%&r^+(\zeta) := \gamma_0(b-\zeta), \qquad t^+(\zeta) := \tau - \gamma_0(b-\zeta) \\
%%\qquad \text{and} \qquad
%&r^-(\zeta) := \gamma_0(\zeta-a), \qquad t^-(\zeta) := \tau - \gamma_0(\zeta-a).
%\end{align*}
In view of~\eqref{eq:ISS-E-tilde, def gamma und kappa} it follows from~\eqref{eq:ISS-E-tilde, abl F_tau} that
\begin{align}
(F_{\tau}^+)'(\zeta) &\ge - \kappa(\zeta) \int_{r^{+}(\zeta)}^{t^{+}(\zeta)} x(s,\zeta)^{\top} \ul{m} \, x(s,\zeta) \d s 
\ge - \kappa(\zeta) F_{\tau}^+(\zeta) 
\label{eq:ISS-E-tilde, diffungl fuer F_tau^+} \\
(F_{\tau}^-)'(\zeta) &\le \kappa(\zeta) \int_{r^{-}(\zeta)}^{t^{-}(\zeta)} x(s,\zeta)^{\top} \ul{m} \, x(s,\zeta) \d s 
\le \kappa(\zeta) F_{\tau}^-(\zeta)
\label{eq:ISS-E-tilde, diffungl fuer F_tau^-}
\end{align}
for all $\tau > 2 \gamma_0(b-a)$ and a.a.~$\zeta \in (a,b)$, where
\begin{align*}
\kappa(\zeta) := \big( 2 \norm{P_1^{-1} P_0} \ol{m} + \norm{\mathcal{H}'(\zeta)} \big) / \ul{m}.
\end{align*}
Since $F_{\tau}^{\pm}$ is absolutely continuous, the differential inequalities~\eqref{eq:ISS-E-tilde, diffungl fuer F_tau^+} and~\eqref{eq:ISS-E-tilde, diffungl fuer F_tau^-} imply that $\zeta \mapsto F_{\tau}^+(\zeta) \exp(-\int_{\zeta}^b \kappa(\theta) \d \theta)$ and $\zeta \mapsto F_{\tau}^-(\zeta) \exp(-\int_a^{\zeta} \kappa(\theta) \d \theta)$ are monotonically increasing or decreasing, respectively. And from this, in turn,~\eqref{eq:ISS-E-tilde, step 1} follows in a straightforward manner.  
\smallskip

%Step 2
As a second step, we show the assertion of the lemma. 
Choose $\gamma_0, \kappa_0$ as in~\eqref{eq:ISS-E-tilde, def gamma und kappa} and let $(\tilde{x}_0,d) \in \mathcal{D}$. %be fixed. 
In view of the first line of~\eqref{eq:class solvb, derivative of energy, estimate} %(or, more precisely, its first~\ref{} lines) 
we see that
\begin{align} \label{eq:ISS-E-tilde, S-tilde passive}
\tilde{E}_{\tilde{x}}(t_2) \le \tilde{E}_{\tilde{x}}(t_1) + \int_{t_1}^{t_2} |d(s)||y(s)| \d s
\end{align}
for all $t_1 \le t_2$. 
It follows from~\eqref{eq:ISS-E-tilde, S-tilde passive} that, for every $\tau > 2 \gamma_0(b-a)$,
\begin{align*}
&\big( \tau - 2\gamma_0(b-a) \big) \tilde{E}_{\tilde{x}}(\tau-\gamma_0(b-a)) 
= \int_{\gamma_0(b-a)}^{\tau-\gamma_0(b-a)} \tilde{E}_{\tilde{x}}(\tau-\gamma_0(b-a)) \d t \\
&\qquad \qquad \le 
\int_{\gamma_0(b-a)}^{\tau-\gamma_0(b-a)} \Bigg( \tilde{E}_{\tilde{x}}(t) + \int_t^{\tau-\gamma_0(b-a)} |d(s)||y(s)| \d s \Bigg) \d t \\
&\qquad \qquad =
\frac{1}{2} \int_a^b \int_{\gamma_0(b-a)}^{\tau-\gamma_0(b-a)} x(t,\zeta)^{\top} \mathcal{H}(\zeta) x(t,\zeta) \d t \d \zeta 
+ \int_{\gamma_0(b-a)}^{\tau-\gamma_0(b-a)} E_c(v(t)) \d t \\
&\qquad \qquad \quad + \int_{\gamma_0(b-a)}^{\tau-\gamma_0(b-a)} \int_t^{\tau-\gamma_0(b-a)} |d(s)||y(s)| \d s \, \d t
\end{align*}
and therefore (increase the intervals of integration!)
\begin{align} \label{eq:ISS-E-tilde, 1}
&\big( \tau - 2\gamma_0(b-a) \big) \tilde{E}_{\tilde{x}}(\tau-\gamma_0(b-a)) 
\le \frac{1}{2} \int_a^b  F_{\tau}^{\pm}(\zeta) \d \zeta + \int_0^{\tau} E_c(v(t)) \d t \notag \\ 
&\qquad \qquad \qquad \qquad + \big( \tau-2\gamma_0(b-a) \big) \int_{\gamma_0(b-a)}^{\tau-\gamma_0(b-a)} |d(s)||y(s)| \d s. 
%\notag \\
%&\le \frac{1}{2} \e^{\kappa_0 (b-a)} (b-a) \min\{ F_{\tau}^+(b), F_{\tau}^-(a) \}
\end{align}
It further follows from~\eqref{eq:ISS-E-tilde, S-tilde passive} that, for every $\tau > 2 \gamma_0(b-a)$,
\begin{align} \label{eq:ISS-E-tilde, 2}
\tilde{E}_{\tilde{x}}(\tau) \le \tilde{E}_{\tilde{x}}(\tau-\gamma_0(b-a))  + \int_{\tau-\gamma_0(b-a)}^{\tau} |d(s)||y(s)| \d s.
\end{align}
Combining~\eqref{eq:ISS-E-tilde, 2} with~\eqref{eq:ISS-E-tilde, 1}  and~\eqref{eq:ISS-E-tilde, step 1},  %using the first step, 
we conclude that
\begin{align} \label{eq:ISS-E-tilde, 3}
\tilde{E}(\tilde{x}(\tau,\tilde{x}_0,d)) 
&\le
\frac{1}{2 \big( \tau - 2\gamma_0(b-a) \big) }  \min\{ F_{\tau}^+(b), F_{\tau}^-(a) \} \e^{\kappa_0(b-a)} (b-a) \notag \\
&\quad + \frac{1}{ \tau - 2\gamma_0(b-a) } \int_0^{\tau} E_c(v(t)) \d t
+ \int_0^{\tau} |d(s)||y(s)| \d s
\end{align}
for every $\tau > 2 \gamma_0(b-a)$. Choosing now $t_0 := 2\gamma_0(b-a) + 1$ and $c_0 \in \mathcal{L}$ such that
\begin{align}
c_0(\tau) = \frac{1}{ \tau - 2\gamma_0(b-a) } \max \big\{ \e^{\kappa_0(b-a)} (b-a)/(2 \ul{m}), 1 \big\} 
\end{align}
for all $\tau \ge t_0$, we obtain the assertion of the lemma from~\eqref{eq:ISS-E-tilde, 3}. 
\end{proof}

A second important ingredient is the following estimate for the integrated controller energy which shows up %on the right-hand side of
in the previous lemma. %Lemma~\ref{lm:ISS, E-tilde}. 

\begin{lm} \label{lm:ISS, E_c}
Under the assumptions of the above theorem,
there exists a constant $C_0 > 0$ such that for every $(\tilde{x}_0,d) \in \mathcal{D}$
\begin{align*}
\int_0^t E_c(v(s)) \d s \le C_0 \bigg( \tilde{E}(\tilde{x}_0) + \int_0^t |y(s)|^2 \d s \bigg)
\end{align*}
for all $t \in [0,\infty)$, where $(x,v) := \tilde{x}(\cdot,\tilde{x}_0,d)$ and $y := y(\cdot,\tilde{x}_0,d)$. 
\end{lm}

\begin{proof}
We can argue as in~\cite{RaZwLe17}, but for the reader's convenience we repeat the arguments here in a streamlined fashion. 
As a first step, we show that $E_c$ is equivalent to the auxiliary functions $V_{\gamma}$ defined by
\begin{align*}
V_{\gamma}(v) := E_c(v) + \gamma \, v_1^{\top} v_2 \qquad (v = (v_1,v_2) \in \R^{2m_c})
\end{align*}
for all sufficiently small $\gamma >0$. %for all $\gamma \in (0,\gamma_0]$ with a sufficiently small $\gamma_0$. 
In fact, we show that there is a $\gamma_0 > 0$ such that for all $\gamma \in (0,\gamma_0]$ and $v \in \R^{2m_c}$ one has
\begin{align} \label{eq:ISS, E_c, step 1}
\frac{1}{2} V_{\gamma}(v) \le E_c(v) \le 2 V_{\gamma}(v).
\end{align}
In order to see this, just observe that by Condition~\ref{cond:ISS, controller} one has for all $v \in \R^{2m_c}$
\begin{align}
|v_1^{\top} v_2| \le \frac{|v_1|^2}{2} + \frac{|v_2|^2}{2} \le \ul{c}_1^{-1} \mathcal{P}(v_1) + \norm{K^{-1}} \frac{v_2^{\top}Kv_2}{2}
\le C_1 E_c(v)
\end{align}
for $C_1 := \max \{ \ul{c}_1^{-1}, \norm{K^{-1}} \}$. So, with $\gamma_0 := 1/(2C_1)$ the equivalence~\eqref{eq:ISS, E_c, step 1} follows. 
\smallskip

As a second step, we show that there exists a constant $C_0 > 0$ and a $\gamma \in (0,\gamma_0]$ such that for every $(\tilde{x}_0,d) \in \mathcal{D}$
\begin{align} \label{eq:ISS, E_c, step 2}
\int_0^t V_{\gamma}(v(s)) \d s \le C_0 \bigg( V_{\gamma}(v_0) + \int_0^t |y(s)|^2 \d s \bigg)
\end{align}
for all $t \in [0,\infty)$, where $(x,v) := \tilde{x}(\cdot,\tilde{x}_0,d)$ and $y := y(\cdot,\tilde{x}_0,d)$. In conjunction with the first step, this proves the lemma. 
%
%Schritt 2.1: Abschaetzung der Ableitung von V_gamma
So let $(\tilde{x}_0,d) \in \mathcal{D}$, then $s \mapsto V_{\gamma \, v}(s) := V_{\gamma}(v(s))$ is continuously differentiable and its derivative satisfies
\begin{align}
V_{\gamma \, v}'(s) 
&= -(Kv_2(s))^{\top}\mathcal{R}(Kv_2(s)) + (Kv_2(s))^{\top} B_c y(s) 
+ \gamma \, v_2(s)^{\top}Kv_2(s) \notag \\
&\quad - \gamma \, v_1(s)^{\top} \nabla \mathcal{P}(v_1(s)) - \gamma \, v_1(s)^{\top}\mathcal{R}(Kv_2(s)) + \gamma \, v_1(s)^{\top} B_c y(s).
\end{align}
for all $s \in [0,\infty)$. With the help of Condition~\ref{cond:ISS, controller} it thus follows that 
\begin{align} \label{eq:ISS, E_c, 2.1}
V_{\gamma \, v}'(s) 
&\le 
-\ol{c}_2^{-1} |Kv_2(s)|^2 + \frac{1}{2\alpha} |Kv_2(s)|^2 + \frac{\alpha}{2} \norm{B_c}^2 |y(s)|^2
+ \gamma \, v_2(s)^{\top}Kv_2(s) \notag \\
&\quad - \gamma \, \ol{c}_1^{-1} \mathcal{P}(v_1(s)) + \frac{\gamma}{2\alpha} |v_1(s)|^2 + \frac{\gamma \alpha}{2} |\mathcal{R}(Kv_2(s)|^2 + \frac{\gamma}{2 \alpha} |v_1(s)|^2 + \frac{\gamma \alpha}{2} \norm{B_c}^2 |y(s)|^2 \notag \\
&\le 
\gamma \Big( -\ol{c}_1^{-1} + \frac{\ul{c}_1^{-1}}{\alpha} \Big) \mathcal{P}(v_1(s)) 
+ 2 \Big( -\ol{c}_2^{-1} \norm{K} + \frac{\norm{K}}{\alpha} + \gamma \big( \frac{\alpha}{2} \ul{c}_2^{-1} \norm{K} + 1 \big) \Big) \cdot \notag \\
&\qquad \qquad \cdot \frac{v_2(s)^{\top}Kv_2(s)}{2} 
+ \frac{\alpha}{2} (1+\gamma) \norm{B_c}^2 |y(s)|^2
%
%\gamma \big( -\ol{c}_1^{-1} + \ul{c}_1^{-1}/\alpha \big) \mathcal{P}(v_1(s)) 
%+ 2 \Big( -\ol{c}_2^{-1} \norm{K} + \norm{K}/\alpha + \gamma \big( \alpha \ul{c}_2^{-1}/2 + 1 \big) \Big) \frac{v_2(s)^{\top}Kv_2(s)}{2} 
%+ \frac{\alpha}{2} (1+\gamma) \norm{B_c}^2 |y(s)|^2
\end{align}
for all $s \in [0,\infty)$ and arbitrary $\alpha, \gamma \in (0,\infty)$. We choose now $\alpha \in (0,\infty)$ so large that
\begin{align} \label{eq:ISS, E_c, alpha}
-\ol{c}_1^{-1} + \frac{\ul{c}_1^{-1}}{\alpha} < 0 
\qquad \text{and} \qquad 
-\ol{c}_2^{-1} \norm{K} + \frac{\norm{K}}{\alpha} < 0
\end{align}
and then we choose $\gamma \in (0,\gamma_0]$ so small that
\begin{align} \label{eq:ISS, E_c, gamma}
\gamma \big( \frac{\alpha}{2} \ul{c}_2^{-1} \norm{K} + 1 \big) < \ol{c}_2^{-1} \norm{K} - \frac{\norm{K}}{\alpha}.
%-\ol{c}_2^{-1} \norm{K} + \frac{\norm{K}}{\alpha} + \gamma \big( \frac{\alpha}{2} \ul{c}_2^{-1} \norm{K} + 1 \big) < 0.
\end{align}
It then follows from~\eqref{eq:ISS, E_c, 2.1} and~\eqref{eq:ISS, E_c, step 1} that the differential inequality 
\begin{align} \label{eq:ISS, E_c, 2.2}
V_{\gamma \, v}'(s)  \le -2\kappa_0 E_c(v(s)) + C_2 |y(s)|^2
\le -\kappa_0 V_{\gamma}(v(s)) + C_2 |y(s)|^2
\end{align}
holds true for all $s \in [0,\infty)$, where 
\begin{align*}
\kappa_0 := -\frac{1}{2} \max \bigg\{ \gamma \Big( -\ol{c}_1^{-1} + \frac{\ul{c}_1^{-1}}{\alpha} \Big), 2 \Big( -\ol{c}_2^{-1} \norm{K} + \frac{\norm{K}}{\alpha} + \gamma \big( \frac{\alpha}{2} \ul{c}_2^{-1} \norm{K} + 1 \big) \Big) \bigg\}
\end{align*}
and $C_2 := \frac{\alpha}{2} (1+\gamma) \norm{B_c}^2$ and where $\kappa_0 > 0$ by virtue of~\eqref{eq:ISS, E_c, alpha} and~\eqref{eq:ISS, E_c, gamma}. 
%
%Schritt 2.2: Integration der oben gezeigten Abschaetzung fuer V_gamma
So, 
\begin{align}
\e^{\kappa_0 t} V_{\gamma}(v(t)) 
%= V_{\gamma}(v_0) + \int_0^t \dds \big( \e^{\kappa_0 s} V_{\gamma}(v(s)) \big) \d s
\le V_{\gamma}(v_0) + C_2 \int_0^t \e^{\kappa_0 s} |y(s)|^2 \d s
\end{align}
and therefore (multiply by $\e^{-\kappa_0 t}$ and integrate!)
\begin{align}
\int_0^{\tau} V_{\gamma}(v(t)) \d t 
&\le 
\bigg( \int_0^{\tau} \e^{-\kappa_0 t} \d t \bigg) V_{\gamma}(v_0) + C_2 \int_0^{\tau} \bigg( \int_s^{\tau} \e^{-\kappa_0 (t-s)} \d t \bigg) |y(s)|^2 \d s \notag \\
&\le 
\frac{1}{\kappa_0} V_{\gamma}(v_0) + \frac{C_2}{\kappa_0} \int_0^{\tau} |y(s)|^2 \d s
\end{align}
for all $\tau \in [0,\infty)$, which proves the claimed estimate~\eqref{eq:ISS, E_c, step 2}. %which concludes the proof of the second step.
\end{proof}

\subsubsection{Conclusion of the proof}

With the above lemmas at hand, we can now conclude the proof of our  input-to-state stability result in two steps. 
Since we already know that the closed-loop system is uniformly globally stable (Theorem~\ref{thm:UGS}), we have only to show that it is of uniform asymptotic gain $\ol{\gamma} := \ul{\psi}^{-1}(2C \, \cdot^2)$ for every $C > 1/(4\varsigma)$.
\smallskip

%Step 1
As a first step, we show that for every $C > 1/(4\varsigma)$ there exist a constant $\beta \in (0,1)$ and a time $\tau \in (0,\infty)$ such that for every $(\tilde{x}_0,d) \in \mathcal{D}$
\begin{align} \label{eq:ISS, step 1}
\tilde{E}(\tilde{x}(\tau,\tilde{x}_0,d)) \le \beta \tilde{E}(\tilde{x}_0) + C \norm{d}_{[0,\tau],2}^2.
\end{align}
So let $C > 1/(4\varsigma)$ and let $\eta$ be the endpoint of $(a,b)$ for which~\eqref{eq:ISS, matrungl} is satisfied. Also, let $(\tilde{x}_0,d) \in \mathcal{D}$ and, as usual, write $(x,v) := \tilde{x}(\cdot,\tilde{x}_0,d)$ and $u := \mathcal{B}x$, $y := \mathcal{C}x$. 
We know from the first line of~\eqref{eq:class solvb, derivative of energy, estimate} that
\begin{align} \label{eq:ISS, step 1, 1}
\tilde{E}_{\tilde{x}}(t) \le \tilde{E}(\tilde{x}_0) &+ \int_0^t |d(s)||y(s)| \d s - \ol{c}_2^{-1} \int_0^t |Kv_2(s)|^2 \d s - \varsigma (1-\eps) \int_0^t |y(s)|^2 \d s \notag \\ 
&- \varsigma \eps \int_0^t |y(s)|^2 \d s
\end{align}
for all $t \in [0,\infty)$ and $\eps \in (0,1)$, where Condition~\ref{cond:ISS, controller} and~\eqref{eq:S_c pos def} have been used. 
Also, from~\eqref{eq:ISS, matrungl} it follows that 
\begin{align} \label{eq:ISS, step 1, 2}
-|y(s)|^2 &\le -\kappa |(\mathcal{H}x(s))(\eta)|^2 + |u(s)|^2 = -\kappa |(\mathcal{H}x(s))(\eta)|^2 + |d(s)-B_c^{\top}Kv_2(s)-y(s)|^2 \notag \\
&\le -\kappa |(\mathcal{H}x(s))(\eta)|^2 + 4 |d(s)|^2 + 4 \norm{B_c}^2 |Kv_2(s)|^2 + 4 \norm{S_c}^2 |y(s)|^2
\end{align}
for all $s \in [0,\infty)$. And from Lemma~\ref{lm:ISS, E-tilde} and~\ref{lm:ISS, E_c} it follows that 
\begin{align} \label{eq:ISS, step 1, 3}
- \int_0^t |(\mathcal{H}x(s))(\eta)|^2 \d s 
\le 
- \frac{1}{c_0(t)}\tilde{E}_{\tilde{x}}(t) &+ C_0 \tilde{E}(\tilde{x}_0) + C_0 \int_0^t |y(s)|^2 \d s  \notag \\
&+ \frac{1}{c_0(t)} \int_0^t |d(s)||y(s)| \d s
\end{align}
for all $t \in [t_0,\infty)$, where $C_0, t_0, c_0$ are the constants and the function from the above lemmas.
Inserting now~\eqref{eq:ISS, step 1, 2} and~\eqref{eq:ISS, step 1, 3} in the last integral on the right-hand side of~\eqref{eq:ISS, step 1, 1}, we obtain
\begin{align*}
&\tilde{E}_{\tilde{x}}(t) \le \tilde{E}(\tilde{x}_0) + \int_0^t |d(s)||y(s)| \d s - \ol{c}_2^{-1} \int_0^t |Kv_2(s)|^2 \d s - \varsigma (1-\eps) \int_0^t |y(s)|^2 \d s \notag \\
&\qquad \qquad - \frac{\varsigma \eps \kappa}{c_0(t)} \tilde{E}_{\tilde{x}}(t) + \varsigma \eps \kappa C_0 \tilde{E}(\tilde{x}_0) + \varsigma \eps \kappa C_0 \int_0^t |y(s)|^2 \d s + \frac{\varsigma \eps \kappa}{c_0(t)} \int_0^t |d(s)||y(s)| \d s \notag \\
&\qquad \qquad + 4 \varsigma \eps \int_0^t |d(s)|^2 \d s + 4 \varsigma \eps \norm{B_c}^2 \int_0^t |Kv_2(s)|^2 \d s + 4 \varsigma \eps \norm{S_c}^2 \int_0^t |y(s)|^2 \d s
\end{align*}
for all $t \in [t_0,\infty)$. So, by regrouping terms, we see that
\begin{align} \label{eq:ISS, step 1, 4}
\bigg( 1 + \frac{\varsigma \eps \kappa}{c_0(t)} \bigg) \tilde{E}_{\tilde{x}}(t) 
&\le 
\big( 1+ \varsigma \eps \kappa C_0 \big) \tilde{E}(\tilde{x}_0) 
+ \big( 4 \varsigma \eps \norm{B_c}^2 - \ol{c}_2^{-1} \big) \int_0^t |Kv_2(s)|^2 \d s \notag \\
&\quad + \bigg(  \bigg( 1 + \frac{\varsigma \eps \kappa}{c_0(t)} \bigg) \frac{1}{2\alpha} + \varsigma \eps \kappa C_0 + 4 \varsigma \eps \norm{S_c}^2 - \varsigma (1-\eps) \bigg) \int_0^t |y(s)|^2 \d s \notag \\
&\quad + \bigg(  \bigg( 1 + \frac{\varsigma \eps \kappa}{c_0(t)} \bigg) \frac{\alpha}{2} + 4 \varsigma \eps \bigg) \int_0^t |d(s)|^2 \d s
\end{align} 
for all $t \in [t_0,\infty)$ and arbitrary $\alpha \in (0,\infty)$. In particular, this holds true for $\alpha$ chosen such that
\begin{align} \label{eq:ISS, step 1, alpha}
\frac{1}{2\varsigma} < \alpha < 2C
%\alpha \in \bigg( \frac{1}{2\varsigma}, 2C \bigg)
\end{align}
(which choice is possible because $C > 1/(4\varsigma)$). We choose now $\tau \in [t_0,\infty)$ so large that
\begin{align} \label{eq:ISS, step 1, tau}
\frac{1}{c_0(\tau)} > C_0
\end{align}
which is possible because $c_0 \in \mathcal{L}$. We also choose $\eps \in (0,1)$ so small that 
\begin{equation} \label{eq:ISS, step 1, eps}
\begin{gathered}
4 \varsigma \eps \norm{B_c}^2 \le \ol{c}_2^{-1}
\qquad \text{and} \qquad
\frac{\alpha}{2} + \frac{4 \varsigma \eps}{1 + \frac{\varsigma \eps \kappa}{c_0(\tau)}} \le C 
\\
\frac{1}{2\alpha} + \varsigma \eps \bigg( \frac{\kappa}{c_0(\tau)} \frac{1}{2\alpha} + \kappa C_0 + 4 \norm{S_c}^2 + 1 \bigg) \le \varsigma
\end{gathered}
\end{equation}
which last two choices are possible because of~\eqref{eq:ISS, step 1, alpha}. Combining now~\eqref{eq:ISS, step 1, 4} with~\eqref{eq:ISS, step 1, eps}, we conclude that 
\begin{align}
\tilde{E}(\tilde{x}(\tau,\tilde{x}_0,d)) \le \beta \tilde{E}(\tilde{x}_0) + C \norm{d}_{[0,\tau],2}^2
\end{align}
where by virtue of~\eqref{eq:ISS, step 1, tau} the constant $\beta$ is smaller than $1$, as desired:
\begin{align}
\beta := \frac{1+ \varsigma \eps \kappa C_0}{1 + \frac{\varsigma \eps \kappa}{c_0(\tau)}} < 1.
\end{align}

%Step 2
As a second step, we show that for every $C > 1/(4\varsigma)$ and for every constant $\beta$ and time $\tau$ as in the first step one has
\begin{align} \label{eq:ISS, step 2}
\tilde{E}(\tilde{x}(t,\tilde{x}_0,d)) \le \frac{1}{\beta} \beta^{t/\tau} \tilde{E}(\tilde{x}_0) + C \norm{d}_{[0,t],2}^2
\end{align}
for every $t \in [0,\infty)$ and every $(\tilde{x}_0,d) \in \mathcal{D}$. 
So let $C > 1/(4\varsigma)$ and let $\beta$ and $\tau$ be as in the first step. It then follows from~\eqref{eq:ISS, step 1} by the cocycle property %of classical solutions 
and induction that
\begin{align} \label{eq:ISS, step 2, 1}
\tilde{E}(\tilde{x}(n \tau,\tilde{x}_0,d)) \le \beta^n \tilde{E}(\tilde{x}_0) + C \norm{d}_{[0,n \tau],2}^2
\end{align}
for every $n \in \N_0$ and every $(\tilde{x}_0,d) \in \mathcal{D}$. If now $t \in [0,\infty)$ is arbitrary, then 
\begin{align*}
t = n \tau + t-n \tau 
\qquad \text{and} \qquad
t-n\tau \in [0,\tau)
\end{align*}
for a unique $n = n_t \in \N_0$ and thus it follows by the cocycle property and by~\eqref{eq:ISS, step 2, 1} and~\eqref{eq:UGS, energy, class} that
\begin{align}
\tilde{E}(\tilde{x}(t,\tilde{x}_0,d)) 
%&= \tilde{E}\Big( \tilde{x}\big(n\tau,\tilde{x}(t-n\tau,\tilde{x}_0,d),d(\cdot+t-n \tau)\big) \Big) \notag \\
&\le \beta^n \tilde{E}\big( \tilde{x}(t-n\tau,\tilde{x}_0,d) \big) + C \norm{d(\cdot+t-n \tau)}_{[0,n\tau],2}^2 
\notag \\
&\le \beta^n \tilde{E}(\tilde{x}_0) + \frac{1}{4\varsigma} \norm{d}_{[0,t-n\tau],2}^2 + C \norm{d}_{[t-n\tau,t],2}^2
\end{align}
every $(\tilde{x}_0,d) \in \mathcal{D}$, from which~\eqref{eq:ISS, step 2} and hence the second step follows. 
%
%Step 3
We conclude by density and continuity that for every $C > 1/(4\varsigma)$ there is a constant $\beta \in (0,1)$ and a time $\tau \in (0,\infty)$ such that~\eqref{eq:ISS, step 2} also holds for arbitrary data $(\tilde{x}_0,d) \in \tilde{X} \times L^2([0,\infty),\R^k)$ and $t \in [0,\infty)$. So, by~\eqref{eq:E-tilde equiv to norm}
\begin{align}
\norm{\tilde{x}(t,\tilde{x}_0,d)} \le \ul{\psi}^{-1}\Big( \frac{2}{\beta} \beta^{t/\tau} \ol{\psi}(\norm{\tilde{x}_0}) \bigg) + \ul{\psi}^{-1}\big( 2 C \norm{d}_{[0,t],2}^2)
\end{align} 
for every $t \in [0,\infty)$ and every $(\tilde{x}_0,d) \in \tilde{X} \times L^2([0,\infty),\R^k)$. In particular, $\tilde{\mathfrak{S}}$ is of uniform asymptotic gain $\ol{\gamma} := \ul{\psi}^{-1}(2C \, \cdot^2)$ for every $C > 1/(4\varsigma)$, as desired. %which concludes the proof of the theorem. %and the proof of the theorem is complete. 

\subsection{Weak input-to-state stability of the closed-loop system} \label{sect:wISS}

In this section we show that for systems $\mathfrak{S}$ of arbitrary order $N \in \N$ and for a general class of controllers $\mathfrak{S}_c$, the closed-loop system $\tilde{\mathfrak{S}}$ is weakly input-to-state stable. We call $\tilde{\mathfrak{S}}$ \emph{weakly input-to-state stable} w.r.t.~inputs from $L^2([0,\infty),\R^k)$ iff it is uniformly globally stable and of weak asymptotic gain. See~\cite{Sc18-GAMM} for other characterizations of weak input-to-state stability. In this context, the \emph{weak asymptotic gain} property by definition means the following: there is a function $\ol{\gamma} \in \mathcal{K} \cup \{0\}$ %a so-called weak (asymptotic) gain function 
such that for every $\eps > 0$ and every $\tilde{x}_0 \in \tilde{X}$ and $d \in L^2([0,\infty),\R^k)$ there is a time $\ol{\tau} = \ol{\tau}(\eps,\tilde{x}_0,d)$ such that 
\begin{align} \label{eq:wAG def}
\norm{\tilde{x}(t,\tilde{x}_0,d)} \le \eps + \gamma(\norm{d}_2)  
\qquad (t \ge \ol{\tau}).
\end{align}
A function $\gamma$ as above is called a \emph{weak (asymptotic) gain (function)} for $\tilde{\mathfrak{S}}$. 
We observe that the only difference to the uniform asymptotic gain property is that the time $\tau$ is allowed to depend on the initial state $\tilde{x}_0$ (instead of only on its norm) and on the disturbance $d$.  In~\cite{MiWi16a} the weak asymptotic gain property is called just asymptotic gain.
%for all $\tilde{x}_0 \in \ol{B}_r(0)$  and $d \in L^2([0,\infty),\R^m)$. 

\begin{lm} \label{lm:wISS implies conv to 0}
If the assumptions (Conditions~\ref{cond:S imp-passive}, \ref{cond:P und R}, \ref{cond:B-tilde surj}) of the solvability results are satisfied and $\tilde{\mathfrak{S}}$ is weakly input-to-state stable w.r.t.~inputs from $L^2([0,\infty),\R^k)$, then for every $(\tilde{x}_0,d) \in \tilde{X} \times L^2([0,\infty),\R^k)$ one has
\begin{align} \label{eq:wISS conv to 0}
\tilde{x}(t,\tilde{x}_0,d) \longrightarrow 0 \qquad (t \to \infty). 
\end{align}
\end{lm}

\begin{proof}
Completely analogous to the proof of Lemma~\ref{lm:ISS implies conv to 0}, the only difference being that there the convergence~\eqref{eq:ISS conv to 0} was even locally uniform w.r.t.~$\tilde{x}_0$. 
\end{proof}

In order to achieve weak input-to-state stability, we add the following conditions on the system $\mathfrak{S}$ and the controller $\mathfrak{S}_c$ to the assumptions from the solvability results.

\begin{cond} \label{cond:wISS, system}
$\mathfrak{S}$ is classically approximately observable in infinite time, that is, for non-zero initial state $0 \ne x_0 \in D(\mathcal{A}) \cap \ker \mathcal{B} = D(A)$ and zero input $u = 0$, the corresponding classical output $y(\cdot,x_0,0) := \mathcal{C}x(\cdot,x_0,0)$ of $\mathfrak{S}$ cannot vanish identically %be identically zero 
on $[0,\infty)$.
\end{cond}

\begin{cond} \label{cond:wISS, controller}
\begin{itemize}
\item[(i)] $\mathcal{R}$ is strictly damping, that is, for some constants $\ul{c}, \ol{c}, \delta > 0$, %in the sense that for some constants $\ul{c}, \ol{c}, \delta > 0$, 
\begin{gather*}
v_2^{\top} \mathcal{R}(v_2) \ge \ul{c} |v_2|^2 \qquad (|v_2| \le \delta) 
\qquad \text{and} \qquad
v_2^{\top} \mathcal{R}(v_2) \ge \ol{c} \qquad (|v_2| > \delta)
\end{gather*}
\item[(ii)] $B_c$, the input operator of the controller, is injective.
\end{itemize}
\end{cond}

\begin{thm}  \label{thm:wISS}
Suppose  that the assumptions (Conditions~\ref{cond:S imp-passive}, \ref{cond:P und R}, \ref{cond:B-tilde surj}) of the solvability  theorems are satisfied along with Conditions~\ref{cond:wISS, system} and~\ref{cond:wISS, controller}. %and that Condition~\ref{cond:ISS} is satisfied. 
Suppose further that the only equilibrium point $\tilde{x}_*$ of the undisturbed system
\begin{align} \label{eq:closed-loop mit zero-input} 
\tilde{x}' = \mathcal{\tilde{A}} \tilde{x} + \tilde{f}(\tilde{x}) \qquad \text{with} \qquad 0 = \mathcal{\tilde{B}}\tilde{x}(t)
\end{align} 
is the point $\tilde{x}_* = 0$. 
%$\mathfrak{S}(\tilde{A}, \tilde{B}, \tilde{B}^*)$ is approximately controllable in infinite time or approximately observable in infinite time.
Then the closed-loop system $\tilde{\mathfrak{S}}$ is weakly input-to-state stable and the function $\ol{\gamma} = 0$ is a weak asymptotic gain for $\tilde{\mathfrak{S}}$. In particular, %for every $\tilde{x}_0 \in \tilde{X}$ and $d \in L^2([0,\infty),\R^k)$ %the corresponding generalized solution $\tilde{x}(\cdot,\tilde{x}_0)$ of~\eqref{eq:ivp closed-loop d ne 0, 2} converges to $0$:
\begin{align}
\tilde{x}(t,\tilde{x}_0,d) \longrightarrow 0 \qquad (t \to \infty)
\end{align}
for every $\tilde{x}_0 \in \tilde{X}$ and $d \in L^2([0,\infty),\R^k)$.
\end{thm}

We now turn to the proof of the theorem and, for that purpose, it will be most helpful to write the nonlinear part of~\eqref{eq:ivp closed-loop d ne 0} as 
$$\tilde{f}(\tilde{x}) = \tilde{B} g(\tilde{B}^*\tilde{x}) + \tilde{B} h(\tilde{C}\tilde{x}),$$ 
where $\tilde{B}: \R^{m_c} \to \tilde{X}$, $\tilde{C}: \tilde{X} \to \R^{m_c}$ and $g, h: \R^{m_c} \to \R^{m_c}$ are such that 
\begin{gather*}
\tilde{B}v_2 := (0,0,v_2), 
\qquad 
%\quad \tilde{B}^* \tilde{x} = Kv_2, \quad 
\tilde{C}\tilde{x} := v_1,\\
g(v_2) := -\mathcal{R}(v_2) \quad \text{and} \quad h(v_1) := v_1-\nabla \mathcal{P}(v_1).
\end{gather*}
In particular, $\tilde{B}^* \tilde{x} = Kv_2$ for every $\tilde{x} = (x,v_1,v_2) \in \tilde{X}$. 
We begin by rewriting the closed-loop equation~\eqref{eq:ivp closed-loop d ne 0} as %($\lambda > 0$)
\begin{equation} \label{eq:ivp closed-loop, rewritten}
\begin{gathered}
\tilde{x}' = \big( \tilde{\mathcal{A}} - \lambda \tilde{B} \tilde{B}^* \big) \tilde{x} + \tilde{B} \big( \lambda \tilde{B}^* \tilde{x} + g(\tilde{B}^*\tilde{x}) \big) + \tilde{B} h(\tilde{C} \tilde{x}) \qquad \text{and} \qquad \tilde{x}(0) = \tilde{x}_0 \\
d(t) = \tilde{\mathcal{B}}\tilde{x}(t),
\end{gathered}
\end{equation}
that is, as a  perturbation of the respective linear boundary control  system %with $\lambda > 0$ %$x' = \big( \tilde{\mathcal{A}} - \lambda \tilde{B} \tilde{B}^* \big) \tilde{x}$.
\begin{gather}
\tilde{x}' = \big( \tilde{\mathcal{A}} - \lambda \tilde{B} \tilde{B}^* \big) \tilde{x} 
\quad \text{with} \quad 
d(t) = \tilde{\mathcal{B}}\tilde{x}(t),
\end{gather} 
where $\lambda > 0$. 
It follows from~\eqref{eq:ivp closed-loop, rewritten} by the transformation~\eqref{eq:Fattorini transformation} and variation of constants that classical solutions of~\eqref{eq:ivp closed-loop d ne 0} satisfy the following integral equation:
\begin{align} \label{eq:voc, kappa ne 0}
\tilde{x}(t,\tilde{x}_0,d) = \e^{(\tilde{A}-\lambda \tilde{B}\tilde{B}^*)t} \tilde{x}_0 
& + \int_0^t \e^{(\tilde{A}-\lambda \tilde{B}\tilde{B}^*)(t-s)} \tilde{B} \big( \lambda \tilde{B}^* \tilde{x}(s,\tilde{x}_0,d) + g(\tilde{B}^* \tilde{x}(s,\tilde{x}_0,d) ) \big) \d s \notag \\
& +  \int_0^t \e^{(\tilde{A}-\lambda \tilde{B}\tilde{B}^*)(t-s)} \tilde{B} h(\tilde{C} \tilde{x}(s,\tilde{x}_0,d) ) \d s
+ \Phi_t^{\lambda}(d)  
\end{align} 
for every $(\tilde{x}_0 ,d) \in \mathcal{D}$, where
\begin{align}  \label{eq:Phi_t^kappa, def}
\Phi_t^{\lambda}(d) := -\e^{(\tilde{A}-\lambda \tilde{B}\tilde{B}^*)t} \tilde{R}d(0) &+ \tilde{R} d(t)  \\
&+ \int_0^t  \e^{(\tilde{A}-\lambda \tilde{B}\tilde{B}^*)(t-s)} \big( (\tilde{\mathcal{A}}-\lambda \tilde{B}\tilde{B}^*) \tilde{R}d(s) - \tilde{R}d'(s) \big) \d s.  \notag
\end{align}

\subsubsection{Central ingredients of the proof}

%We now turn to/record the central ingredients needed to proceed from~\eqref{eq:voc, kappa ne 0}
A first important ingredient to proceed from~\eqref{eq:voc, kappa ne 0} is the following approximate observability result for the collocated linear system $\mathfrak{S}(\tilde{A}, \tilde{B}, \tilde{B}^*)$.

\begin{lm} \label{lm:S(A,B,B*) appr beobb}
Under the assumptions of the above theorem, the linear system $\mathfrak{S}(\tilde{A}, \tilde{B}, \tilde{B}^*)$ is approximately observable in infinite time.
\end{lm}

\begin{proof}
Suppose that for some $\tilde{x}_0 \in \tilde{X}$ we have
\begin{align} \label{eq:S(A,B,B*) appr beobb, 1}
\tilde{B}^* \e^{\tilde{A}t} \tilde{x}_0 = 0 
\qquad (t \in [0,\infty)).
\end{align}
We then have to show that $\tilde{x}_0 = 0$. 
Setting 
\begin{align*}
(x_{n 0}, v_{1 n 0}, v_{2 n 0}) := \tilde{x}_{n 0} := n \int_0^{1/n} \e^{\tilde{A}s} \tilde{x}_0 \d s 
\end{align*}
for $n \in \N$, we have of course that
\begin{align} \label{eq:S(A,B,B*) appr beobb, 2}
\tilde{x}_{n 0} \in D(\tilde{A}) \qquad (n \in \N) \qquad \text{and} \qquad \tilde{x}_{n 0}  \longrightarrow \tilde{x}_0 \qquad (n \to \infty)
\end{align}
and in view of~\eqref{eq:S(A,B,B*) appr beobb, 1} we also have that
\begin{align} \label{eq:S(A,B,B*) appr beobb, 3}
v_{2 n}(t) = \tilde{B}^* \e^{\tilde{A}t} \tilde{x}_{n 0} = 0 
\qquad (n \in \N, t \in [0,\infty)),
\end{align}
where we used the notation $(x_n(t),v_{1 n}(t), v_{2 n}(t)) := \tilde{x}_n(t) := \e^{\tilde{A}t} \tilde{x}_{n 0}$.
%\begin{align*}
%(x_n(t),v_{1 n}(t), v_{2 n}(t)) := \tilde{x}_n(t) := \e^{\tilde{A}t} \tilde{x}_{n 0}. 
%\end{align*}
We will also use the abbreviations %$(x_{n 0}, v_{1 n 0}, v_{2 n 0}) := \tilde{x}_{n 0}$
\begin{align*}
u_n(t) := \mathcal{B}x_n(t) \qquad \text{and} \qquad y_n(t) := \mathcal{C}x_n(t)
\end{align*}
in the following. 
Choose and fix now $n \in \N$. 
We know from~(\ref{eq:S(A,B,B*) appr beobb, 2}.a) that $\tilde{x}_n = \e^{\tilde{A} \cdot} \tilde{x}_{n 0}$ is continuously differentiable with
\begin{align}  \label{eq:S(A,B,B*) appr beobb, 4}
\tilde{x}_n'(t) 
=
\begin{pmatrix}
x_n'(t) \\ v_{1 n}'(t) \\ v_{2 n}'(t)
\end{pmatrix}
= 
\tilde{A}
\begin{pmatrix}
x_n(t) \\ v_{1 n}(t) \\ v_{2 n}(t)
\end{pmatrix}
= 
\begin{pmatrix}
\mathcal{A}x_n(t) \\ K v_{2 n}(t) \\ -v_{1n}(t) + B_c y_n(t)
\end{pmatrix}
\end{align}
for all $t \in [0,\infty)$. So, by~\eqref{eq:S(A,B,B*) appr beobb, 3} it follows that $v_{1n} = v_{1n0}$ is constant and that
\begin{align} \label{eq:S(A,B,B*) appr beobb, 5}
0 = v_{2n}'(t) = -v_{1n0} + B_c y_n(t)
\qquad (t \in [0,\infty)).
\end{align}
And from this, in turn, it follows by the injectivity of $B_c$ (Condition~\ref{cond:wISS, controller}~(ii)) that 
\begin{align} \label{eq:S(A,B,B*) appr beobb, 6}
y_n(t) = L_c v_{1n0} =: y_{n0} 
\qquad (t \in [0,\infty)),
\end{align}
where $L_c$ is an arbitrary left-inverse of $B_c$. 
We also know from~(\ref{eq:S(A,B,B*) appr beobb, 2}.a) that $\tilde{x}_n(t) = \e^{\tilde{A} t} \tilde{x}_{n 0} \in D(\tilde{A}) \subset \ker \tilde{\mathcal{B}}$ for all $t$ and so by~\eqref{eq:S(A,B,B*) appr beobb, 3} and~\eqref{eq:S(A,B,B*) appr beobb, 6} it follows that
\begin{align} \label{eq:S(A,B,B*) appr beobb, 7}
0 = \tilde{\mathcal{B}}\tilde{x}_n(t) = \mathcal{B}x_n(t) + B_c^{\top} K v_{2n}(t) + S_c \mathcal{C} x_n(t)
= u_n(t) + S_c y_{n0} 
\qquad (t \in [0,\infty)). 
\end{align}
Since $\mathfrak{S}$ is impedance-passive (Condition~\ref{cond:S imp-passive}) and since $x_n = x(\cdot,x_{n0},u_n)$ by the uniqueness statement around~\eqref{eq:class-sol-of-S}, it further follows using~\eqref{eq:S(A,B,B*) appr beobb, 6} and~\eqref{eq:S(A,B,B*) appr beobb, 7} that
\begin{align*}
0 \le E(x_n(t)) \le E(x_{n0}) + \int_0^t u_n(s)^{\top} y_n(s) \d s = E(x_{n0})  - \big(y_{n0}^{\top} S_c y_{n0} \big) t
\end{align*}
for all $t \in [0,\infty)$. So, by the positive definiteness of $S_c$, we see that $y_{n0} = 0$ and thus by~\eqref{eq:S(A,B,B*) appr beobb, 6}, \eqref{eq:S(A,B,B*) appr beobb, 5}, \eqref{eq:S(A,B,B*) appr beobb, 7} that
\begin{align} \label{eq:S(A,B,B*) appr beobb, 8}
y_n(t) = 0 \qquad \text{and} \qquad v_{1n0} = 0 \qquad \text{and} \qquad u_n(t) = 0  
\end{align}
for all $t \in [0,\infty)$.
Since now $\mathfrak{S}$ is classically approximately observable in infinite time (Condition~\ref{cond:wISS, system}) and since $y_n = y(\cdot,x_{n0},u_n)$, we conclude by~(\ref{eq:S(A,B,B*) appr beobb, 8}.a) and (\ref{eq:S(A,B,B*) appr beobb, 8}.c) that
\begin{align} \label{eq:S(A,B,B*) appr beobb, 9}
x_{n0} = 0.
\end{align}
Combining~\eqref{eq:S(A,B,B*) appr beobb, 3}, (\ref{eq:S(A,B,B*) appr beobb, 8}.b), \eqref{eq:S(A,B,B*) appr beobb, 9} with~(\ref{eq:S(A,B,B*) appr beobb, 2}.b), we finally get $\tilde{x}_0 = 0$ as desired.
\end{proof}

A second important ingredient is the following stabilization result for the collocated linear system $\mathfrak{S}(\tilde{A}, \tilde{B}, \tilde{B}^*)$, %which essentially goes back to~\cite{Os00}.
which hinges on the approximate observability property just established (Lemma~\ref{lm:S(A,B,B*) appr beobb}) %in the previous lemma 
and on the compactness of the resolvent of $\tilde{A}$ (Lemma~\ref{lm:sgr gen with comp resolv}).

\begin{lm} \label{lm:lin stabsatz}
Under the assumptions of the above theorem, one has for every $\lambda \in (0,\infty)$: %Suppose  Condition~\ref{cond:imp-passive}, \ref{cond:B surj} and~\ref{cond:R and P, strengthened} are satisfied. Suppose further that $\mathfrak{S}(\tilde{A}, \tilde{B}, \tilde{B}^*)$ is approximately controllable in infinite time or approximately observable in infinite time and let $\lambda \in (0,\infty)$. Then %for every $\lambda \in (0,\infty)$
\begin{itemize}
\item[(i)] the linear operator $$L^2([0,\infty),\R^{m_c}) \ni u \mapsto \int_0^{\infty} \e^{(\tilde{A}-\lambda \tilde{B}\tilde{B}^*)s} \tilde{B} u(s) \d s \in \tilde{X}$$ is well-defined and bounded with operator norm less than or equal to $1$ %$1/2$
\item[(ii)] the semigroup $\e^{(\tilde{A}-\lambda \tilde{B}\tilde{B}^*)\cdot}$ is strongly stable and $(\tilde{A}-\lambda \tilde{B}\tilde{B}^*)^{-1}$ is a compact operator in $\tilde{X}$
\item[(iii)] for every $u \in L^2([0,\infty),\R^{m_c})$, 
\begin{align*}
\int_0^t \e^{(\tilde{A}-\lambda \tilde{B}\tilde{B}^*)(t-s)} \tilde{B} u(s) \d s 
\longrightarrow 0 \qquad (t \to \infty)
\end{align*}
\item[(iv)] for every %bounded and absolutely continuous $u: [0,\infty) \to U$
$u \in AC_{\mathrm{loc}}([0,\infty),\R^{m_c}) \cap L^{\infty}([0,\infty),\R^{m_c})$ with $u' \in L^2([0,\infty),\R^{m_c})$ and every sequence $(t_n)$ with $t_n \longrightarrow \infty$, there is a subsequence $(t_{n_l})$ such that the following limits exist and conincide: 
\begin{align*}
\lim_{l \to \infty} \int_0^{t_{n_l}} \e^{(\tilde{A}-\lambda \tilde{B}\tilde{B}^*)(t_{n_l}-s)} \tilde{B} u(s) \d s 
= -\lim_{l\to\infty} (\tilde{A}-\lambda \tilde{B}\tilde{B}^*)^{-1} \tilde{B} u(t_{n_l}).
\end{align*} 
\end{itemize}
\end{lm}

\begin{proof}
Since $\mathfrak{S}(\tilde{A}, \tilde{B}, \tilde{B}^*)$ is approximately observable in infinite time (Lemma~\ref{lm:S(A,B,B*) appr beobb}), assertion~(i), the strong stability part of assertion~(ii), and assertion~(iii) follow by a well-known stability result for collocated systems, see \cite{CurtainZwart} or \cite{Oo00} (Lemma~2.2.6 and~2.1.3). 
Also, by the strong stability of $\e^{(\tilde{A}-\lambda \tilde{B}\tilde{B}^*)\cdot}$ and the compactness of the resolvent of $\tilde{A}$ and hence $\tilde{A}-\lambda \tilde{B}\tilde{B}^*$ (Lemma~\ref{lm:sgr gen with comp resolv}), we have
\begin{align*}
0 \notin \sigma_p(\tilde{A}-\lambda \tilde{B}\tilde{B}^*) = \sigma(\tilde{A}-\lambda \tilde{B}\tilde{B}^*).
\end{align*}
So, $(\tilde{A}-\lambda \tilde{B}\tilde{B}^*)^{-1}$ exists and is compact.
And finally, assertion~(iv) follows by partial integration:
\begin{align*}
\int_0^{t_n} \e^{(\tilde{A}-\lambda \tilde{B}\tilde{B}^*)(t_{n}-s)} \tilde{B} u(s) \d s 
&=  - (\tilde{A}-\lambda \tilde{B}\tilde{B}^*)^{-1} \e^{(\tilde{A}-\lambda \tilde{B}\tilde{B}^*)(t_{n}-s)} \tilde{B} u(s) \Big|_{s=0}^{s=t_n} \notag \\
&\quad +  (\tilde{A}-\lambda \tilde{B}\tilde{B}^*)^{-1} \int_0^{t_n} \e^{(\tilde{A}-\lambda \tilde{B}\tilde{B}^*)(t_{n}-s)} \tilde{B} u'(s) \d s
\end{align*}
and by then exploiting %applying/making use of/appealing to 
assertions~(ii) and~(iii). 
\end{proof}

\begin{lm} \label{lm:lin stabsatz anwendbar}
Under the assumptions of the above theorem, one has for every $\tilde{x}_0 \in X$ and $d \in L^2([0,\infty),\R^k)$:
%Suppose Condition~\ref{cond:imp-passive}, \ref{cond:B surj} and~\ref{cond:R and P, strengthened} are satisfied. Then, for every $\tilde{x}_0 \in X$ and $d \in L^2([0,\infty),\R^k)$, 
\begin{itemize}
\item[(i)] $\tilde{B}^* \tilde{x}(\cdot,\tilde{x}_0,d)$ and $g\circ (\tilde{B}^* \tilde{x}(\cdot,\tilde{x}_0,d))$ belong to $L^2([0,\infty),\R^{m_c})$. 
\item[(ii)] $h \circ (\tilde{C} \tilde{x}(\cdot,\tilde{x}_0,d) )$ belongs to $AC_{\mathrm{loc}}([0,\infty),\R^{m_c}) \cap L^{\infty}([0,\infty),\R^{m_c})$ and the derivative $(h \circ (\tilde{C} \tilde{x}(\cdot,\tilde{x}_0,d) ))'$ belongs to $L^2([0,\infty),\R^{m_c})$. 
\end{itemize}
\end{lm}

\begin{proof}
In the entire proof, we adopt the short-hand notation $R(\tilde{x}_0,d)$ for $\tilde{x}_0 \in \tilde{X}$ and $d \in L^2([0,\infty),\R^k)$ from the proof of Lemma~\ref{lm:S_lin admissible}. We also denote by $L_g(R)$ and $L_h(R)$ for $R \in [0,\infty)$ Lipschitz constants of $g|_{\ol{B}_R(0)}$ and  $h|_{\ol{B}_R(0)}$ such that $R \mapsto L_g(R), L_h(R)$ are continuous and monotonically increasing. 
\smallskip

%Aussage (i)
%
(i)
We first show that for every $(\tilde{x}_0,d) \in \tilde{X} \times L^2([0,\infty),\R^k)$ and every $t \in [0,\infty)$ one has the following estimates: 
\begin{align} 
\norm{ \tilde{B}^* \tilde{x}(\cdot,\tilde{x}_0,d) }_{[0,t],2}^2 
&\le
\Big( 1/\ul{c} + \norm{K} R(\tilde{x}_0,d)^2/\ol{c} \Big) \bigg( \tilde{E}(\tilde{x}_0) + \frac{1}{4\varsigma} \norm{d}_{[0,t],2}^2 \bigg) 
%\big( \ul{c}^{-1} + \ol{c}^{-1} \norm{K} R(\tilde{x}_0,d)^2 \big) %\bigg( \frac{1}{\ul{c}} + \norm{K} \frac{R(\tilde{x}_0,d)^2}{\ol{c}} \bigg) 
\label{eq:B*x in L^2} \\
\norm{ g \circ \big( \tilde{B}^* \tilde{x}(\cdot,\tilde{x}_0,d) \big) }_{[0,t],2}^2
&\le 
L_g\big( \norm{K}^{1/2} R(\tilde{x}_0,d) \big)^2 \Big( 1/\ul{c} + \norm{K} R(\tilde{x}_0,d)^2/\ol{c} \Big) \cdot \notag \\ 
& \qquad  \qquad \qquad \quad \cdot \bigg( \tilde{E}(\tilde{x}_0) + \frac{1}{4\varsigma} \norm{d}_{[0,t],2}^2 \bigg). 
\label{eq:g circ B*x in L^2} 
\end{align}
Since both sides of~\eqref{eq:B*x in L^2} and~\eqref{eq:g circ B*x in L^2} are continuous w.r.t.~$(\tilde{x}_0,d) \in \tilde{X} \times L^2([0,\infty),\R^k)$ by Corollary~\ref{cor:general sol, properties}, %and since $\mathcal{D}$ is dense in $L^2([0,\infty),\R^k)$
it is sufficient to prove these estimates only for all classical data $(\tilde{x}_0,d) \in \mathcal{D}$. So let $(\tilde{x}_0,d) \in \mathcal{D}$ and $t \in [0,\infty)$ and write $(x,v_1,v_2) := \tilde{x} := \tilde{x}(\cdot,\tilde{x}_0,d)$,
%\begin{align*}
%(x,v_1,v_2) := \tilde{x} := \tilde{x}(\cdot,\tilde{x}_0,d).
%\end{align*} 
so that $\tilde{B}^* \tilde{x}(\cdot,\tilde{x}_0,d) = Kv_2$. 
It follows by Condition~\ref{cond:wISS, controller}~(i) 
%and by Theorem~\ref{thm:global klass loesb, klass UGS}~(ii) 
%by $\norm{Kv_2}_{ [0,t],\infty} \le \norm{K^{1/2}} \norm{\tilde{x}(\cdot,\tilde{x}_0,d)}_{[0,t],\infty} \le \norm{K^{1/2}} R(\tilde{x}_0,d)$
that
\begin{align} \label{eq:B*x in L^2, klass, 1}
\norm{ \tilde{B}^* \tilde{x}(\cdot,\tilde{x}_0,d) }_{[0,t],2}^2  
&= \int_0^t |Kv_2(s)|^2 \d t 
= \int_{J_{\le \delta}^t} |Kv_2(s)|^2 \d t +  \int_{J_{> \delta}^t} |Kv_2(s)|^2 \d t \notag \\
&\le 
\Big( 1/\ul{c} + \norm{K} R(\tilde{x}_0,d)^2/\ol{c} \Big) \int_0^t (Kv_2(s))^{\top} \mathcal{R}(Kv_2(s)) \d s
\end{align}
for every $t \in [0,\infty)$, where we used the abbreviations 
\begin{align*}
J_{\le \delta}^t := \big\{ s \in [0,t]: |Kv_2(s)| \le \delta \big\} 
\qquad \text{and} \qquad
J_{> \delta}^t := \big\{ s \in [0,t]: |Kv_2(s)| > \delta \big\}
\end{align*}
as well as the fact that %by Theorem~\ref{thm:global klass loesb, klass UGS} (ii)
\begin{align} \label{eq:absch sup|Kv_2|}
\sup_{s\in [0,\infty)} |Kv_2(s)|^2 \le  \norm{K} \sup_{s \in [0,\infty)} \norm{\tilde{x}(s,\tilde{x}_0,d)}^2 \le  \norm{K} R(\tilde{x}_0,d)^2
%\norm{Kv_2}_{ [0,t],\infty} \le \norm{K^{1/2}} \norm{\tilde{x}(\cdot,\tilde{x}_0,d)}_{[0,t],\infty} \le \norm{K^{1/2}} R(\tilde{x}_0,d)
\end{align}
(Theorem~\ref{thm:class solvb} (ii)).
It further follows from~\eqref{eq:class solvb, derivative of energy, estimate} %(or, more precisely, the first~\ref{} lines of it) %(up to its \ref{} line) %(ignoring its last line) 
with $\alpha := 1/(2\varsigma)$ that
\begin{align} \label{eq:B*x in L^2, klass, 2}
\int_0^t (Kv_2(s))^{\top} \mathcal{R}(Kv_2(s)) \d s 
\le \tilde{E}(\tilde{x}_0) + \frac{1}{4\varsigma} \norm{d}_{[0,t],2}^2 
%
%\int_0^t (Kv_2(s))^{\top} \mathcal{R}(Kv_2(s)) \d s 
%&\le -\int_0^t \tilde{E}_{\tilde{x}}'(s) \d s +  \frac{1}{4\varsigma} \int_0^t |d(s)|^2 \d s  \notag \\
%&= \tilde{E}(\tilde{x}_0) + \frac{1}{4\varsigma} \norm{d}_{[0,t],2}^2 
\end{align}
for every $t \in [0,\infty)$. 
Combining~\eqref{eq:B*x in L^2, klass, 1} and~\eqref{eq:B*x in L^2, klass, 2} we obtain~\eqref{eq:B*x in L^2}. And since $g(0) = 0$ we we also obtain~\eqref{eq:g circ B*x in L^2} using~\eqref{eq:absch sup|Kv_2|}. Assertion~(i) now follows by letting $t \to \infty$ in~\eqref{eq:B*x in L^2} and~\eqref{eq:g circ B*x in L^2}.
\smallskip

%Aussage (ii)
%
(ii) Choose and fix $\tilde{x}_0  \in \tilde{X}$ and $d \in L^2([0,\infty),\R^k)$ and write $(x,v_1,v_2) := \tilde{x} := \tilde{x}(\cdot,\tilde{x}_0,d)$,
%\begin{align*}
%(x,v_1,v_2) := \tilde{x} := \tilde{x}(\cdot,\tilde{x}_0,d)
%\end{align*}
so that $\tilde{C} \tilde{x}(\cdot,\tilde{x}_0,d) = v_1$. Since $\tilde{x}(\cdot,\tilde{x}_0,d) $ is the locally uniform limit of classical solutions by definition~\eqref{eq:general sol, def}, %Theorem~\ref{thm:density and approximation},
we see that
\begin{align} \label{eq:v_1' = Kv_2}
v_1(t) = v_1(0) + \int_0^t Kv_2(s) \d s
\end{align}
for all $t \in [0,\infty)$ and hence $v_1$ is locally Lipschitz continuous (even continuously differentiable). %$v_1$ is continuously differentiable with $v_1' = Kv_2$. 
So, since $h$ is locally Lipschitz continuous as well and since
\begin{align} \label{eq:absch sup|v_1|}
\sup_{s\in [0,\infty)} |v_1(s)| \le \sup_{s \in [0,\infty)} \norm{\tilde{x}(s,\tilde{x}_0,d)} \le R(\tilde{x}_0,d)
\end{align}
by Theorem~\ref{thm:UGS},  
%Theorem~\ref{thm:global klass loesb, klass UGS}~(ii) and approximation, 
it follows that the composition $h \circ v_1$ is locally Lipschitz continuous and bounded. In particular, $h \circ (\tilde{C} \tilde{x}(\cdot,\tilde{x}_0,d) ) = h \circ v_1$ belongs to $AC_{\mathrm{loc}}([0,\infty),\R^{m_c}) \cap L^{\infty}([0,\infty),\R^{m_c})$ and thus is differentiable almost everywhere. %with derivative satisfying
In view of~\eqref{eq:absch sup|v_1|} and~\eqref{eq:v_1' = Kv_2} its derivative satisfies
\begin{align*}
|(h\circ v_1)'(s)| &= \lim_{\tau \to 0} \bigg| \frac{h(v_1(s+\tau))-h(v_1(s))}{\tau} \bigg| \le L_h(R(\tilde{x}_0,d)) \lim_{\tau \to 0} \bigg| \frac{v_1(s+\tau)-v_1(s)}{\tau} \bigg| \notag \\
%&= L_h(R(\tilde{x}_0,d))  |Kv_2(s)|
&= L_h(R(\tilde{x}_0,d))  |\tilde{B}^* \tilde{x}(s,\tilde{x}_0,d) |
\end{align*}
for almost every $s \in [0,\infty)$. So, 
\begin{align} \label{eq:(h circ Cx)' in L^2}
\norm{(h \circ (\tilde{C} \tilde{x}(\cdot,\tilde{x}_0,d) ))'}_{[0,t],2} = \norm{(h\circ v_1)'}_{[0,t],2} \le L_h(R(\tilde{x}_0,d)) \norm{ \tilde{B}^* \tilde{x}(\cdot,\tilde{x}_0,d) }_{[0,t],2}
\end{align}
for every $t \in [0,\infty)$. 
In particular, $(h \circ (\tilde{C} \tilde{x}(\cdot,\tilde{x}_0,d) ))' = (h \circ v_1)'$ belongs to $L^2([0,\infty),\R^{m_c})$ by assertion~(i), which concludes the proof of assertion~(ii).
\end{proof}

A third important ingredient is the following lemma which says that the linear boundary control system
\begin{gather}
\tilde{x}' = \big( \tilde{\mathcal{A}} - \lambda \tilde{B} \tilde{B}^* \big) \tilde{x} 
\quad \text{with} \quad 
d(t) = \tilde{\mathcal{B}}\tilde{x}(t)
\end{gather} 
for $\lambda > 0$ is infinite-time admissible w.r.t.~inputs $d \in L^2([0,\infty),\R^k)$. 

\begin{lm} \label{lm:S^kappa_lin infinite-time admissible}
Under the assumptions of the above theorem, 
%Suppose  Condition~\ref{cond:imp-passive}, \ref{cond:B surj} and~\ref{cond:R and P, strengthened} are satisfied. Suppose further that $\mathfrak{S}(\tilde{A}, \tilde{B}, \tilde{B}^*)$ is approximately controllable in infinite time or approximately observable in infinite time and let $\lambda \in (0,\infty)$. Then 
the linear operator $\Phi_t^{\lambda}: C^2_c([0,\infty),\R^k) \to \tilde{X}$ defined by~\eqref{eq:Phi_t^kappa, def} can be (uniquely) extended to a bounded linear operator  $\ol{\Phi}_t^{\lambda}: L^2([0,\infty),\R^k) \to \tilde{X}$ for every $t \in [0,\infty)$ and $\lambda \in (0,\infty)$ and 
\begin{align} \label{eq:S^kappa_lin infinite-time admissible}
\sup_{t\in[0,\infty) } \norm{ \ol{\Phi}_t^{\lambda} } < \infty.
\end{align}
\end{lm}

\begin{proof}
Choose and fix $\lambda \in (0,\infty)$. We adopt the short-hand notations $R(\tilde{x}_0,d)$ and $L_g(R), L_h(R)$ from the proof of Lemma~\ref{lm:lin stabsatz anwendbar}. 
\smallskip

%Schritt 1
As a first step, we show that the restriction $\Phi_{t,0}^{\lambda}$ of $\Phi_t^{\lambda}$ to 
\begin{align*}
C_{c,0}^2([0,\infty),\R^k) := \big\{ d \in C_c^2([0,\infty),\R^k): d(0) = 0 \big\}
\end{align*}
can be (uniquely) extended to a bounded linear operator  $\ol{\Phi}_{t,0}^{\lambda}: L^2([0,\infty),\R^k) \to \tilde{X}$ for every $t \in [0,\infty)$ and 
\begin{align} \label{eq:S^kappa_lin infinite-time admissible, 1}
\sup_{t\in[0,\infty) } \norm{ \ol{\Phi}_{t,0}^{\lambda} } < \infty.
\end{align}
So let $t \in [0,\infty)$. We see by~\eqref{eq:voc, kappa ne 0} that
\begin{align} \label{eq:voc, kappa ne 0, x_0 = 0}
\Phi_{t,0}^{\lambda}(d) = \tilde{x}(t,0,d) 
&- \int_0^t \e^{(\tilde{A}-\lambda \tilde{B}\tilde{B}^*)(t-s)} \tilde{B} \big( \lambda \tilde{B}^* \tilde{x}(s,0,d) + g(\tilde{B}^* \tilde{x}(s,0,d) ) \big) \d s \notag \\
& -  \int_0^t \e^{(\tilde{A}-\lambda \tilde{B}\tilde{B}^*)(t-s)} \tilde{B} h(\tilde{C} \tilde{x}(s,0,d) ) \d s
\end{align}
for every $d \in C_{c,0}^2([0,\infty),\R^k)$. 
%
%Schritt 1.1
It follows by Theorem~\ref{thm:class solvb}~(ii) that
\begin{align}  \label{eq:S^kappa_lin infinite-time admissible, 1.1}
\norm{\tilde{x}(t,0,d)} \le \ul{\gamma}(\norm{d}_{[0,t],2})
\end{align}
for all $d \in C_{c,0}^2([0,\infty),\R^k)$.
%
%Schritt 1.2
Also, it follows by Lemma~\ref{lm:lin stabsatz}~(i) together with~\eqref{eq:B*x in L^2}, \eqref{eq:g circ B*x in L^2} that
\begin{align}  \label{eq:S^kappa_lin infinite-time admissible, 1.2}
&\norm{  \int_0^t \e^{(\tilde{A}-\lambda \tilde{B}\tilde{B}^*)(t-s)} \tilde{B} \big( \lambda \tilde{B}^* \tilde{x}(s,0,d) + g(\tilde{B}^* \tilde{x}(s,0,d) ) \big) \d s  } \notag \\
&\quad \le 
\lambda \norm{ \tilde{B}^* \tilde{x}(\cdot,0,d) }_{[0,t],2} + \norm{ g \circ (\tilde{B}^* \tilde{x}(\cdot,0,d)) }_{[0,t],2} \notag \\
& \quad \le 
\Big( \lambda + L_g\big(\norm{K}^{1/2} R(0,d) \big) \Big) \Big( 1/\ul{c} + \norm{K} R(\tilde{x}_0,d)^2/\ol{c} \Big)^{1/2} 
\Big( \frac{1}{4 \varsigma} \Big)^{1/2} \norm{d}_{[0,t],2}
%
%\frac{1}{2 \varsigma^{1/2}} \norm{d}_{[0,t],2}
\end{align} 
for all $d \in C_{c,0}^2([0,\infty),\R^k)$.
%
%Schritt 1.3 
And finally, it follows by integration by parts (which is allowed by Lemma~\ref{lm:lin stabsatz anwendbar}~(ii)) 
and by Lemma~\ref{lm:lin stabsatz}~(i) together with~\eqref{eq:(h circ Cx)' in L^2}, \eqref{eq:B*x in L^2} %after an integration by parts (which is allowed by Lemma~\ref{lm:lin stabsatz anwendbar}) and 
that
\begin{align}  \label{eq:S^kappa_lin infinite-time admissible, 1.3}
&\norm{ \int_0^t \e^{(\tilde{A}-\lambda \tilde{B}\tilde{B}^*)(t-s)} \tilde{B} h(\tilde{C} \tilde{x}(s,0,d) ) \d s } 
\le 
\norm{ (\tilde{A}-\lambda \tilde{B}\tilde{B}^*)^{-1} \tilde{B} h(\tilde{C} \tilde{x}(t,0,d) )  } \notag \\
&\qquad + \norm{  (\tilde{A}-\lambda \tilde{B}\tilde{B}^*)^{-1} \int_0^t \e^{(\tilde{A}-\lambda \tilde{B}\tilde{B}^*)(t-s)} \tilde{B} \big( h \circ (\tilde{C} \tilde{x}(\cdot,0,d) )\big)'(s) \d s }  \notag \\
&\quad \le 
\norm{ (\tilde{A}-\lambda \tilde{B}\tilde{B}^*)^{-1}  } \bigg(  \|\tilde{B}\| \,  L_h(R(0,d)) \norm{\tilde{x}(t,0,d)} + \norm{ \big( h \circ (\tilde{C} \tilde{x}(\cdot,0,d) )\big)' }_{[0,t],2} \bigg) \notag \\
&\quad \le 
\norm{ (\tilde{A}-\lambda \tilde{B}\tilde{B}^*)^{-1}  } L_h(R(0,d)) \cdot \notag \\
& \qquad \quad 
\cdot \bigg( \|\tilde{B}\|\,  \ul{\gamma}( \norm{d}_{[0,t],2} ) + \Big( 1/\ul{c} + \norm{K} R(\tilde{x}_0,d)^2/\ol{c} \Big)^{1/2} 
\Big( \frac{1}{4 \varsigma} \Big)^{1/2} \norm{d}_{[0,t],2} \bigg) 
\end{align}
for all $d \in C_{c,0}^2([0,\infty),\R^k)$.
%
%Schritt 1.4
Combining now~\eqref{eq:voc, kappa ne 0, x_0 = 0} and~\eqref{eq:S^kappa_lin infinite-time admissible, 1.1}, \eqref{eq:S^kappa_lin infinite-time admissible, 1.2}, \eqref{eq:S^kappa_lin infinite-time admissible, 1.3} we see that
\begin{align} \label{eq:absch Phi_t,0^kappa(d)}
\big\| \ol{\Phi}_{t,0}^{\lambda}(d) \big\| \le C_{\lambda}(R(0,d)) \Big( \norm{d}_{[0,t],2} + \ul{\gamma}( \norm{d}_{[0,t],2} ) \Big)
\end{align}
for all $d \in C_{c,0}^2([0,\infty),\R^k)$, where $[0,\infty) \ni R \mapsto C_{\lambda}(R)$ is a continuous monotonically increasing function. 
It follows %from this 
that $$\Phi_{t,0}^{\lambda}: C_{c,0}^2([0,\infty),\R^k) \to \tilde{X}$$ is a linear operator that is bounded w.r.t.~the norm of $L^2([0,\infty),\R^k)$ and therefore, by the density of $C_{c,0}^2([0,\infty),\R^k)$ in $L^2([0,\infty),\R^k)$, can be uniquely extended to a bounded linear operator $\ol{\Phi}_{t,0}^{\lambda}: L^2([0,\infty),\R^k) \to \tilde{X}$. 
It further follows that
\begin{align} \label{eq:absch ol(Phi)_t,0^kappa(d)}
\sup_{t\in [0,\infty)} \big\| \ol{\Phi}_{t,0}^{\lambda} \big\| 
&= \sup_{t\in [0,\infty)} \sup \big\{ \big\| \ol{\Phi}_{t,0}^{\lambda}(d) \big\| : d \in C_{c,0}^2([0,\infty),\R^k) \text{ with } \norm{d}_2 \le 1\big\}  \notag \\
&\le C_{\lambda}(R(0,1)) \big( 1 + \ul{\gamma}( 1 ) \big) < \infty.
\end{align}
%It further follows that
%\begin{align} \label{eq:absch ol(Phi)_t,0^kappa(d)}
%\sup_{t\in [0,\infty)} \norm{\ol{\Phi}_{t,0}^{\lambda}(d)} \le C_{\lambda}(R(0,d)) \Big( \norm{d}_{2} + \ul{\gamma}( \norm{d}_{2} ) \Big) < \infty
%\end{align}
%for every $d \in L^2([0,\infty),\R^k)$ and thus \eqref{eq:S^kappa_lin infinite-time admissible, 1} holds true by the theorem of Banach and Steinhaus.
%\smallskip

%Schritt 2
As a second step, we observe that also the non-restricted operator $\Phi_t^{\lambda}$ can be (uniquely) extended to a bounded linear operator  $\ol{\Phi}_{t}^{\lambda}: L^2([0,\infty),\R^k) \to \tilde{X}$ and that $\ol{\Phi}_{t}^{\lambda} = \ol{\Phi}_{t,0}^{\lambda}$ for every $t \in [0,\infty)$. %, which in conjunction with~\eqref{eq:S^kappa_lin infinite-time admissible, 1} proves the lemma. 
In order to do so, we have only to show that $$\ol{\Phi}_{t,0}^{\lambda}(d) = \Phi_t^{\lambda}(d)$$ for every $d \in C_c^2([0,\infty),\R^k)$. And this, in turn, can be achieved in the same way as in the second step of the proof of Lemma~\ref{lm:S_lin admissible}. (Instead of~\eqref{eq:absch Phi_t,0(d)} the essential ingredient now is the estimate~\eqref{eq:absch Phi_t,0^kappa(d)}, for it again allows us to conclude $\Phi_{t/n,0}^{ \lambda}(d_n) \longrightarrow 0$ as $n \to \infty$ for every sequence $(d_n)$ as in~\eqref{eq:d_n for extension from C_c,0^2 to C_c^2}.) 
\smallskip

%Abschluss und Bemerkung zu Alternativbeweis der blossen stet Fb
In conjunction with~\eqref{eq:S^kappa_lin infinite-time admissible, 1} the second step proves the lemma. 
We finally remark %point out 
that the mere bounded extendability of $\Phi_t^{\lambda}$ (without the finiteness condition~\eqref{eq:S^kappa_lin infinite-time admissible}) 
could alternatively also be concluded from Lemma~\ref{lm:S_lin admissible} by a perturbation argument (for which the approximate  observability of $\mathfrak{S}(\tilde{A},\tilde{B}^*,\tilde{B})$ would not be needed).
\end{proof}

\subsubsection{Conclusion of the proof} %Core of the proof

With the above lemmas at hand, we can now conclude the proof of weak input-to-state stability with weak asymptotic gain $\ol{\gamma} = 0$ in three simple steps. %our weak input-to-state stability result in three simple steps. 
Since we already know that the closed-loop system is uniformly globally stable (Theorem~\ref{thm:UGS}), we have only to show that
\begin{align} \label{eq:wISS, convergence to 0}
\tilde{x}(t,\tilde{x}_0,d) \longrightarrow 0 \qquad (t \to \infty)
\end{align}
for every $\tilde{x}_0 \in \tilde{X}$ and $d \in L^2([0,\infty),\R^k)$.
So let $(\tilde{x}_0,d) \in \tilde{X} \times L^2([0,\infty),\R^k)$. It then follows from~\eqref{eq:voc, kappa ne 0} by a simple approximation argument (using~\eqref{eq:general sol, def} and Lemma~\ref{lm:S^kappa_lin infinite-time admissible}) that
\begin{align} \label{eq:voc, kappa ne 0, verallg}
\tilde{x}(t,\tilde{x}_0,d) = \e^{(\tilde{A}-\lambda \tilde{B}\tilde{B}^*)t} \tilde{x}_0 
& + \int_0^t \e^{(\tilde{A}-\lambda \tilde{B}\tilde{B}^*)(t-s)} \tilde{B} \big( \lambda \tilde{B}^* \tilde{x}(s,\tilde{x}_0,d) + g(\tilde{B}^* \tilde{x}(s,\tilde{x}_0,d) ) \big) \d s \notag \\
& +  \int_0^t \e^{(\tilde{A}-\lambda \tilde{B}\tilde{B}^*)(t-s)} \tilde{B} h(\tilde{C} \tilde{x}(s,\tilde{x}_0,d) ) \d s
+ \ol{\Phi}_t^{\lambda}(d)  
\end{align} 
for all $t \in [0,\infty)$ and $\lambda \in (0,\infty)$. 
We now show in three steps
that all terms on the right-hand side of~\eqref{eq:voc, kappa ne 0, verallg} converge to $0$ as $t \to \infty$. 
\smallskip

%Schritt 1
As a first step, we observe from Lemma~\ref{lm:lin stabsatz}~(ii) and from Lemma~\ref{lm:lin stabsatz}~(iii) in conjunction with Lemma~\ref{lm:lin stabsatz anwendbar}~(i) that
\begin{equation} \label{eq:stabsatz, 1}
\begin{gathered}
\e^{(\tilde{A}-\lambda \tilde{B}\tilde{B}^*)t} \tilde{x}_0 \longrightarrow 0 \qquad (t \to \infty), \\
\int_0^t \e^{(\tilde{A}-\lambda \tilde{B}\tilde{B}^*)(t-s)} \tilde{B} \big( \lambda \tilde{B}^* \tilde{x}(s,\tilde{x}_0,d) + g(\tilde{B}^* \tilde{x}(s,\tilde{x}_0,d) ) \big) \d s
\longrightarrow 0 \qquad (t \to \infty)
\end{gathered}
\end{equation} 
for every $\lambda \in (0,\infty)$. 
\smallskip

%Schritt 2
As a second step, we observe from Lemma~\ref{lm:S^kappa_lin infinite-time admissible} that
\begin{align} \label{eq:stabsatz, 2}
\ol{\Phi}_t^{\lambda}(d)  \longrightarrow 0 \qquad (t \to \infty)
\end{align}
for every $\lambda \in (0,\infty)$. Indeed, for $d_0 \in C_c^2([0,\infty),\R^k)$ the convergence $\Phi_t^{\lambda}(d_0)  \longrightarrow 0$ as $t \to \infty$ immediately follows from~\eqref{eq:Phi_t^kappa, def} by virtue of Lemma~\ref{lm:lin stabsatz}~(ii). And from this, in turn, the desired convergence~\eqref{eq:stabsatz, 2} follows by density and Lemma~\ref{lm:S^kappa_lin infinite-time admissible}. 
\smallskip

%Schritt 3
As a third and last step, we show that also
\begin{align} \label{eq:stabsatz, 3}
\int_0^t \e^{(\tilde{A}-\lambda \tilde{B}\tilde{B}^*)(t-s)} \tilde{B} h(\tilde{C} \tilde{x}(s,\tilde{x}_0,d) ) \d s
\longrightarrow 0 \qquad (t \to \infty)
\end{align}
for every $\lambda \in (0,\infty)$. Choose an arbitrary sequence $(t_n)$ with $t_n \longrightarrow \infty$ as $n \to \infty$. Then by Lemma~\ref{lm:lin stabsatz}~(iv) in conjunction with Lemma~\ref{lm:lin stabsatz anwendbar}~(ii) there exists a subsequence $(t_{n_l})$ such that for both $\lambda = 1$ and $\lambda = 2$ one has:
\begin{align} \label{eq:stabsatz, 3.1}
&\lim_{l\to \infty} \int_0^{t_{n_l}} \e^{(\tilde{A}-\lambda \tilde{B}\tilde{B}^*)(t_{n_l}-s)} \tilde{B} h(\tilde{C} \tilde{x}(s,\tilde{x}_0,d) ) \d s \notag \\
&\qquad \qquad \qquad \qquad = - \lim_{l\to\infty} (\tilde{A}-\lambda \tilde{B}\tilde{B}^*)^{-1} \tilde{B} h(\tilde{C}\tilde{x}(t_{n_l},\tilde{x}_0,d)).
\end{align}
Combining now~\eqref{eq:voc, kappa ne 0, verallg} with~\eqref{eq:stabsatz, 1}, \eqref{eq:stabsatz, 2}, \eqref{eq:stabsatz, 3.1} we see that $\tilde{x}_* = \lim_{l\to \infty} \tilde{x}(t_{n_l},\tilde{x}_0,d)$
%\begin{align}
%\tilde{x}_* = \lim_{l\to \infty} \tilde{x}(t_{n_l},\tilde{x}_0,d)
%\end{align}
exists in $\tilde{X}$ and that
\begin{align}
\tilde{x}_* =  -(\tilde{A}-\lambda \tilde{B}\tilde{B}^*)^{-1} \tilde{B} h(\tilde{C}\tilde{x}_*)
\end{align}
for both $\lambda=1$ and $\lambda=2$. So, $\tilde{x}_* \in D(\tilde{A})$ and
\begin{align}
(\tilde{A}-\tilde{B}\tilde{B}^*) \tilde{x}_* =  -\tilde{B} h(\tilde{C}\tilde{x}_*) = (\tilde{A}-2\tilde{B}\tilde{B}^*)\tilde{x}_*.
\end{align}
It follows from this that on the one hand
\begin{align} \label{eq:stabsatz, 3.2}
\tilde{B}\tilde{B}^* \tilde{x}_* = 0 \qquad \text{hence} \qquad \tilde{B}^* \tilde{x}_* = 0
\end{align}
and on the other hand
\begin{align} \label{eq:stabsatz, 3.3}
\tilde{A} \tilde{x}_* + \tilde{f}(\tilde{x}_*) = \tilde{A} \tilde{x}_* + \tilde{B} g(\tilde{B}^* \tilde{x}_*) + \tilde{B} h(\tilde{C} \tilde{x}_*)
=  \tilde{A} \tilde{x}_* + \tilde{B} h(\tilde{C} \tilde{x}_*)
= 0.
\end{align}
In other words, \eqref{eq:stabsatz, 3.2} and \eqref{eq:stabsatz, 3.3} say that $\tilde{x}_*$ is an equilibrium point of~\eqref{eq:closed-loop mit zero-input} %satisfying $v_{*2} = K^{-1}   \tilde{B}^* \tilde{x}_* = 0$ 
and thus $\tilde{x}_* = 0$ by assumption. 
So, summarizing we have shown that for every sequence $(t_n)$ with $t_n \longrightarrow \infty$ there exists a subsequence $(t_{n_l})$ such that
\begin{align}
\lim_{l\to \infty} \tilde{x}(t_{n_l},\tilde{x}_0,d) = \tilde{x}_* = 0. 
\end{align}
And from this, in turn, the asserted convergence~\eqref{eq:wISS, convergence to 0} %$\tilde{x}(t,\tilde{x}_0,d) \longrightarrow 0$ as $t \to \infty$ 
follows (and therefore, in view of the first two steps, also~\eqref{eq:stabsatz, 3}).

\subsection{Some remarks on the assumptions} %interrelation and variations

In this section, we discuss specializations and generalizations of the assumptions from the weak input-to-state stability result (Theorem~\ref{thm:wISS}). In particular, we clarify their relation to the assumptions from the uniform input-to-state stability result (Theorem~\ref{thm:ISS}). As the following two lemmas show, the assumption on $\mathfrak{S}$ from Theorem~\ref{thm:ISS} (Condition~\ref{cond:ISS, system}) is more or less -- apart from some minor extra conditions which are often satisfied in applications -- a sufficient condition both for the approximate observability assumption (Condition~\ref{cond:wISS, system}) and for the equilibrium point assumption from Theorem~\ref{thm:wISS}. See the conference paper~\cite{ScZw18-MTNS} for %accordingly simplified 
versions of the stability theorems above that are simplified accordingly.  
In our applied examples (Section~\ref{sect:applications}), we will make ample use of this. %often use this fact. 
%In the next two lemmas, we will see that the assumptions from the weak ISS result that are not related to the controller

\begin{lm} \label{lm:hinr bed fuer klass appr beobb-vor an S aus wISS-satz}
Suppose  that the assumptions (Conditions~\ref{cond:S imp-passive}, \ref{cond:P und R}, \ref{cond:B-tilde surj}) of the solvability  theorems are satisfied and, in addition, that $\mathfrak{S}$ is even impedance-energy-preserving meaning that~\eqref{eq:S imp-passive} holds with equality. If then Condition~\ref{cond:ISS, system} is satisfied, then so is Condition~\ref{cond:wISS, system}.
\end{lm}

\begin{proof}
Suppose that Condition~\ref{cond:ISS, system} is satisfied and that $x_0 \in D(A)$ is such that
\begin{align} \label{eq:hinr bed klass appr beobb, 1}
y(t,x_0,0) = \mathcal{C}x(t,x_0,0) = 0 \qquad (t \in [0,\infty)).
\end{align}
Since $x(\cdot,x_0,0) = \e^{A\cdot} x_0$ is a classical solution of $x'= \mathcal{A}x$ with $t \mapsto \norm{x(t,x_0,0)}_X$ being monotonically decreasing and since Condition~\ref{cond:ISS, system} is satisfied, there exist positive constants $C_0,t_0$ such that
\begin{align} \label{eq:hinr bed klass appr beobb, 2}
\norm{x(t,x_0,0)}_X^2 
&\le C_0 \int_0^t |(\mathcal{H}x(s,x_0,0))(\eta)|^2 \d s \notag \\
&\le 
C_0/\kappa \int_0^t |(\mathcal{B}x(s,x_0,0))(\eta)|^2 + |(\mathcal{C}x(s,x_0,0))(\eta)|^2 \d s
\end{align}
for all $t \in [t_0,\infty)$ (Lemma~9.1.2 of~\cite{JacobZwart} or, more precisely, the third step in the proof of Theorem~3.5 of~\cite{Sc18-BV}). 
So, as $\mathcal{B}x(\cdot,x_0,0) \equiv 0$ and $\mathcal{C}x(\cdot,x_0,0) \equiv 0$ by~\eqref{eq:hinr bed klass appr beobb, 1}, we conclude that
\begin{align} \label{eq:hinr bed klass appr beobb, 3}
x(t,x_0,0) = 0 
\end{align}
at least for all $t \in [t_0,\infty)$. Since by the assumed impedance-energy-preservation of $\mathfrak{S}$ and by~\eqref{eq:hinr bed klass appr beobb, 1} the energy $t \mapsto 1/2 \norm{x(t,x_0,0)}_X^2$ is constant on the whole of $[0,\infty)$, it follows from~\eqref{eq:hinr bed klass appr beobb, 3} that $x_0 = 0$ as desired. 
\end{proof}

\begin{lm} \label{lm:hinr bed fuer aequilibrpkt-vor aus wISS-satz}
Suppose  that the assumptions (Conditions~\ref{cond:S imp-passive}, \ref{cond:P und R}, \ref{cond:B-tilde surj}) of the solvability  theorems are satisfied and, in addition, that $0$ is the only critical point of the potential energy $\mathcal{P}$ of $\mathfrak{S}_c$. If then Condition~\ref{cond:ISS, system} is satisfied, then the only equilibrium point $\tilde{x}_*$ of~\eqref{eq:closed-loop mit zero-input} is $0$. 
\end{lm}

\begin{proof}
Suppose that Condition~\ref{cond:ISS, system} is satisfied and that $\tilde{x}_* = (x_*,v_{*1},v_{*2})$ is an equilibrium point of~\eqref{eq:closed-loop mit zero-input}. 
Then $\tilde{x}_* \in D(\tilde{\mathcal{A}})$ with
\begin{align} \label{eq:hinr bed aequilibrpkt-vor, 1}
0 = \tilde{\mathcal{A}}\tilde{x}_* + \tilde{f}(\tilde{x}_*)
= 
\begin{pmatrix}
\mathcal{A}x_* \\ Kv_{*2} \\ B_c \mathcal{C}x_* - \nabla \mathcal{P}(v_{*1}) - \mathcal{R}(Kv_{*2})
\end{pmatrix}
= 
\begin{pmatrix}
\mathcal{A}x_* \\ Kv_{*2} \\ B_c \mathcal{C}x_* - \nabla \mathcal{P}(v_{*1})
\end{pmatrix}
\end{align} 
and
\begin{align}  \label{eq:hinr bed aequilibrpkt-vor, 2}
0 = \tilde{\mathcal{B}} \tilde{x}_* = \mathcal{B}x_* + B_c^{\top}Kv_{*2} + S_c \mathcal{C}x_*
= \mathcal{B}x_* + S_c \mathcal{C}x_*.
\end{align}
It follows from~\eqref{eq:hinr bed aequilibrpkt-vor, 1} %first row
and~\eqref{eq:hinr bed aequilibrpkt-vor, 2} using the impedance-passivity of $\mathfrak{S}$ (Condition~\ref{cond:S imp-passive}) that
\begin{align}  \label{eq:hinr bed aequilibrpkt-vor, 3}
0 = \scprd{x_*, \mathcal{A}x_*}_X \le (\mathcal{B}x_*)^{\top} \mathcal{C}x_* = - (\mathcal{C}x_*)^{\top} S_c \mathcal{C}x_*
\le 0.
\end{align}
So, by~\eqref{eq:S_c pos def} and~\eqref{eq:hinr bed aequilibrpkt-vor, 2} we see that
\begin{align} \label{eq:hinr bed aequilibrpkt-vor, 4}
\mathcal{C}x_* = 0 \qquad \text{and} \qquad \mathcal{B}x_* = 0.
\end{align}
In view of~\eqref{eq:hinr bed aequilibrpkt-vor, 1} %third row
and our extra assumption on the critical points of $\mathcal{P}$, this implies that $v_{*1} = 0$ and it remains to show that $x_* = 0$ as well. 
Since by~\eqref{eq:hinr bed aequilibrpkt-vor, 1} %first row
the constant function $t \mapsto x(t) := x_*$ is a classical solution of $x'=\mathcal{A}x$ and since Condition~\ref{cond:ISS, system} is satisfied, there exist positive constants $C_0, t_0$ such that
\begin{align}
\norm{x_*}_X^2 = \norm{x(t)}_X^2 \le C_0 \int_0^t |(\mathcal{H}x(s))(\eta)|^2 \d s
\le 
C_0/\kappa \int_0^t |\mathcal{B}x_*|^2 + |\mathcal{C}x_*|^2 \d s
\end{align}
for all $t \in [t_0,\infty)$ (Lemma~9.1.2 of~\cite{JacobZwart} or, more precisely, the third step in the proof of Theorem~3.5 of~\cite{Sc18-BV}). 
In view of~\eqref{eq:hinr bed aequilibrpkt-vor, 4}, this implies that $x_* = 0$ as desired. 
\end{proof}

An inspection of the proof of Theorem~\ref{thm:wISS} reveals that Condition~\ref{cond:wISS, system} and Condition~\ref{cond:wISS, controller}~(ii) are needed only to obtain the approximate observability of $\mathfrak{S}(\tilde{A}, \tilde{B}, \tilde{B}^*)$ (Lemma~\ref{lm:S(A,B,B*) appr beobb}) -- for all the other  %subsequent 
lemmas and arguments, the following Condition~\ref{cond:S(A,B,B*) appr beobb/steuerb} and Condition~\ref{cond:wISS, controller}~(i) (along with Conditions~\ref{cond:S imp-passive}, \ref{cond:P und R}, \ref{cond:B-tilde surj}, of course) are sufficient.

\begin{cond} \label{cond:S(A,B,B*) appr beobb/steuerb}
$\mathfrak{S}(\tilde{A}, \tilde{B}, \tilde{B}^*)$ is approximately observable in infinite time or approximately controllable in infinite time.
\end{cond} 

Consequently, the assertions of Theorem~\ref{thm:wISS} remain true if we replace Condition~\ref{cond:wISS, system} by the more general  %Lemma~\ref{lm:S(A,B,B*) appr beobb}!
Condition~\ref{cond:S(A,B,B*) appr beobb/steuerb} and omit Condition~\ref{cond:wISS, controller}~(ii) (while leaving all other assumptions of the theorem unchanged). Yet, the thus modified assumptions are %theorem is 
not much more general. In fact, by the following lemma, the modified assumptions entail all the original assumptions except, possibly, the not very restrictive injectivity assumption on $B_c$ (Condition~\ref{cond:wISS, controller}~(ii)).  
%Yet, the thus modified assumptions are %theorem is 
%not much more general, as the following lemma indicates.  
%(In fact, if we add to the thus modified assumptions the fairly innocent Condition~\ref{cond:wISS, controller}~(ii), then they become equivalent to the original assumptions of Theorem~\ref{thm:wISS}.) 

\begin{lm} \label{lm:alte S(A,B,B*)-vor nicht viel allgemeiner als vor direkt an S}
Suppose  that the assumptions (Conditions~\ref{cond:S imp-passive}, \ref{cond:P und R}, \ref{cond:B-tilde surj}) of the solvability  theorems are satisfied %Condition~\ref{cond:S(A,B,B*) appr beobb/steuerb} then implies Condition~\ref{cond:wISS, system}. %(merely/solely Condition~\ref{cond:wISS, controller}~(ii) might not be satisfied)
and that $\mathfrak{S}(\tilde{A}, \tilde{B}, \tilde{B}^*)$ is approximately observable in infinite time. %Condition~\ref{cond:S(A,B,B*) appr beobb/steuerb}~(i)
Then Condition~\ref{cond:wISS, system} is satisfied as well.
\end{lm}

\begin{proof}
Suppose $x_0 \in D(A)$ is such that the output $y(t,x_0,0) = \mathcal{C}x(t,x_0,0) = 0$ for all $t \in [0,\infty)$. We have to show that $x_0 =0$.
Set $\tilde{x}_0 := (x_0,0,0) \in \tilde{X}$ and $\tilde{x}(t) := (x(t,x_0,0),0,0) = (\e^{At}x_0,0,0) \in \tilde{X}$, then $\tilde{x}$ is continuously differentiable and, by the zero-output assumption,
\begin{align*}
\tilde{x}(t) \in D(\tilde{\mathcal{A}}) \cap \ker \tilde{\mathcal{B}} = D(\tilde{A}) 
\qquad \text{and} \qquad
\tilde{x}'(t) = \tilde{A}\tilde{x}(t)
\end{align*} 
for all $t \in [0,\infty)$. Consequently, $\tilde{x}(t) = \e^{\tilde{A}t}\tilde{x}_0$ and thus
\begin{align*}
\tilde{B}^* \e^{\tilde{A}t}\tilde{x}_0 = \tilde{B}^* \tilde{x}(t) = 0 
\qquad (t \in [0,\infty)).
\end{align*} 
Since $\mathfrak{S}(\tilde{A}, \tilde{B}, \tilde{B}^*)$ is approximately observable in infinite time, we obtain $\tilde{x}_0 =0$ as desired.
\end{proof}

\section{Some applications} \label{sect:applications}

%\subsection{Applications to systems of order $N=1$}

In this section, we apply the uniform and the weak input-to-state stability results to a vibrating string and a Timoshenko beam.

\begin{ex}
Consider a vibrating string~\cite{Villegas}, \cite{JacobZwart}, \cite{Augner}, that is, the transverse displacement $w(t,\zeta)$ of the string at %the horizontal 
position $\zeta \in (a,b)$ evolves according to the partial differential equation
\begin{align} \label{eq:string pde}
\rho(\zeta) \partial_t^2 w(t,\zeta) =  \partial_{\zeta} \big( T(\zeta) \partial_{\zeta}w(t,\zeta) \big)
 \qquad (t \in [0,\infty), \zeta \in (a,b))
\end{align}
(vibrating string equation) and the energy $E_w(t)$ of the string at time $t$ is given by 
\begin{align} \label{eq:string energy}
E_w(t) = \frac{1}{2} \int_a^b \rho(\zeta) \big( \partial_t w(t,\zeta) \big)^2 + T(\zeta) \big( \partial_{\zeta} w(t,\zeta) \big)^2 \d \zeta.
\end{align}
In these equations, $\rho$, $T$ are the mass density and the Young modulus of the string and they are assumed to be absolutely continuous and to be bounded below and above by positive finite constants. Also, assume that the string is clamped at its left end, that is, 
\begin{align} \label{eq:string bdry cond}
\partial_t w(t,a) = 0 \qquad (t \in [0,\infty))
\end{align}  
and that the control input $u(t)$ and observation output $y(t)$ are given respectively by the force and by the velocity at the right end of the string, that is,
\begin{align} \label{eq:string input/output}
u(t) = T(b) \partial_{\zeta} w(t,b)
\qquad \text{and} \qquad
y(t) = \partial_t w(t,b)
\end{align}
for all $t \in [0,\infty)$. With the choices
\begin{align*}
x(t)(\zeta) 
:=
\begin{pmatrix}
\rho(\zeta) \partial_t w(t,\zeta) \\ \partial_{\zeta} w(t,\zeta)
\end{pmatrix},
\qquad %\text{and} \qquad
\mathcal{H}(\zeta) 
:=
\begin{pmatrix}
1/\rho(\zeta) & 0 \\ 0 & T(\zeta)
\end{pmatrix},
\qquad
P_1 := \begin{pmatrix} 0 & 1 \\ 1 & 0 \end{pmatrix}
\end{align*}
and $P_0 := 0 \in \R^{2\times 2}$, the pde~\eqref{eq:string pde} takes the form~\eqref{eq:S diff eq} of a port-Hamiltonian system of order $N =1$ and, moreover, the boundary condition~\eqref{eq:string bdry cond} and the in- and output conditions~\eqref{eq:string input/output} take the desired form~\eqref{eq:S bdry cond} and~\eqref{eq:S input/output eq}, \eqref{eq:S bdry control/observ operators} with matrices $W_{B,1}, W_{B,2}, W_C \in \R^{1\times 4}$. 
It is straightforward to verify that this system $\mathfrak{S}$ is impedance-energy-preserving, that the matrix $W \in \R^{3\times 4}$ from~\eqref{eq:W-matrix, def} has full rank, and that $\mathcal{H}$ is absolutely continuous and~\eqref{eq:ISS, matrungl} holds true. %with $\eta := b$
In particular, Conditions~\ref{cond:S imp-passive}, \ref{cond:B-tilde surj}, \ref{cond:ISS, system} are satisfied. 
So, as soon as the controller $\mathfrak{S}_c$ is chosen such that
\begin{itemize}
\item Conditions~\ref{cond:P und R} and~\ref{cond:ISS, controller} are satisfied, or
\item Conditions~\ref{cond:P und R} and~\ref{cond:wISS, controller} are satisfied and $0$ is the only critical point of $\mathcal{P}$,
\end{itemize}
respectively, the resulting closed-loop system $\tilde{\mathfrak{S}}$ is input-to-state stable or weakly input-to-state stable, respectively (Lemma~\ref{lm:hinr bed fuer klass appr beobb-vor an S aus wISS-satz} and~\ref{lm:hinr bed fuer aequilibrpkt-vor aus wISS-satz}!).~$\blacktriangleleft$
\end{ex}

\begin{ex}
Consider a beam modelled according to Timoshenko~\cite{Villegas}, \cite{JacobZwart}, \cite{Augner}, that is, the transverse displacement $w(t,\zeta)$ and the rotation angle $\phi(t,\zeta)$ of the beam at %the horizontal 
position $\zeta \in (a,b)$ evolve according to the partial differential equations
\begin{gather} \label{eq:Timoshenko pde}
\rho(\zeta) \partial_t^2 w(t,\zeta) = \partial_{\zeta} \Big( K(\zeta) \big( \partial_{\zeta}w(t,\zeta) - \phi(t,\zeta) \big) \Big) \\
I_r(\zeta) \partial_t^2 \phi(t,\zeta) = \partial_{\zeta} \big( E I(\zeta) \partial_{\zeta} \phi(t,\zeta) \big) + K(\zeta) \big( \partial_{\zeta}w(t,\zeta) - \phi(t,\zeta) \big)
% \qquad (t \in [0,\infty), \zeta \in (a,b))
\end{gather}
for $t \in [0,\infty), \zeta \in (a,b)$ (Timoshenko beam equations) and the energy $E_{w,\phi}(t)$ of the beam at time $t$ is given by 
\begin{align} \label{eq:Timoshenko energy}
E_{w,\phi}(t) &= \frac{1}{2} \int_a^b \rho(\zeta) \big( \partial_t w(t,\zeta) \big)^2 + K(\zeta)  \big( \partial_{\zeta}w(t,\zeta) - \phi(t,\zeta) \big)^2 \notag \\
&\qquad \qquad + I_r(\zeta) \big( \partial_t \phi(t,\zeta) \big)^2 + E I(\zeta) \big( \partial_{\zeta} \phi(t,\zeta)\big)^2  \d \zeta.
\end{align}
In these equations, $\rho$, $E$, $I$, $I_r$, $K$ are respectively the mass density, the Young modulus, the moment of inertia, the rotatory moment of inertia, and the shear modulus of the beam and they are assumed to be absolutely continuous and to be bounded below and above by positive finite constants. Also, assume that the beam is clamped at its left end, that is, 
\begin{align} \label{eq:Timoshenko bdry cond}
\partial_t w(t,a) = 0 \qquad \text{and} \qquad \partial_t \phi(t,a) = 0
\qquad (t \in [0,\infty))
\end{align}  
(velocity and angular velocity at the left endpoint $a$ are zero), and that the control input $u(t)$ is given by the force and the torsional moment at  the right end of the beam and the observation output $y(t)$ is given by the velocity and angular velocity at the right end of the beam, that is,
\begin{align} \label{eq:Timoshenko input/output}
u(t) = \begin{pmatrix} K(b) \big( \partial_{\zeta} w(t,b) - \phi(t,b) \big) \\ E I(b) \partial_{\zeta} \phi(t,b) \end{pmatrix},
\qquad
y(t) = \begin{pmatrix} \partial_t w(t,b) \\ \partial_t \phi(t,b) \end{pmatrix}
\end{align}
for all $t \in [0,\infty)$. With the choices
\begin{align*}
x(t)(\zeta) 
:=
\begin{pmatrix}
\partial_{\zeta} w(t,\zeta) - \phi(t,\zeta) \\
\rho(\zeta) \partial_t w(t,\zeta) \\
\partial_{\zeta} \phi(t,\zeta) \\
I_r(\zeta) \partial_t \phi(t,\zeta)
\end{pmatrix},
\qquad %\text{and} \qquad
\mathcal{H}(\zeta) 
:=
\begin{pmatrix}
K(\zeta) & 0 & 0 & 0 \\
0 & 1/\rho(\zeta) & 0 & 0 \\ 
0 & 0 & EI(\zeta) & 0 \\
0 & 0 & 0 & 1/I_r(\zeta)
\end{pmatrix},
\end{align*}
and an appropriate choice of $P_1, P_0 \in \R^{4\times 4}$, the pde~\eqref{eq:Timoshenko pde} take the form~\eqref{eq:S diff eq} of a port-Hamiltonian system of order $N =1$ and, moreover, the boundary condition~\eqref{eq:Timoshenko bdry cond} and the in- and output conditions~\eqref{eq:Timoshenko input/output} take the desired form~\eqref{eq:S bdry cond} and~\eqref{eq:S input/output eq}, \eqref{eq:S bdry control/observ operators} with matrices $W_{B,1}, W_{B,2}, W_C \in \R^{2\times 8}$. 
It is straightforward to verify that this system $\mathfrak{S}$ is impedance-energy-preserving, that the matrix $W \in \R^{6\times 8}$ from~\eqref{eq:W-matrix, def} has full rank, and that $\mathcal{H}$ is absolutely continuous and~\eqref{eq:ISS, matrungl} holds true. %with $\eta := b$
In particular, Conditions~\ref{cond:S imp-passive}, \ref{cond:B-tilde surj}, \ref{cond:ISS, system} are satisfied. 
So, as soon as the controller $\mathfrak{S}_c$ is chosen such that
\begin{itemize}
\item Conditions~\ref{cond:P und R} and~\ref{cond:ISS, controller} are satisfied, or
\item Conditions~\ref{cond:P und R} and~\ref{cond:wISS, controller} are satisfied and $0$ is the only critical point of $\mathcal{P}$,
\end{itemize}
respectively, the resulting closed-loop system $\tilde{\mathfrak{S}}$ is input-to-state stable or weakly input-to-state stable, respectively (Lemma~\ref{lm:hinr bed fuer klass appr beobb-vor an S aus wISS-satz} and~\ref{lm:hinr bed fuer aequilibrpkt-vor aus wISS-satz}!).~$\blacktriangleleft$
\end{ex}

\section*{Acknowledgements}

J. Schmid gratefully acknowledges financial support by the German Research Foundation (DFG) through the research training group ``Spectral theory and dynamics of quantum systems'' (GRK 1838)  and through the grant ``Input-to-state stability and stabilization of distributed-parameter systems'' (DA 767/7-1)

\begin{small}

\end{small}

\end{document}